%% file: MD_MFC.tex
\documentclass{article}

    \usepackage[preprint]{neurips_2023}

\usepackage[utf8]{inputenc} 
\usepackage[T1]{fontenc}    
\usepackage{hyperref}       
\usepackage{url}            
\usepackage{booktabs}       
\usepackage{amsfonts}       
\usepackage{nicefrac}       
\usepackage{microtype}      
\usepackage{xcolor}         

\usepackage{graphicx}
\usepackage{subcaption}
\usepackage{algorithm} 
\usepackage{algorithmic}
\usepackage{hyperref}       
\usepackage{natbib}
\usepackage{makecell}
\usepackage{float}

\usepackage{amsmath}
\usepackage{amssymb}
\usepackage{mathtools}
\usepackage{amsthm}

\usepackage[capitalize,noabbrev]{cleveref}

\theoremstyle{plain}
\newtheorem{theorem}{Theorem}[section]
\newtheorem{proposition}[theorem]{Proposition}
\newtheorem{lemma}[theorem]{Lemma}

\theoremstyle{definition}
\newtheorem{definition}[theorem]{Definition}

\theoremstyle{remark}

\DeclareMathOperator*{\argmax}{arg\,max}
\DeclareMathOperator*{\argmin}{arg\,min}
\usepackage{dsfont}

\usepackage{pgfplots}
\pgfplotsset{compat=1.17}

\title{Reimagining Demand-Side Management with Mean Field Learning}

\newcommand{\footremember}[2]{%
    \footnote{#2}
    \newcounter{#1}
    \setcounter{#1}{\value{footnote}}%
}
\newcommand{\footrecall}[1]{%
    \footnotemark[\value{#1}]%
} 

\author{
 Bianca Marin Moreno \\
  Inria\footremember{inria}{Univ. Grenoble Alpes, Inria, CNRS, Grenoble INP,
LJK, 38000 Grenoble, France.}, EDF R$\&$D\footremember{edf}{EDF Lab, 7 bd Gaspard Monge, 91120 Palaiseau, France} \\
   \And
 Margaux Brégère \\
 Sorbonne Université\footnote{Sorbonne Université LPSM, Paris, France}, EDF R$\&$D\footrecall{edf} \\
  \And
 Pierre Gaillard \\
  Inria\footrecall{inria} \\
  \And
 Nadia Oudjane\\
  EDF R$\&$D\footrecall{edf}\\
}

\begin{document}

\maketitle

\begin{abstract}

Integrating renewable energy into the power grid while balancing supply and demand is a complex issue, given its intermittent nature. Demand side management (DSM) offers solutions to this challenge. We propose a new method for DSM, in particular the problem of controlling a large population of electrical devices to follow a desired consumption signal. We model it as a finite horizon Markovian mean field control problem. We develop a new algorithm, MD-MFC, which provides theoretical guarantees for convex and Lipschitz objective functions. What distinguishes MD-MFC from the existing load control literature is its effectiveness in directly solving the target tracking problem without resorting to regularization techniques on the main problem. A non-standard Bregman divergence on a mirror descent scheme allows dynamic programming to be used to obtain simple closed-form solutions. In addition, we show that general mean-field game algorithms can be applied to this problem, which expands the possibilities for addressing load control problems. We illustrate our claims with experiments on a realistic data set.

\end{abstract}

\section{Introduction}
Climate change is a complex problem, for which action takes many forms. Of the top 20 solutions identified by \citet{drawdown} to reverse global warming, six are related to the energy sector, including integrating renewables into the electricity system and increasing the number of electric vehicles as a primary mode of transportation. In addition, the study managed by \citet{RTE} showed that achieving carbon neutrality by 2050 in the French electricity scenario requires a decrease in final energy consumption and strong growth in renewable energy. 

However,  it is extremely difficult to make these solutions economically viable while also scaling them up. The intermittent nature of renewable energy sources can cause significant fluctuations in energy demand and supply, which may impact the balance of the power grid. Current solutions to keep the system in balance rely heavily on fossil fuel power plants, which have significant environmental costs, or on energy imports, which have capital and operating costs. Demand Side Management (DSM) are strategies to reduce energy acquisition costs and associated penalties by continuously monitoring energy consumption and managing devices \citep{bakare_comprehensive_2023}, which provides flexibility and improves the reliability of energy systems. Yet, implementing DSM solutions is challenging, as it involves large-scale data processing and near real-time scenarios. For this reason, machine learning solutions have recently emerged to solve DSM problems \citep{AI_for_dsm} with examples ranging from using multi-armed bandits to develop pricing solutions \citep{bregere}, to deep learning models for smart charging of electric vehicles \citep{deep_l_vehicles}.

The goal of this paper is to make a new contribution to this field by proposing a new solution to a DSM problem concerning the control of thermostatically controled loads (TCLs: flexible appliances such as water heaters, air conditioners, refrigerators, etc). The aim is to control the aggregate power consumption of a large population of water heaters in order to follow a target consumption profile. To this end, we consider a finite time horizon Markovian mean field control (MFC) problem, and we propose a new algorithm based on mirror descent. We also adapt other mean field learning algorithms from the literature for this purpose. 

\paragraph{Contributions} In this paper, we propose and compare two new approaches to solve a DSM problem: a new algorithm, MD-MFC, for general Markovian MFC problems, and an adaptation of existing algorithms in the mean field learning literature for game problems. First we provide a modeling of the management system in question as a Markov decision process (MDP) in Section~\ref{model}. The literature review of previous modeling and solutions to the load control problem, as well as a discussion of the main ingredient of our algorithms, mean field learning, are postponed to Subection~\ref{litterature}. Our main results are stated in Section~\ref{algorithm}: we introduce the MD-MFC algorithm for a Markovian MFC problem, and we prove a convergence rate of order $\smash{1/\sqrt{K}}$, where $K$ is the number of iterations, by linking it to a mirror descent \citep{md_original} scheme. This implies a non-trivial reformulation of a non-convex problem in a measure space into a convex problem. A good choice of non-standard regularization allows dynamic programming \citep{DP_principles}. This results in the first algorithm that efficiently and directly solves the target tracking problem without resorting to regularization techniques in the main problem. Section~\ref{experiments} illustrates the results with simulations based on a realistic data set \citep{smach}. A series of future works concludes the paper.

\section{Setting and model}\label{model}
Our framework consists in modeling the random dynamics of a population of water heaters in order to control their average consumption to follow a target signal. From now on, for a finite set $S$, we define $\Delta_S$ to be the simplex of dimension $|S|$, the cardinal of $S$.

\subsection{Randomized controlled dynamics for one water heater}\label{subsec:random}

Let us consider a discretisation of the time for $n = 1,...,N$. At each time step $n$, the state of a water heater is described by a variable $x_n = (m_n, \theta_n) \in \mathcal{X} := \{0,1\} \times \Theta$, where $m_n$ indicates the operating state of the heater (ON if $1$, OFF if $0$), and $\theta_n$ represents the average temperature of the water in the tank. For the sake of simplification we consider only temperatures inside a finite set $\Theta$.

We call the uncontrolled dynamics the nominal dynamics. A water heater that follows the nominal dynamics \citep{ana_busic} obeys a cyclic ON/OFF decision rule with a deadband to ensure that the temperature is between a lower limit $T_\text{min}$ and an upper limit $T_\text{max}$. Thus, if the water heater is turned on, it heats water with the maximum power until its temperature exceeds $T_\text{max}$. Then, the heater turns off and the water temperature decreases until it reaches $T_\text{min}$, where the heater turns on again and a new cycle begins. The nominal dynamics at a discretized time is illustrated in Figure~\ref{fig:nominal_temp_example}. The temperature at each time step is calculated by approximating an ordinary differential equation (ODE) depending on the current operating state of the heater and the hot water drawn at each time step (see Appendix~\ref{nominal_behavior}). We assume that the event of a water withdrawn is random and independent at each time step with a known probability distribution.  

\begin{figure}[ht]
\centerline{\input{graphs/temp_nominal}}
\caption{Temperature evolution of a water heater following the nominal dynamics.}
\label{fig:nominal_temp_example}
\end{figure}
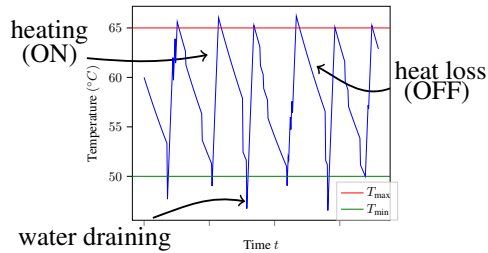

In order to have a controllable model, we fit the nominal dynamics of a water heater to a Markov decision process. The finite state space is given by $\mathcal{X}$, and we consider an action space given by $\mathcal{A} := \{0,1\}$. At time step $n$, choosing action $1$ means turning the heater on except when $\theta_n \geq T_\text{max}$. Conversely, choosing action $0$ means turning the heater off except when $\theta_n \leq T_\text{min}$.  The nominal dynamics deterministically chooses action $0$ if the heater is off and $1$ if it is on, independent of the heater's temperature. Unlike the nominal dynamics, we want to consider stochastic strategies for choosing actions. If the heater is in state $x_n = (m_n, \theta_n)$, the next temperature is computed by $\theta_{n+1} := T(m_n, \theta_n, \epsilon_n)$ where $T$ is a function defined as the solution of Equation~\eqref{temperature_ode} of Appendix~\ref{nominal_behavior} and $\epsilon_n$ is the random variable corresponding to a water-withdraw event at time step $n$. The action $a_n$ is sampled with probability $\pi_n(\cdot|x_n) \in \Delta_\mathcal{A}$, and the next operating state is given by $m_{n+1} := M(a_n, \theta_{n+1})$, where
\begin{equation*}
    M(a_n, \theta_{n+1}) :=  a_n \mathds{1}_{\theta_{n+1} \in [T_\text{min}, T_\text{max}]} + \mathds{1}_{\theta_{n+1} < T_\text{min}}.
\end{equation*}

Hence, the probability kernel for each time step $n$ is given by 
\begin{equation*}
    {p_{n+1}(x_{n+1}|x_n,a_n) :=\mathbb{P}\big(\theta_{n+1} = T(m_n,\theta_n,\epsilon_n) | \theta_n, m_n \big) \mathbb{P}\big(m_{n+1} = M(a_n,\theta_{n+1}) | a_n, \theta_{n+1}\big)}.
\end{equation*}
Moreover if $\theta_{n+1} \in [T_\text{min}, T_\text{max}]$, the action $a_n \sim \pi_n(\cdot|x_n)$ defines the next operating state of the heater. For more details on our modeling of the dynamics of a water heater, see Appendix~\ref{nominal_behavior}.

\subsection{Optimisation problem}
Consider a population of $M$ water heaters indexed by $i$ and described at time step $n$ by $X_n^i = (m_n^i, \theta_n^i)$ following the randomized dynamics described in Subsection~\ref{subsec:random}. We suppose all water heaters to be homogeneous, i.e. they have the same dynamics, and follow the same policy $\pi$. Let $\smash{\bar{m}_n := \frac{1}{M} \sum_{i=1}^M m_n^i}$ denote the average consumption. We assume for simplicity that the maximum power of each water heater is $p_\text{max} = 1$ so that the average consumption is equal to the proportion of heaters at state ON. Note that $\bar{m}_n$ depends on the policy $\pi$ that the water heaters follow, thus we can denote it as $\bar{m}_n(\pi)$. Let ${\gamma = (\gamma_n)_{1 \leq n \leq N} \in [0,1]^N}$ be our target consumption profile (for example, the energy production at each time step divided by the number of devices). Our goal is to solve the problem
\begin{equation}\label{opt_finite_heaters}
    \min_{\pi \in (\Delta_\mathcal{A})^{\mathcal{X} \times N}} \mathbb{E} \left[ \sum_{n=1}^N (\bar{m}_n(\pi) - \gamma_n)^2 \right],
\end{equation}
where we have chosen to work with a quadratic loss.

Let $\mu := (\mu_n)_{n \in [0,...,N]}$ such that $\mu_n$ is the state-action distribution of the entire population of heaters at time $n$. We denote by $\mu^\pi$ a state-action distribution sequence induced by a policy sequence $\pi$ such as in Definition~\ref{mu_induced_by_pi}.

\begin{definition}[Distribution induced by a policy $\pi$]\label{mu_induced_by_pi}
Given an initial distribution $\mu_0$ fixed, the state-action distributions sequence induced by the policy sequence $\pi = (\pi_n)_{1 \leq n \leq N}$ is denoted $\mu^\pi := (\mu^\pi_n)_{1 \leq n \leq N}$ and is defined recursively by  
\begin{align*}
    \begin{split}
        & \mu_0^\pi(x',a') := \mu_0(x',a') \\
        & \mu_{n+1}^{\pi} (x',a') := \sum_{x\in\mathcal{X}} \sum_{a \in\mathcal{A}} \mu_n^\pi(x,a) p_{n+1}(x'| x, a) \pi_{n+1}(a' | x').
    \end{split}
\end{align*}
\end{definition}

For a function $\varphi: \mathcal{X} \rightarrow \mathbb{R}$, we define  ${\mu_n(\varphi) := \sum_x \varphi(x) \mu_n(x, a)}$ for all $1 \leq n \leq N$. We are particularly interested in a function $\varphi$ such that $\mu_n(\varphi)$ gives us the average consumption of our water heater's population. Thus, we consider from now on
\begin{equation}\label{varphi_func}
\begin{split}
    \varphi: \quad \mathcal{X} &\rightarrow \mathbb{R} \\
    (m,\theta) &\mapsto m.
\end{split}
\end{equation}
For such a function $\varphi$, when $M \to \infty$, the mean field approximation \citep{chaos_mfg} of Problem~\eqref{opt_finite_heaters} consists in the main mean field control problem considered in the paper, and is given by 
\begin{equation}\label{main_optimisation_problem}
    \min_\pi F(\mu^\pi) := \sum_{n=1}^N f_n(\mu^\pi_n),
\end{equation}
where ${f_n(\mu_n^\pi) := (\mu_n^\pi(\varphi) - \gamma_n)^2}$. 

It is important to mention that the algorithms and results presented in Section~\ref{algorithm} for solving Problem~\eqref{main_optimisation_problem} remain valid for any general finite horizon Markovian MFC problem with finite state and action spaces, where the cost functions $f_n$ are convex and Lipschitz with respect to the $\|\cdot\|_1$ norm. 

\subsection{Literature discussion}\label{litterature}

Decision-making problems formulated as such mean-field models are a popular framework for stochastic optimization problems in many applications, ranging from robotics \citet{Shiri_auto, robotics_survey} to finance \citet{Achdou_eco, finance01}, energy management \citet{energy01, kl_ana_busic}, epidemic modeling \citet{health01}, and more recently, machine learning \citet{mean_field_ml01, mean_field_ml02, mfg_ml_03, mgf_ml_04}. Thus, although this paper focuses on a demand management problem, our results also provide a new approach to solving problems in many other areas. 

\paragraph{Load control}
Controlling the sum of the consumption of a large number of TCLs started being investigated around $1980$ by \citet{tcl_01, malhame_85, tcl_02} establishing the first physically based modeling for a TCL population. In the works of \citet{malhame_kiz_2013, malhame_kiz_2014}, the difficulty due to the large number of devices is circumvented by a mean field approximation.

For water heater control, \citet{ana_busic_KL_0} use a quadratic objective and a Kullback-Leibler (KL) penalty allowing a Lagrangian approach that learns both the control and the probability transition kernel, but cannot handle uncontrolled state parts, so uncertainties like water withdrawals must be modeled as deterministic. More recently, \citet{ana_busic_kl} takes into account the uncontrolled stochastic environment in the KL quadratic control framework by adding constraints on the probability transition kernel. It is the KL penalty on the main problem that allows them to obtain their main results. However, there is a trade-off between adding the KL and obtaining a good target tracking curve. We therefore propose to solve directly the same quadratic control framework but without the KL penalty. We have successfully provided the first algorithm for direclty solving the target tracking problem. 

\paragraph{Mean field learning}

 Mean field games (MFG) have been introduced by \citet{MFG_original} and \citet{huang_malhame_mfg} to tackle the issue of games with a large number of symmetric and anonymous players, by passing to the limit of an infinite number of players interacting through the population distribution. Although MFG focuses on finding Nash equilibria (NE), social optima on cooperative setting have also been studied under the term of mean field control (MFC) \citep{mfc}. 

Lately, iterative learning methods such as fictitious play and online mirror descent have been adapted to the MFG scenario in \citet{FP_MFG_finitespace} and \citet{OMD_perolat}. \citet{concave_utility} show an equivalence between Frank Wolfe's classical optimization algorithm \citep{frank_wolfe} and the fictitious play for potential structured games. Similarly, we show an equivalence between our MFC problem and potential games, that open up a new range of solutions to the DSM problem considered using the above studies.   

\section{Main results: algorithmic approaches}\label{algorithm}

\subsection{Building a new algorithm}
Consider the set of state-action distributions sequences initialized at $\mu_0 \in \Delta_{\mathcal{X} \times \mathcal{A}}$ and satisfying a specific constrained evolution given by
\begin{equation}\label{set_M}
    \begin{split}
        \mathcal{M}_{\mu_0} :=  \bigg\{ \mu \in 
        (\Delta_{\mathcal{X} \times \mathcal{A}})^N  \big| \; &\sum_{a' \in \mathcal{A}} \mu_{n+1}(x',a')  \\
        &= \sum_{x \in \mathcal{X} , a \in \mathcal{A}} p_{n+1}(x'|x,a) \mu_{n}(x,a)\;, \forall x' \in\mathcal{X}, 
         \forall n \in [0,...,N]\bigg\}.
    \end{split} 
\end{equation}
The set $\smash{\mathcal{M}_{\mu_0}}$ describes the sequences of state-action distribution respecting the dynamics of the Markov model. Furthermore, this set is convex \citep{ana_busic_kl}.

\begin{proposition}\label{opt_mu_equal_pi}
Let $\mu_0 \in \Delta_{\mathcal{X} \times \mathcal{A}}$. The application ${\pi \mapsto \mu^{\pi}}$ is a surjection from $(\Delta_\mathcal{A})^{\mathcal{X} \times N}$ to $\mathcal{M}_{\mu_0}$.
\end{proposition}
The idea of the proof of Proposition~\ref{opt_mu_equal_pi}, reported to Appendix~\ref{missing_proofs}, is that one can retrieve the policy sequence $\pi$ inducing the state-action distribution sequence $\mu$ by taking $\smash{\pi_n(a|x) = \frac{\mu_n(x,a)}{\rho_n(x)}}$, where $\smash{\rho_n(x) := \sum_{a \in \mathcal{A}} \mu_n(x,a)}$. Let $\smash{\mathcal{M}_{\mu_0}^*}$ denotes the subset of $\smash{\mathcal{M}_{\mu_0}}$ where the corresponding policies $\pi$ are such that $\pi_n(a|x) \neq 0$ for all $(x,a) \in \mathcal{X} \times \mathcal{A}$ and $1 \leq n \leq N$. We define the regularization function $\smash{\Gamma : \mathcal{M}_{\mu_0} \times \mathcal{M}_{\mu_0}^* \to \mathbb{R}}$ as
\begin{equation}\label{gamma}
       \Gamma(\mu^\pi, \mu^{\pi'}) := \sum_{n=1}^{N} \mathbb{E}_{(x,a) \sim \mu{^\pi}_n(\cdot)}\bigg[\log\bigg(\frac{\pi_{n}(a|x)}{\pi'_{n}(a|x)}\bigg)\bigg].
\end{equation}
Before giving a solution to Problem~\eqref{main_optimisation_problem}, we consider the following auxiliary optimization problem, which will later help us build the new algorithm. This iterative scheme is possible thanks to Proposition~\ref{opt_mu_equal_pi} which guarantees the existence of a strategy $\pi$ for any $\smash{\mu \in \mathcal{M}_{\mu_0}}$. Here, $k$ represents an iteration:
\begin{equation}\label{MD_opt_problem}
\begin{split}
\mu^{k+1} \in \argmin_{\mu^\pi\in\mathcal{M}_{\mu_0}}
\bigg\{
\langle \nabla F(\mu^k),\mu^\pi\rangle 
+\frac{1}{\tau_k}  \Gamma(\mu^\pi, \mu^k)
\bigg\}.
\end{split}
\end{equation}
We consider $\tau_k > 0$ and $\smash{\langle \nabla F(\mu^k),\mu^\pi\rangle := \sum_{n=1}^N \langle \nabla f_n(\mu_n^k), \mu_n^\pi \rangle}$. At iteration $k+1$, we want to find $\mu^\pi$ by minimizing a linearization of $F$ around $\mu^k$, the distribution sequence found at the previous iteration, and at the same time penalizing the distance between $\pi$ generating $\mu^\pi$ and $\pi^k$ generating $\mu^k$. Choosing this non-standard regularization $\Gamma$ in Equation~\eqref{gamma} instead of the traditional KL divergence on marginal state-action distributions is what enables us to obtain a simple closed-form solution for the iterative scheme. Later we show that $\Gamma$ is a Bregman divergence. Thus, the use of $\Gamma$ brings a significant improvement to the solution of MFC problems because it allows to obtain low complexity solutions with theoretical bounds, as we will prove later.

Let for all $(x_n, a_n, \mu_n) \in \mathcal{X} \times\mathcal{A} \times \Delta_{\mathcal{X} \times\mathcal{A}}$, ${r_n(x_n, a_n, \mu_n) := - \nabla f_n(\mu_n)(x_n, a_n)}$. We show in Theorem~\ref{MD_explicit_result} that, due to the choice of penalizing strategies, the iterative scheme in Equation~\eqref{MD_opt_problem} can be solved through dynamic programming \citep{DP_principles} by building a Bellman recursion:
\begin{theorem}\label{MD_explicit_result}
Let $k \geq 0$. The solution of Problem~\eqref{MD_opt_problem} is $\mu^{k+1} = \mu^{\pi^{k+1}}$ (as in Definition~\ref{mu_induced_by_pi}), where for all ${1 \leq n \leq N}$, and $(x,a) \in \mathcal{X} \times \mathcal{A}$, 
\begin{equation}\label{policy_update}
    \pi_{n}^{k+1}(a|x) := \frac{\pi_{n}^{k}(a|x) \exp\left(\tau_k \tilde{Q}_{n}^{k}(x,a) \right)}{\sum_{a' \in \mathcal{A}}\pi_{n}^{k}(a'|x) \exp\left(\tau_k \tilde{Q}_{n}^{k}(x,a') \right)},
\end{equation}
where $\tilde{Q}$ is a regularized $Q$-function satisfying the following recursion
\begin{equation}\label{Bellman_Q_tilde}
    \begin{cases}
        \Tilde{Q}^k_N(x,a) = r_N(x,a, \mu_N^k) \\
        \!\begin{aligned}
            \Tilde{Q}^k_n(x,a) &= \max_{\pi_{n+1} \in (\Delta_\mathcal{A})^{\mathcal{X}}} \Bigg\{ r_n(x,a,\mu_n^k) + \sum_{x'} p_{n+1}(x'|x,a) \\
             &\sum_{a'} \pi_{n+ 1}(a'|x') \bigg[  - \frac{1}{\tau_k}\log\left( \frac{\pi_{n+1}(a'|x')}{\pi_{n+1}^k(a'|x')}\right)  
             + \Tilde{Q}^k_{n+1}(x',a') \bigg] \Bigg\}, \quad \forall 1 \leq n \leq N.
        \end{aligned}
    \end{cases}
\end{equation}
\end{theorem}
\begin{proof}
See Appendix~\ref{thm_1_proof}.
\end{proof}

Notice that the value $\smash{\pi_{n+1} \in (\Delta_\mathcal{A})^{\mathcal{X}}}$ maximizing the equation to find $\smash{\Tilde{Q}^k_{n}}$ in the Recursion~\eqref{Bellman_Q_tilde} is given by $\smash{\pi_{n+1}^{k+1}}$. We can then build the MD-MFC method in Algorithm~\ref{alg:MD}. Note that Algorithm~\ref{alg:MD} is well defined because the policy update in Equation~\eqref{policy_update} ensures that each iteration remains in $\mathcal{M}_{\mu_0}^*.$

\begin{algorithm}[ht]
\caption{MD-MFC}\label{alg:MD}
\begin{algorithmic}
\STATE {\bfseries Input:} number of iterations $K$, initial sequence of policies $\smash{\pi^0 \in (\Delta_\mathcal{A})^{\mathcal{X} \times N}}$ such that $\smash{\mu^0 := \mu^{\pi_0} \in \mathcal{M}_{\mu_0}^*}$, initial state-action  distribution $\mu_0$ (always fixed), sequence of non-negative learning rates $(\tau_k)_{k \leq K}$.
    \FOR{$k= 0,...,K-1$} 
    \STATE $\mu^k = \mu^{\pi^k}$ as in           
    Definition~\ref{mu_induced_by_pi}.
    \STATE $\tilde{Q}_N^{k}(x,a) = r_N(x,a,\mu_N^k)$ for all $(x,a) \in \mathcal{X} \times \mathcal{A}$.
    \FOR{$n = N,...,1$}
    \STATE $\forall (x,a) \in \mathcal{X} \times \mathcal{A}:$
    \STATE $\pi_n^{k+1}(a|x)=  \frac{\pi_{n}^{k}(a|x) \exp\left(\tau_k \tilde{Q}_{n}^{k}(x,a) \right)}{\sum_{a'}\pi_{n}^{k}(a'|x) \exp\left(\tau_k \tilde{Q}_{n}^{k}(x,a') \right)}$.
    \STATE $\tilde{Q}_{n-1}^{k}(x,a)$ using the recursion in Equation~\eqref{Bellman_Q_tilde}.
    \ENDFOR
    \ENDFOR
\STATE {\bfseries return} $\pi^K$
\end{algorithmic}
\end{algorithm}

\subsection{Convergence properties of the algorithm}\label{sec:algo_convergence}
We present a result on the convergence rate of Algorithm~\ref{alg:MD}. 

\begin{theorem}\label{thm:convergence_rate}
Let $\pi^*$ be a minimizer of Problem~\eqref{main_optimisation_problem}.
Applying $K$ iterations of Algorithm~\ref{alg:MD} to this problem, with, for each $1 \leq k \leq K$,
$$\tau_k := \frac{\sqrt{2 \Gamma(\mu^{\pi^*},\mu^0)} }{L} \frac{1}{\sqrt{K}},$$
gives the following convergence rate
$$ \min_{0 \leq s \leq K} F(\mu^{\pi^s}) - F(\mu^{\pi^*}) \leq L \frac{\sqrt{2 \Gamma(\mu^{\pi^*},\mu^0)} }{\sqrt{K}}.$$
\end{theorem}
\begin{proof} 
The proof consists in showing that Algorithm~\ref{alg:MD} is a mirror descent scheme applied to
\begin{equation}\label{convex_main_optimisation_problem}
    \min_{\mu \in \mathcal{M}_{\mu_0}} F(\mu),
\end{equation}
that is equivalent to Problem~\eqref{main_optimisation_problem} as a direct consequence of Proposition~\eqref{opt_mu_equal_pi}, and that the new Problem~\eqref{convex_main_optimisation_problem} does satisfy the necessary hypothesis for mirror descent convergence \citep{MD} with a non-standard Bregman divergence. The strength of this result is showing that the complex non-convex Problem~\eqref{main_optimisation_problem} can be solved using a classical optimization algorithm, which, with the right choice of regularizer, has an efficient solution thanks to dynamic programming. 

Let us start by showing that $\Gamma$ is indeed a Bregman divergence. For ease of notation, for any probability measure $\eta \in \Delta_E$, whatever the (finite) space $E$, we introduce the neg-entropy function, with the convention that $0 \log(0) = 0$,
\begin{equation*}
    \phi (\eta):=\sum_{x\in E} \eta (x)\log \eta (x).
\end{equation*}
\begin{proposition}\label{penalization_is_bregman}
    Let $\mu, \mu' \in \mathcal{M}_{\mu_0}$ with marginals given by $\rho, \rho' \in (\Delta_\mathcal{X})^N$, induced by the policy sequences $\pi, \pi'$ respectively. The divergence $\Gamma$ is a Bregman divergence induced by the function
    \begin{equation*}
     \psi(\mu) := \sum_{n=1}^N \phi(\mu_n) - \sum_{n=1}^N \phi(\rho_n).
    \end{equation*}
    Also, $\Gamma$ is $1$-strongly convex with respect to the $\sup_{1 \leq n \leq N} \|\cdot\|_1$ norm.
\end{proposition}
The proof is in Appendix~\ref{penalization_is_bregman_proof} and consists in showing and exploring that the $\Gamma$ divergence taking values on the marginal state-action distributions is in fact the KL divergence on the joint distribution. 

Next, if $f_n$ is convex and $l_n$ Lipschitz with respect to the norm $\|\cdot\|_1$ for any $1 \leq n \leq N$, then $F$ is also convex and Lipschitz with constant $\smash{L := (\sum_{n=1}^N l_n^2)^{\nicefrac{1}{2}}}$ (see Appendix~\ref{proof_conv_MD_MFC}). The proof that our cost functions for the DSM model satisfy these assumptions is given in Appendix~\ref{proof_conv_heater}. Since the set $\mathcal{M}_{\mu_0}$ is convex, we also satisfy the convexity assumptions for the convergence of the mirror descent. The rate of convergence is thus a direct consequence of the application of the proof of convergence of mirror descent for Problem~\eqref{convex_main_optimisation_problem}.
\end{proof}

\subsection{Potential games}\label{potential_games}
In Appendix~\ref{potential_games_discussion} we provide an equivalence between the MFC problem considered and a MFG by considering a game whose reward is given by $r_n(x_n, a_n, \mu_n) := - \nabla f_n(\mu_n)(x_n, a_n)$ for all $(x_n,a_n, \mu_n) \in \mathcal{X} \times \mathcal{A} \times \Delta_{\mathcal{X} \times \mathcal{A}}$. We call this type of game a potential game. The work done in \citet{concave_utility} and \citet{Pfeiffer_potential_mfg} relates the optimality conditions of optimization problems to the concept of Nash equilibrium in game problems. We do not go into details here, and leave more in-depth discussions to the Appendix section. The main purpose of this section is the realization that we can apply MFG algorithms to the DSM problem, which is a major breakthrough in this area because it opens up a new range of algorithms to this type of management system problems.

\section{Experiments}\label{experiments}

\subsection{Simulating the nominal dynamics}

To simulate the nominal dynamics, we use the nominal model presented in Appendix~\ref{nominal_behavior} and data from the SMACH (\textit{Simulation Multi-Agents des Comportements Humains}) platform \citep{smach} to approximate the probability of having a water withdrawal for each time step. In addition, we take a time frequency $\delta_t = 10$ minutes, and a temperature deadband with $T_\text{min} = 50^\circ C$ and $T_\text{max}=65^\circ C$. For more details on how the simulations are performed, see Appendix~\ref{water_heater_sim}. Figure~\ref{fig:sim_1000_oneweek} shows the simulation of the average drain and power consumption of $10^4$ water heaters following the nominal dynamics over the period of one week day respectively. The states (operating state and temperature) are randomly initialized for each water heater. 

\begin{figure}[ht]
\centering
\begin{subfigure}{.45\textwidth}
\centering
  \input{graphs/drain_onehour}
  \caption{Average drain in Joules}
  \label{fig:drain}
\end{subfigure}%
\begin{subfigure}{.45\textwidth}
  \centering
  \input{graphs/av_consumption_oneday}
  \caption{Average consumption}
  \label{fig:consump_1000}
\end{subfigure}
\caption{Average drain and power consumption for a simulation of $10^4$ water heaters over a period of one day.}
\label{fig:sim_1000_oneweek}
\end{figure}
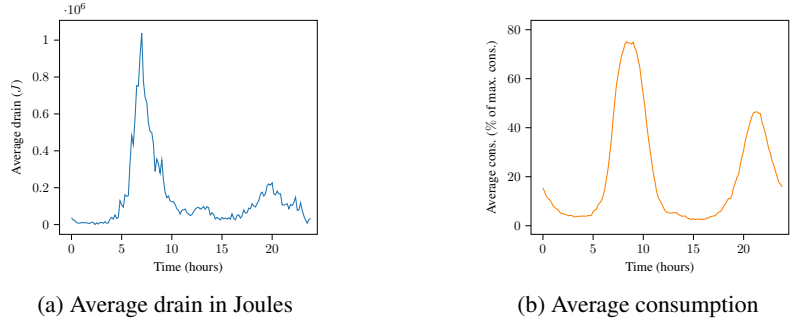

The target signal $\gamma = (\gamma_n)_{1 \leq n \leq N}$ is built as a sum of a baseline $b=(b_n)_{n\leq N}$ and a deviation signal $\lambda=(\lambda_n)_n$, $\gamma = \lambda + b(w)$, where $b(w)$ is the nominal dynamics obtained by simulating the water heaters (as in Figure~\ref{fig:sim_1000_oneweek}), and $w$ represents a random initialization of their states. If the deviation is zero, the average consumption is equal to the baseline. The deviation signal should have zero energy on the time considered for the simulations, i.e. $\smash{\sum_{n=0}^N \lambda_n = 0}$, in order to ensure a stationary process. We consider the two deviation signals illustrated in Figure~\ref{fig:deviations}.

\begin{figure}[ht]
\centering
\begin{subfigure}{.45\textwidth}
\centering
  \input{graphs/dev_one}
  \caption{One hour step.}
  \label{fig:one_hour_step}
\end{subfigure}%
\begin{subfigure}{.45\textwidth}
  \centering
  \input{graphs/dev_eight}
  \caption{Eight hours step.}
  \label{fig:eight_hours_step}
\end{subfigure}
\caption{Deviation signals $(\lambda_n)_{n\leq N}$}
\label{fig:deviations}
\end{figure}
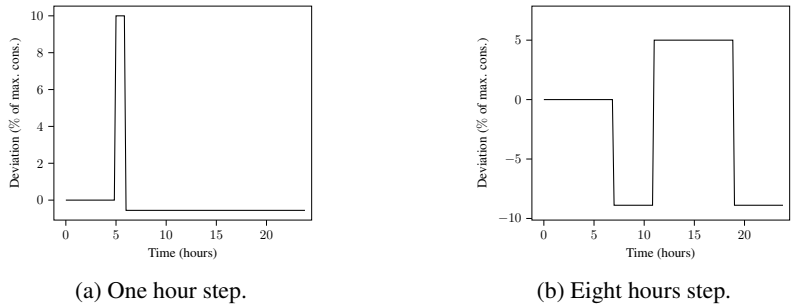

\subsection{Results}\label{results}

For a population of water heaters following the randomized dynamics we compare the optimal policy sequence obtained after $100$ iterations of MD-MFC, and two mean field game algorithms: Fictitious Play for MFG (FP-MFG) from \citet{FP_MFG_finitespace} and Online Mirror Descent for MFG (OMD-MFG) from \citet{OMD_perolat} (see Algorithms~\ref{alg:FP} and~\ref{alg:OMD_MFG} respectively in Appendix~\ref{algo_appen}). At each iteration, we compute a policy sequence of size $144$ (number of time steps). The heater's state space $\mathcal{X}$ is of size $2*41$ (two ON/OFF operating states times $41$ possible temperatures - integers from the ambient temperature $T_\text{amb} = 25$ to $T_\text{max} = 65$), and its action space $\mathcal{A}$ is of size $2$. We simulate each policy on $10^4$ water heaters and analyze the average consumption curve. The water heater's initial state distribution is equal to the initial distribution of the nominal consumption. The distribution of actions is initialized uniformly. The three algorithms have a memory complexity of order $N \times |\mathcal{X}| \times |\mathcal{A}|$, and a computational complexity of order $K \times N \times (|\mathcal{X}| \times |\mathcal{A}|)^2$.

In Figure~\ref{fig:policies_sim}, the consumption simulated by the best policies for all three algorithms appears to track the target better than the nominal consumption. This is not a surprise because all algorithms are supposed to converge to the same minima. However, they do so by finding different strategies and with different convergence rates. Figure~\ref{fig:obj_func} shows the logarithm of the objective function per iteration, and to visualize the policies obtained we plot in Figure~\ref{fig:policies}, at each time step  [$x$ axis], the probability of choosing the action $1$ (ON) [colors] for all possible temperatures between $T_\text{min} = 50$ and $T_\text{max} = 65$ [$y$ axis], when the current state is ON [up] or OFF [down]. The policies plots show that MD-MFC returns a more regular policy than FP-MFG.

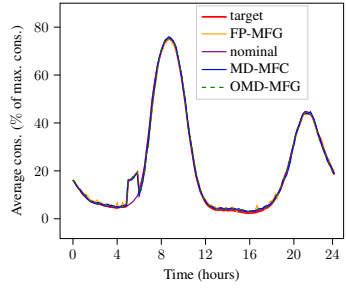
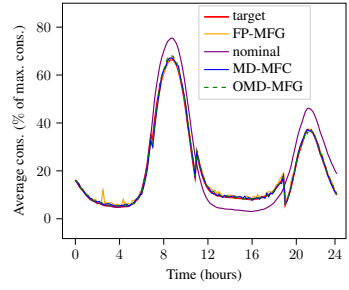
\begin{figure*}[t]
\centering
\begin{subfigure}{.5\textwidth}
\centering
  \input{graphs/consumption_target_one_10_4.tex}
  \caption{Target with one hour step dev.}
  \label{fig:cons_one}
\end{subfigure}%
\begin{subfigure}{.5\textwidth}
  \centering
  \input{graphs/consumption_target_eight_10_4heaters.tex}
  \caption{Target with eight hours step dev.}
  \label{fig:cons_eight}
\end{subfigure}

\caption{Simulation of the power consumption of $10^4$ water heaters for the optimal policy computed through different algorithms, for targets constructed with the deviations of one hour [left] and eight hours [right]. We compare with the nominal consumption (without deviation).}
\label{fig:policies_sim}
\end{figure*}

\begin{figure*}[t]
\centering
\begin{subfigure}{.51\textwidth}
\centering
  \includegraphics[scale=0.23]{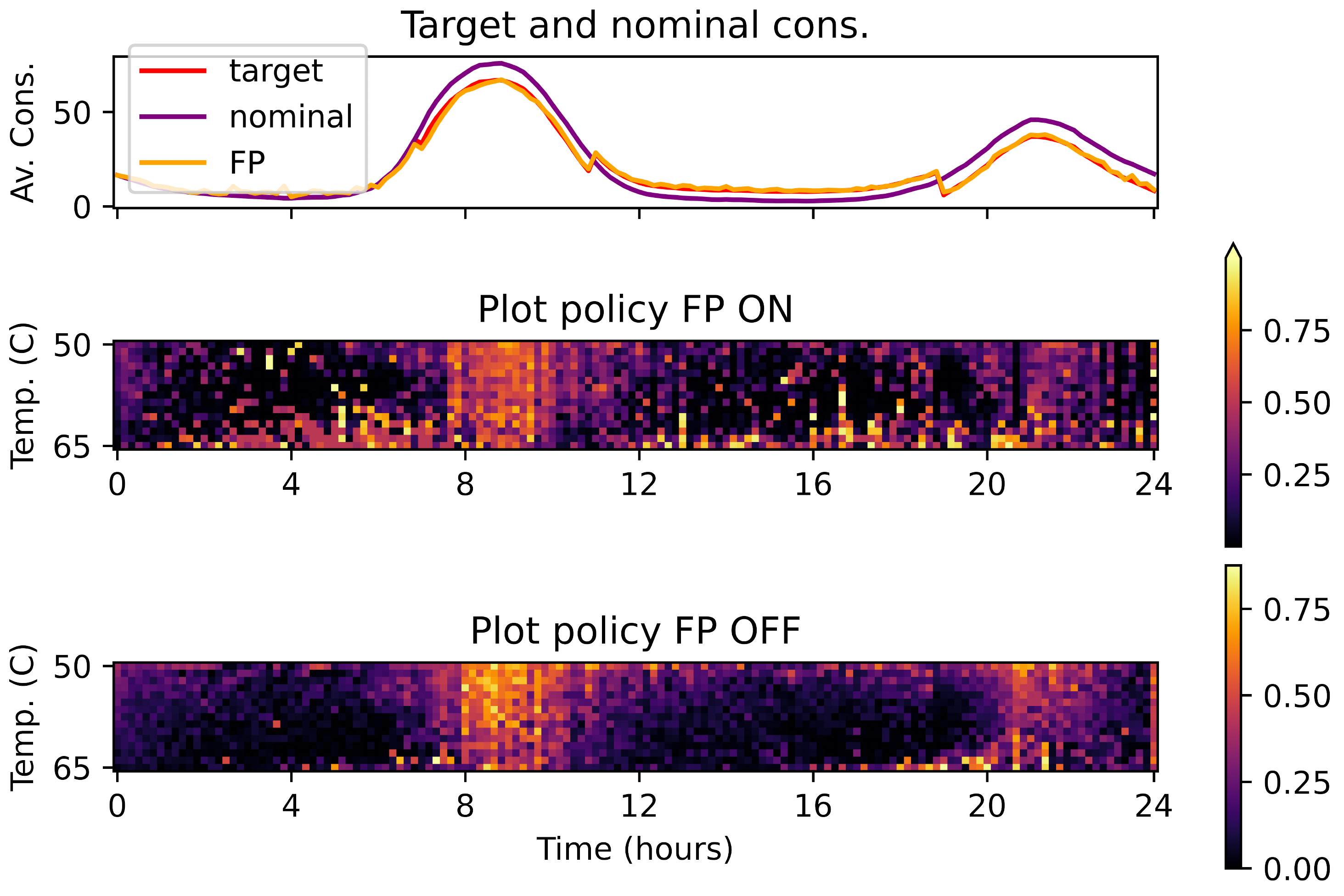}
  \caption{Policy FP-MFG.}
  \label{fig:FP_policy}
\end{subfigure}%
\begin{subfigure}{.51\textwidth}
  \centering
  \includegraphics[scale=0.23]{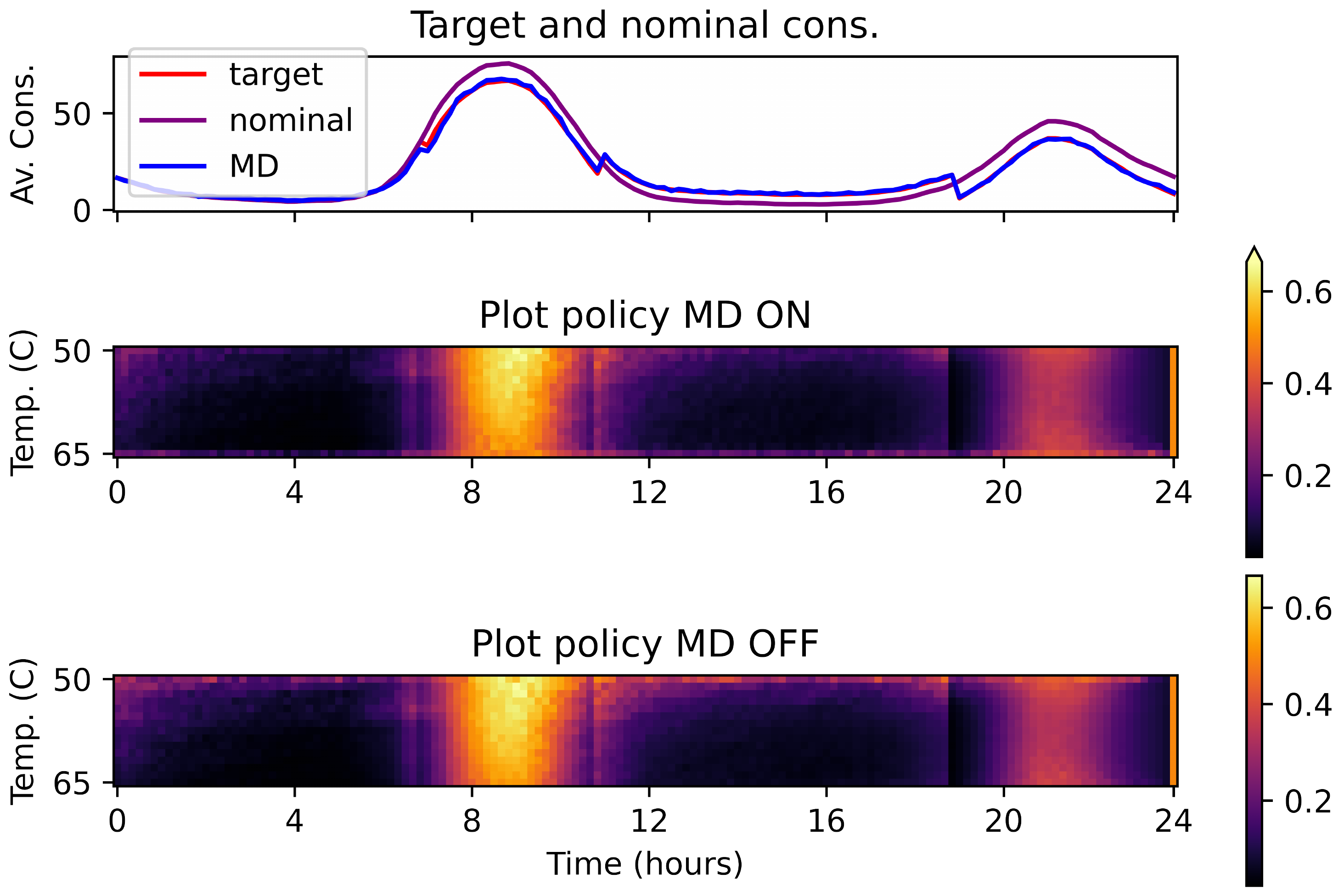}
  \caption{Policy MD-MFC.}
  \label{fig:MD_policy}
\end{subfigure}
\caption{[top] Target, average consumption obtained by the nominal policy and by the policy computed by FP-MFG (left) and MD-MFC (right). [middle] Probability of choosing the ON action when in the ON state. [bottom] Probability of choosing the ON action when in the OFF state. For all temperatures between $T_\text{min}=50$ and $T_\text{max} = 65$ [$y$ axis], over the course of a day with a time step of $10$ minutes [$x$ axis], for a target with a deviation step of eight hours.}
\label{fig:policies}
\end{figure*}

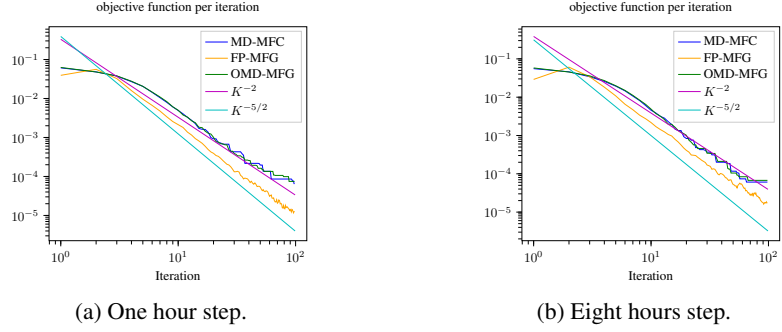
\begin{figure}
\centering
\begin{subfigure}{.45\textwidth}
\centering
  \input{graphs/logobjfunc_onehour_line}
  \caption{One hour step.}
  \label{fig:obj_func_one}
\end{subfigure}%
\begin{subfigure}{.45\textwidth}
  \centering
  \input{graphs/logobjfunc_eighthours_line}
  \caption{Eight hours step.}
  \label{fig:obj_func_eight}
\end{subfigure}
\caption{Log-log plot of the objective function per iteration for each method when using a target with an one hour step [left] and eight hours step [right] deviations.}
\label{fig:obj_func}
\end{figure}

\paragraph{Different initialisations impact the number of switches}
We noticed that different initialisations of MD-MFC lead to different policies. Given a state distribution sequence $\rho$, the policy generating this distribution is not necessarily unique. In particular, these policies, while providing the same $\rho$, may differ in terms of the average number of ON/OFF switches induced over the time horizon considered. In our model, no switching limit is assumed, but a large number of switches can be detrimental to the device. This non-uniqueness helps us reduce switch count without adding new constraints by finding multiple policies that achieve the right consumption and selecting the one with the fewest switches. This can also be useful for MFC problems in other areas, e.g. transaction costs in finance.


In the case illustrated here the average number of daily switches is $33$, while the nominal dynamic averages only $3$ switches per day. By initializing the MD-MFC algorithm with a policy that is a $0.1$ deviation from the nominal policy as in Figure~\ref{fig:nominal_policy_1}, we find that the number of switches decreases to a daily average of $9.2$ while still following the target curve, see Figure~\ref{fig:MD_nominal_policy_1}. The same does not happen with FP-MFG, which makes it less interesting for the real-world scenarios we consider here.

Finally, Table~\ref{comparing_algorithms} gives a global comparison between the three algorithms. FP-MFG converges faster but is not suitable for controlling switch count, being less interesting for the considered DSM problem, and needs a smooth objective function assumption. OMD-MFG is empirically as good as MD-MFC but lacks convergence proof for discrete cases.

\begin{figure*}[t]
\centering
\begin{subfigure}{.51\textwidth}
\centering
  \includegraphics[scale=0.23]{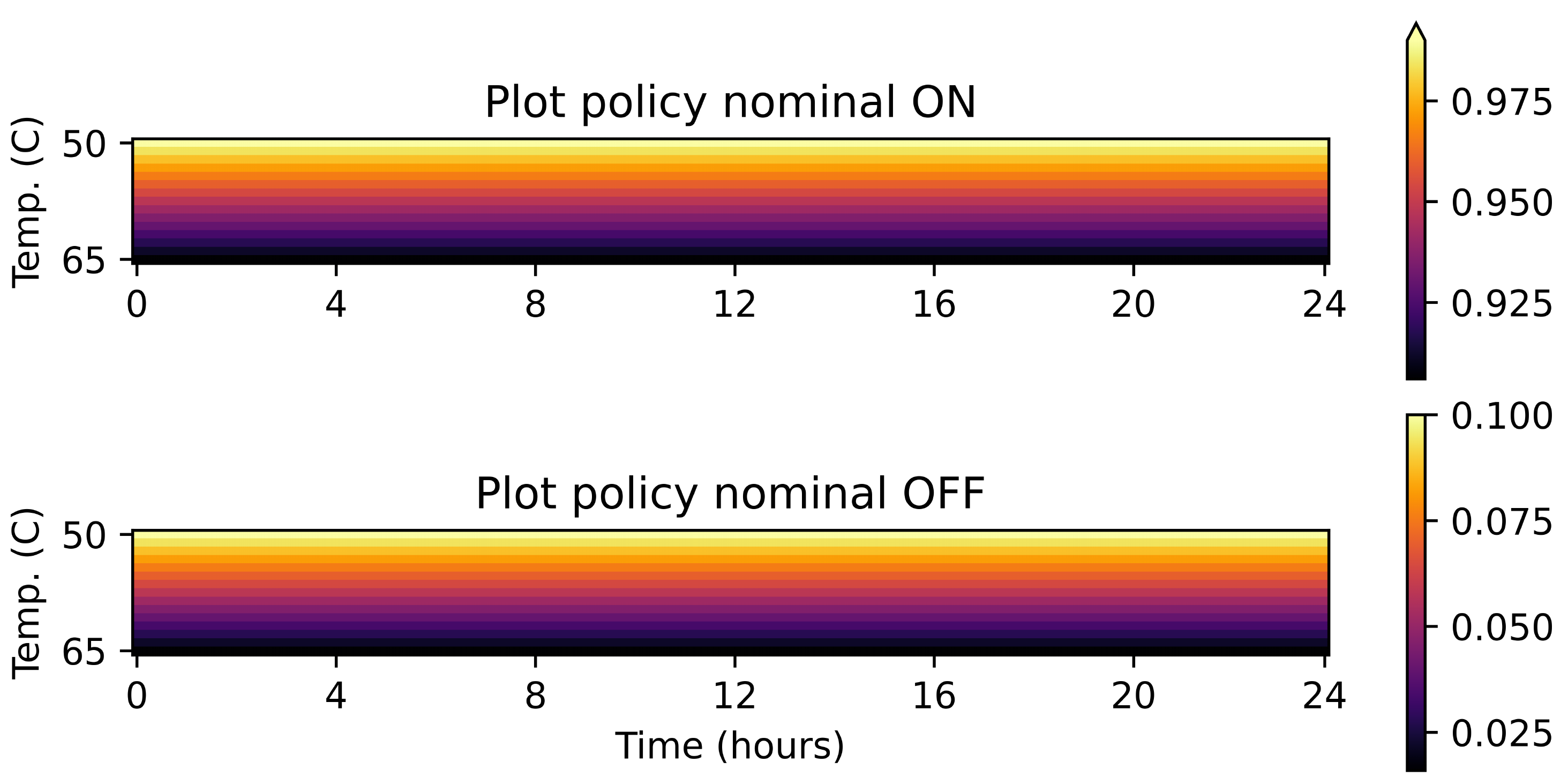}
  \caption{Nominal policy deviation.}
  \label{fig:nominal_policy_1}
\end{subfigure}%
\begin{subfigure}{.51\textwidth}
  \centering
  \includegraphics[scale=0.23]{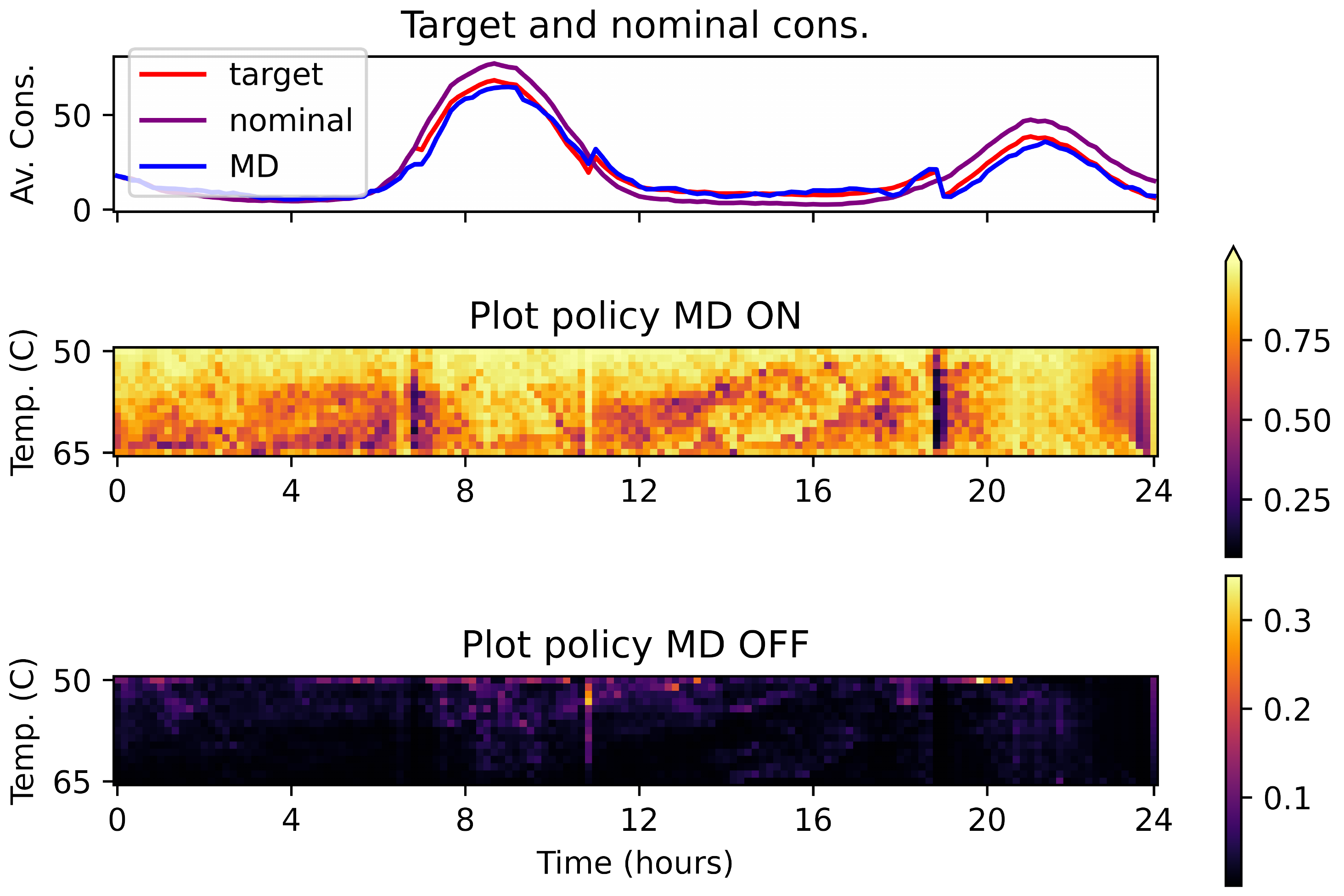}
  \caption{Policy MD-MFC with nominal policy deviation.}
  \label{fig:MD_nominal_policy_1}
\end{subfigure}
\caption{[left] Initial policy sequence $\pi^0$ with a deviation of $0.1$ from the nominal policy. [right] Output policy sequence of Algorithm~\ref{alg:MD} initialized with the policy at left.}
\label{fig:policies_no_switch}
\end{figure*}

\begin{table}
  \caption{Comparing MD-MFC, OMD-MFG and FP-MFG}
  \label{comparing_algorithms}
  \centering
  \begin{tabular}{llll}
    \toprule
    Algorithm   & \makecell{Convergence \\ rate} & \makecell{Flexibility on \\applications} & \makecell{Convergence \\ hypothesis} \\
    \midrule
    MD-MFC & \(K^{-1/2}\) & Yes (switches) & convex + Lispschitz     \\
    OMD-MFG & no proof & Yes (switches) &   convex + Lispschitz    \\
    FP-MFG & \(K^{-1}\) & No & convex + Lipschitz + smooth   \\
    \bottomrule
  \end{tabular}
\end{table}


\section{Future work}\label{future_work}
Future work involves adapting existing algorithms to real-time algorithms, proposing schemes where each iteration corresponds to a time step. We further aim to generalize to a model-free scenario, learning user behavior on the fly while preserving privacy with partially observable states. Moreover, we believe we can extend our approach to the accelerated version of mirror descent \citep{accelerated_md} providing a better theoretical convergence rate of order $1/K^{2}$.

\bibliography{bib}

\newpage
\appendix
\onecolumn
\section{Missing proofs}\label{missing_proofs}

\subsection{Proof of Proposition \ref{opt_mu_equal_pi}}

\begin{proof}
Consider a fixed initial state-action distribution $\mu_0 \in \Delta_{\mathcal{X} \times \mathcal{A}}$. Let $\mu \in \mathcal{M}_{\mu_0}$ and define $\rho = (\rho_n)_{1 \leq n \leq N}$ such that for all $x \in \mathcal{X}$, $\rho_n(x) = \sum_{a} \mu_n(x,a)$  (the associated state distribution). First, let us deal with the case where $\rho_n(x) \neq 0$. Define a policy sequence $\pi \in (\Delta_\mathcal{A})^{\mathcal{X} \times N}$ such that $\smash{\pi_{n}(a|x) = \frac{\mu_{n}(x,a)}{\rho_{n}(x)}}$ for all $(x,a) \in \mathcal{X} \times \mathcal{A}$. We want to show that $\mu^\pi = \mu$ for this policy $\pi$. We reason by induction. For $n=0$, $\mu_0^\pi = \mu_0$ by definition. Suppose $\mu^\pi_n = \mu_n$, thus for $n+1$ and for all $(x',a') \in \mathcal{X} \times \mathcal{A}$ 
    \begin{align*}
        \mu^\pi_{n+1}(x',a') &= \sum_{x,a} p_{n+1}(x'|x,a) \mu_n^\pi(x,a) \pi_{n+1}(a'|x') \\
        &= \sum_{x,a} p_{n+1}(x'|x,a) \mu_n(x,a) \frac{\mu_{n+1}(x',a')}{\rho_{n+1}(x')} \\
        &= \sum_{a} \mu_{n+1}(x',a) \frac{\mu_{n+1}(x',a')}{\rho_{n+1}(x')} \\
        &= \rho_{n+1}(x') \frac{\mu_{n+1}(x',a')}{\rho_{n+1}(x')} \\
        &= \mu_{n+1}(x',a'),
    \end{align*}
where the first equality comes from Definition~\ref{mu_induced_by_pi}, the second equality comes from the induction assumption and the way we defined the strategy $\pi$, and the third comes from the assumption that $\mu \in \mathcal{M}_{\mu_0}$.

In the case $\rho_n(x) = 0$, we therefore have $\mu_n(x,a) = 0$ for all $a \in \mathcal{A}$, so any choice of $\pi_n(a|x)$ would work.

\end{proof}


\section{Missing proofs: algorithm~\ref{alg:MD} scheme and convergence rate}\label{missing_proofs_2}

By abuse of notations, for any probability measure $\eta \in \Delta_E$ whatever the finite space $E$ on which it is defined we introduce the neg-entropy function, with the convention $0 \log(0) = 0$,
$$
\phi (\eta):=\sum_{x\in E} \eta (x)\log \eta (x) ,
$$ 
 to which we associate the Bregman divergence $D$, also known as the KL divergence, such that for any pair $(\eta,\nu)\in\Delta_E\times \Delta_E$, 
$$
D(\eta,\nu):=\phi(\eta)-\phi(\nu)-\langle \phi'(\nu),\eta-\nu\rangle.
$$

Let $\rho_n$ denote the marginal probability distribution on $\mathcal{X}$ associated with $\mu_n$ i.e., for all $x \in \mathcal{X}$
$$
\rho_n(x):= \sum_{a \in \mathcal{A}} \mu_n(x,a)\ .
$$
Observe that to any $\mu=(\mu_n)_{1 \leq n \leq N}\in \mathcal{M}_{\mu_0}$ one can associate a unique probability mass function on $\mathcal{P}(\mathcal{X}\times \mathcal{A})^{N}$ denoted by $\mu_{1:N}$ such that $\mu_{1:N}$ is \textit{generated} by the strategy $\pi=(\pi_n)_{1 \leq n \leq N}$ associated with $\mu$ which is determined by 
$$
\pi_{n}(a\vert x)=\frac{\mu_n(x,a)}{\rho_n(x)}\ ,
$$
when $\rho_n(x) \neq 0$, otherwise we fix an arbitrary strategy $\pi_n(a \vert x) = \frac{1}{|\mathcal{A}|}$.

Before proving Theorems~\eqref{MD_explicit_result} and \eqref{thm:convergence_rate} we state and prove a Lemma which is key to proving both theorems.

\begin{lemma}\label{prop:D_decomp}
For any $\mu \in \mathcal{M}_{\mu_0}$ and $\mu' \in \mathcal{M}_{\mu_0}^*$, with associated probability mass functions $\mu_{1:N},\mu'_{1:N}\in \mathcal{P}\big ((\mathcal{X}\times \mathcal{A})^{N}\big)$ \textit{generated} by $\pi,\pi'$ respectively with the same initial state-action distribution, i.e. $\mu_0 = \mu'_0$, we have
\begin{equation}\label{eq:youpi}
\begin{split}
D(\mu_{1:N},\mu'_{1:N})&=\sum_{n=1}^{N} \mathbb{E}_{(x,a) \sim \mu_n(\cdot)}\left[\log\bigg(\frac{\pi_{n}(a|x)}{\pi'_{n}(a|x)}\bigg)\right]\\
&={\displaystyle \sum_{n=1}^{N} D(\mu_n,\mu'_n)- \sum_{n=0}^{N} D(\rho_n,\rho'_n)}
\end{split}
\end{equation}
\end{lemma}

\begin{proof}
For each $1 \leq n \leq N$, let us define a transition matrix $P^{\pi_n}$ for all $x, x' \in \mathcal{X}$ and $a, a' \in \mathcal{A}$,
\begin{equation*}
    P^{\pi_n}(x',a'|x,a) := p_n(x'|x,a) \pi_n(a'|x').
\end{equation*}
Given Definition~\ref{mu_induced_by_pi}, for any randomized policy the state-action distributions evolve according to linear dynamics
\begin{equation*}
\mu_n(x',a') = \langle \mu_{n-1}(\cdot), P^{\pi_n}(x',a'| \cdot) \rangle.
\end{equation*}
Any randomized policy $\pi$ gives a probability mass function $\mu_{1:N}$  that is Markovian:
\begin{equation}\label{mu_matrix_decomp}
    \mu_{1:N}(\vec{y}) = \mu_0(y_0) P^{\pi_1}(y_1|y_0)...P^{\pi_N}(y_N|y_{N-1}),
\end{equation}
where $\vec{y}$ represents the elements of $(\mathcal{X} \times \mathcal{A})^{N+1}$ such that $y_i = (x_i,a_i)$ for all $0\leq i \leq N$. Note that $\mu_n(y_n)$ is the marginal probability mass function. 

Consider $\mu, \mu' \in \mathcal{M}_{\mu_0}$ the state-action distribution sequences induced by $\pi, \pi'$ respectively (i.e, $\mu = \mu^\pi$ and $\mu' = \mu^{\pi'})$. Thus, computing the relative entropy between the probability mass functions $\mu_{1:N},\mu'_{1:N}$ gives
\begin{align*}
    D(\mu_{1:N}, \mu'_{1:N}) &= \sum_{\vec{y}} \mu_{1:N}(\vec{y}) \log\left(\frac{\mu_{1:N}(\vec{y})}{\mu'_{1:N}(\vec{y})}\right) \\
    &= \sum_{y_0,...,y_N} \mu_{1:N}(\vec{y}) \log\left(\frac{\mu_0(y_0) P^{\pi_1}(y_1|y_0)...P^{\pi_N}(y_N|y_{N-1})}{\mu'_0(y_0) P^{\pi'_1}(y_1|y_0)...P^{\pi'_N}(y_N|y_{N-1})}\right) \\
    &= \sum_{y_0,...,y_N} \mu_{1:N}(\vec{y}) \sum_{i=1}^N\log\left(\frac{P^{\pi_i}(y_i|y_{i-1})}{P^{\pi'_i}(y_i|y_{i-1})}\right).
\end{align*}

Where
\begin{align*}
    \sum_{i=1}^N\log\left(\frac{P^{\pi_i}(y_i|y_{i-1})}{P^{\pi'_i}(y_i|y_{i-1})}\right) &= \sum_{i=1}^N\log\left( \frac{p_i(x_i|x_{i-1},a_{i-1}) \pi_i(a_i|x_i)}{p_i(x_i|x_{i-1},a_{i-1}) \pi_i'(a_{i}|x_i)} \right) \\
    &= \sum_{i=1}^N \log\left( \frac{ \pi_i(a_i|x_i)}{\pi_i'(a_{i}|x_i)} \right).
\end{align*}

Thus,
\begin{align*}
    D(\mu_{1:N}, \mu'_{1:N}) &= \sum_{\vec{y}} \mu_{1:N}(\vec{y}) \sum_{i=1}^N \log\left( \frac{ \pi_i(a_i|x_i)}{\pi_i'(a_{i}|x_i)} \right) \\
    &= \sum_{\vec{y}} \mu_0(y_0) P^{\pi_1}(y_1|y_0)...P^{\pi_N}(y_N|y_{N-1}) \sum_{i=1}^N \log\left( \frac{ \pi_i(a_i|x_i)}{\pi_i'(a_{i}|x_i)} \right) \\
    &= \sum_{i=1}^N \sum_{x \in \mathcal{X}}\sum_{a \in \mathcal{A}} \mu_i(x,a) \log\left(\frac{\pi_i(a|x)}{\pi'_i(a|x)}\right).
\end{align*}

Where for the last equality we used that
\[
    \sum_{y_0,...,y_{i-1}} \mu_0(y_0) P^{\pi_1}(y_1|y_0)...P^{\pi_{i}}(y_{i}| y_{i-1}) = \sum_{y_{i}} \mu_{i}(y_{i}) 
\]
and for a fixed $y_i$,
\[
    \sum_{y_{i+1},...,y_N} P^{\pi_{i+1}}(y_{i+1}|y_{i})...
    P^{\pi_{N}}(y_{N}| y_{N-1}) = 1.
\]

This proves the first equality of the Lemma. We now prove the second. For this, we recall that Proposition~\ref{opt_mu_equal_pi} gives a unique relation between a state-action distribution sequence $\mu \in \mathcal{M}_{\mu_0}$ and the policy sequence $\pi \in (\Delta_\mathcal{A})^{\mathcal{X} \times N}$ inducing it by taking for all $1 \leq i \leq N$, $(x,a) \in \mathcal{X} \times \mathcal{A}$, 
$$
\pi_{i}(a\vert x)=\frac{\mu_i(x,a)}{\rho_i(x)},\,
$$
where $\rho$ is the marginal on the states of $\mu$.
Using this relation, we have then that
\begin{equation*}
    \begin{split}
         D(\mu_{1:N},\mu'_{1:N}) &=  \sum_{i=1}^{N} \sum_{x \in \mathcal{X}} \sum_{a \in \mathcal{A}} \mu_{i}(x,a) \log\bigg(\frac{\pi_{i}(a|x)}{\pi'_{i}(a|x)}\bigg) \\
         &=  \sum_{i=1}^{N} \sum_{x \in \mathcal{X}} \sum_{a \in \mathcal{A}} \mu_{i}(x,a) \log\bigg(\frac{\mu_{i}(a|x)}{\rho_{i}(x)} \frac{\rho'_{i}(x)}{\mu'_{i}(a|x)}\bigg) \\
         &= \sum_{i=1}^{N} \sum_{x \in \mathcal{X}} \sum_{a \in \mathcal{A}} \mu_{i}(x,a) \log\bigg(\frac{\mu_{i}(a|x)}{\mu'_{i}(a|x)}\bigg) - \sum_{i=1}^{N} \sum_{x \in \mathcal{X}} \sum_{a \in \mathcal{A}} \mu_{i}(x,a) \log\bigg(\frac{\rho_{i}(x)}{\rho'_{i}(x)}\bigg) \\
         &=  \sum_{i=1}^{N} \sum_{x \in \mathcal{X}} \sum_{a \in \mathcal{A}} \mu_{i}(x,a) \log\bigg(\frac{\mu_{i}(a|x)}{\mu'_{i}(a|x)}\bigg) - \sum_{i=1}^{N} \sum_{x \in \mathcal{X}} \rho_{i}(x) \log\bigg(\frac{\rho_{i}(x)}{\rho'_{i}(x)}\bigg) \\
         &= \sum_{i=1}^{N} D(\mu_i, \mu_i') - \sum_{i=1}^{i} D(\rho_i, \rho_i') 
    \end{split}
\end{equation*}
which concludes the proof.
\end{proof}

\subsection{Proof of Theorem \ref{MD_explicit_result}: formulation of Algorithm~\ref{alg:MD}}\label{thm_1_proof}
\begin{proof}

At each iteration we seek to solve
\begin{equation}
\label{MD_opt_problem_appendix}
\begin{split}
\mu^{k+1} \in \argmin_{\mu^\pi\in\mathcal{M}_{\mu_0}}
\bigg\{
\langle \nabla F(\mu^k),\mu^\pi\rangle 
+\frac{1}{\tau_k} \sum_{n=1}^{N} \mathbb{E}_{(x,a) \sim \mu_n(\cdot)}\left[\log\bigg(\frac{\pi_{n}(a|x)}{\pi^k_{n}(a|x)}\bigg)\right]
\bigg\}\,
\end{split}
\end{equation}

where recall that $\langle \nabla F(\mu^k),\mu^\pi\rangle := \sum_{n=1}^N \langle \nabla f_n(\mu_n^k), \mu^\pi_n \rangle$. We further use that ${r_n(x_n, a_n, \mu_n) := - \nabla f_n(\mu_n)(x_n, a_n)}$.

Now, we use the optimality principle to solve this optimization problem with an algorithm backward in time. Remember that the initial distribution $\mu_0$ is always fixed. The equivalence between solving a minimization problem on sequences of state-action distributions in $\mathcal{M}_{\mu_0}$ and on sequences of policies in $(\Delta_\mathcal{A})^{\mathcal{X} \times N}$ (see Proposition~\ref{opt_mu_equal_pi}), allows us to reformulate Problem~\eqref{MD_opt_problem_appendix} on $\mathcal{M}_{\mu_0}$ into a problem on $(\Delta_\mathcal{A})^{\mathcal{X} \times N}$, thus

\begin{equation*}
    \begin{split}
        \eqref{MD_opt_problem_appendix} &= \max_{\pi \in (\Delta_\mathcal{A})^{\mathcal{X} \times N}} \bigg\{\sum_{n=0}^N \sum_{x,a} \mu_n^\pi(x,a) r_n(x,a,\mu_n^k) \\
        &\qquad- \frac{1}{\tau_k} \sum_{n=1}^N \sum_{x,a} \mu_{n-1}^\pi(x,a) \sum_{x',a'} p_n(x'|x,a) \pi_n(a'|x') \log\left( \frac{\pi_n(a'|x')}{\pi_n^k(a'|x')}\right) \bigg\} \\
        &= \max_{\pi \in (\Delta_\mathcal{A})^{\mathcal{X} \times N}} \bigg\{\sum_{n=0}^N \sum_{x,a} \mu^\pi_n(x,a) \bigg[ r_n(x,a,\mu_n^k) \\
        &\qquad - \frac{1}{\tau_k} \sum_{x',a'} p_{n+1}(x'|x,a) \pi_{n+1}(a'|x') \log\left( \frac{\pi_{n+1}(a'|x')}{\pi_{n+1}^k(a'|x')}\right) \bigg] \bigg\}\\
        &= \max_{\pi \in (\Delta_\mathcal{A})^{\mathcal{X} \times N}} \bigg\{\mathbb{E}_\pi \bigg[ r_N(x_N,a_N,\mu_N^k) + \sum_{n=0}^{N-1}  r_n(x_n,a_n,\mu_n^k) \\
        &\qquad- \frac{1}{\tau_k} \sum_{x',a'} p_{n+1}(x'|x_n,a_n) \pi_{n+1}(a'|x') \log\left( \frac{\pi_{n+1}(a'|x')}{\pi_{n+1}^k(a'|x')}\right) \bigg] \bigg\}.
    \end{split}
\end{equation*}

Let us define a regularized version of the state-action value function that we denote by $\Tilde{Q}^k$, such that for all $1 \leq i \leq N$, $(x,a) \in \mathcal{X} \times \mathcal{A}$,
\begin{equation}
\begin{split}
    \Tilde{Q}^k_i(x,a) &=  \max_{\pi_{i+1:N} \in (\Delta_\mathcal{A})^{\mathcal{X} \times {N-i}}} \mathbb{E}_\pi \bigg[ r_N(x_N,a_N,\mu_N^k) + \sum_{n=i}^{N-1} \bigg\{ r_n(x_n,a_n,\mu_n^k) \\
    &- \frac{1}{\tau_k} \sum_{x',a'} p_{n+1}(x'|x_n,a_n) \pi_{n+1}(a'|x') \log\left( \frac{\pi_{n+1}(a'|x')}{\pi_{n+1}^k(a'|x')}\right) \bigg\} \bigg| (x_i,a_i) = (x,a) \bigg],  
\end{split}
\end{equation}
where $\pi_{i+1:N} = \{\pi_{i+1},...,\pi_N\}$.

First, note that $\mathbb{E}_{(x,a) \sim \mu_0(\cdot)}[\Tilde{Q}^k_0(x,a)] =$ \eqref{MD_opt_problem_appendix}. Moreover, the optimality principle states that this regularized state-action value function satisfies the following recursion
\begin{equation*}
    \begin{cases}
    \Tilde{Q}_N(x,a) = r_N(x,a, \mu_N^k) \\
    \!\begin{aligned}
            \Tilde{Q}_i(x,a) &= \max_{\pi_{i+1} \in (\Delta_\mathcal{A})^{\mathcal{X}}} \bigg\{ r_i(x,a,\mu_i^k) + \\
            &\qquad \sum_{x'} p_{i+1}(x'|x,a) \sum_{a'} \pi_{i+ 1}(a'|x') \left[  - \frac{1}{\tau_k}\log\left( \frac{\pi_{i+1}(a'|x')}{\pi_{i+1}^k(a'|x')}\right)  + \Tilde{Q}_{i+1}(x',a') \right] \bigg\}.
    \end{aligned}
    \end{cases}
\end{equation*}

Thus, to solve \eqref{MD_opt_problem_appendix} we compute backwards in time, i.e. for $i = N-1,...,0$, for all $x \in \mathcal{X}$,
\begin{equation*}
    \pi_{i+1}^{k+1}(\cdot|x) \in \argmax_{\pi(\cdot|x) \in \Delta_\mathcal{A}} \left\{\big\langle \pi(\cdot|x), \Tilde{Q}_{i+1}^k(x,\cdot) \big\rangle - \frac{1}{\tau_k} D\big(\pi(\cdot|x), \pi^k_{i+1}(\cdot|x)\big) \right\},
\end{equation*}
where $D$ is the KL divergence.

The solution of this optimisation problem for each time step $i$ can be found by writing the Lagrangian function $\mathcal{L}$ associated. Let $\lambda$ be the Lagrangian multiplier associated to the simplex constraint. For simplicity, let $\pi_x := \pi(\cdot|x)$, $\pi_x^k := \pi_{i+1}^k(\cdot|x)$ and $\tilde{Q}^k_x := \tilde{Q}_{i+1}^{k}(x,\cdot)$. Thus,
\begin{equation*}
    \mathcal{L}(\pi_x, \lambda) = \langle \pi_x, \tilde{Q}^k_x \rangle - \frac{1}{\tau_k} D(\pi_x, \pi_x^k) - \lambda \left(\sum_{a \in \mathcal{A}} \pi_x(a) - 1 \right).
\end{equation*}
Taking the gradient of the Lagrangian with respect to $\pi_x(a)$ for each $a \in \mathcal{A}$ gives
\begin{equation*}
    \frac{\partial \mathcal{L}}{\partial \pi_x(a)} = \tilde{Q}^k_x(a) - \frac{1}{\tau_k} \log\left(\frac{\pi_x(a)}{\pi_x^k(a)}\right) - \frac{1}{\tau_k} - \lambda,
\end{equation*}
and thus
\begin{equation*}
    \begin{split}
        \frac{\partial \mathcal{L}}{\partial \pi_x(a)} = 0 \Longrightarrow \quad \pi_x(a) = \pi_x^k(a) \exp{\left(\tau_k \tilde{Q}_x^k(a) - 1 - \tau_k\lambda\right)}.
    \end{split}
\end{equation*}

Applying the simplex constraint, $\sum_{a\in \mathcal{A}} \pi_x(a) = 1$, we find the value of the Lagrangian multipler $\lambda$, and we get for all $a \in \mathcal{A}$

\begin{equation*}
   \pi_x(a) = \frac{\pi_x^k(a) \exp{\left(\tau_k \tilde{Q}_x^k(a)\right)}}{\sum_{a'\in \mathcal{A}} \pi_x^k(a') \exp{\left(\tau_k \tilde{Q}_x^k(a')\right)},}
\end{equation*}
which proves the theorem. 

\end{proof}

\subsection{Proof of Proposition~\ref{penalization_is_bregman}}\label{penalization_is_bregman_proof}

\begin{proof}
 Lemma~\ref{prop:D_decomp} states that 
 \begin{equation*}
     \Gamma(\mu, \mu') := \sum_{n=1}^{N} \mathbb{E}_{(x,a) \sim \mu_n(\cdot)}\left[\log\bigg(\frac{\pi_{n}(a'|x')}{\pi^k_{n}(a'|x')}\bigg)\right]
={\displaystyle \sum_{t=0}^{n} D(\mu'_t,\mu_t)- \sum_{t=0}^{n} D(\rho'_t,\rho_t)}.
 \end{equation*}
 
Recall that $\phi$ is the negentropy and that $D$ is the Bregman divergence induced by the negentropy. Define the function $\psi: (\Delta_{\mathcal{X} \times \mathcal{A}})^N \to \mathbb{R}$ such that
\begin{equation*}
        \psi(\mu) := \sum_{n=0}^N \phi(\mu_n) - \sum_{n=0}^N \phi(\rho_n).
\end{equation*}
Note that for $\mu, \mu' \in (\Delta_{\mathcal{X} \times \mathcal{A}})^N $ with marginals given by $\rho, \rho' \in (\Delta_{\mathcal{X}})^N$, using the second equality of Lemma~\ref{prop:D_decomp}, 
    \begin{equation*}
        \psi(\mu) - \psi(\mu') - \langle \nabla \psi(\mu'), \mu - \mu' \rangle = \Gamma(\mu, \mu').
    \end{equation*}
    Thus, for $\Gamma$ to be a Bregman divergence it is sufficient to show that $\psi$ is a convex function. Recall that the marginal $\rho$ is such that for each $1 \leq n \leq N$, and for all $x \in \mathcal{X}$, $\rho_n(x) = \sum_{a \in \mathcal{A}} \mu_n(x,a)$. Thus,
    \begin{equation*}
    \begin{split}
        \psi(\mu) &= \sum_n \left[\sum_{x,a} \mu_n(x,a) \log(\mu_n(x,a)) - \sum_{x} \rho_n(x) \log(\rho_n(x)) \right]  \\
        &= \sum_n \sum_{x,a} \mu_n(x,a) \log\left(\frac{\mu_n(x,a)}{\sum_{a'}\mu_n(x,a')}\right).
    \end{split}
    \end{equation*}
    Computing the first order partial derivative of $\psi$ with respect to $\mu_n(x,a)$ for any  $(x,a) \in \mathcal{X} \times \mathcal{A}$ and $1 \leq n \leq N$, we get
        \begin{equation*}
    \begin{split}
        \frac{\partial \psi}{\partial \mu_n(x,a)}(\mu) &= \log\left(\frac{\mu_n(x,a)}{\sum_{a'}\mu_n(x,a')}\right) + \mu_n(x,a)\frac{1}{\mu_n(x,a)} - \sum_{a'}\mu_n(x,a') \frac{1}{\sum_{a'}\mu_n(x,a')} \\
        &= \log\left(\frac{\mu_n(x,a)}{\sum_{a'}\mu_n(x,a')}\right) \\
        &=  \log\left(\frac{\mu_n(x,a)}{\rho_n(x)}\right).
    \end{split}
    \end{equation*}
    
    Now we apply the following convexity property \citep{boyd}: $\psi$ is convex if and only if for all $\mu, \mu' \in (\Delta_{\mathcal{X} \times \mathcal{A}})^N $, $\langle \psi'(\mu) - \psi'(\mu'), \mu - \mu'\rangle \geq 0$.
    Indeed,
    \begin{equation*}
    \begin{split}
        \langle \psi'(\mu) - \psi'(\mu'), \mu - \mu'\rangle &= \sum_n \sum_{x,a} \left[\frac{\partial \psi}{\partial \mu_n(x,a)}(\mu) - \frac{\partial \psi}{\partial \mu_n(x,a)}(\mu') \right] \big(\mu_n(x,a) - \mu '_n(x,a)\big) \\
        &= \sum_n \sum_{x,a} \left[\log\left(\frac{\mu_n(x,a)}{\rho_n(x)}\right) - \log\left(\frac{\mu'_n(x,a)}{\rho'_n(x)}\right) \right] \big(\mu_n(x,a) - \mu '_n(x,a)\big) \\
        &\overset{(a)}{=} \sum_n D(\mu_n, \mu_n') + D(\mu_n, \mu'_n) - D(\rho_n, \rho'_n) - D(\rho'_n, \rho_n) \\
        &\overset{(b)}{=} \Gamma(\mu, \mu') + \Gamma(\mu', \mu) \\
        &\overset{(c)}{=} D(\mu_{1:N}, \mu'_{1:N}) + D(\mu'_{1:N}, \mu_{1:N}) \overset{(d)}{\geq} 0,
    \end{split}
    \end{equation*}
    where $(a)$ comes from the definition of the KL divergence $D$, $(b)$ comes from the definition of $\Gamma$, $(c)$ comes from Lemma~\ref{prop:D_decomp} and $(d)$ comes from a property of Bregman divergences that they are always positive.
    As $\psi$ is convex and induces the divergence $\Gamma$ then $\Gamma$ is a Bregman divergence.
    After writing this proof, we came across a different strategy to prove that $\Gamma$ is a Bregman divergence that is presented in Appendix A of \citet{Neu17}.

Now we prove that $\Gamma$ is $1$-strongly convex with respect to the $\sup_{1 \leq n \leq N}\|\cdot\|_1$ norm. By Lemma~\ref{prop:D_decomp},
    \begin{equation*}
        \begin{split}
            \Gamma(\mu, \mu') &= \sum_{n=1}^N D(\mu_n,\mu_n')-\sum_{n=1}^N D(\rho_n,\rho'_n) \\
            &=D(\mu_{1:N},\mu'_{1:N}) \\
            &\geq 2\Vert \mu_{1:N}-\mu_{1:N}\Vert^2_{\textrm{TV}}\\
            &=\frac{1}{2} \Vert \mu_{1:N}-\mu'_{1:N}\Vert_1^2,
        \end{split}
    \end{equation*}
the last inequality being a consequence of Pinsker's inequality. The norm $\|\cdot\|_{\text{TV}}$ stands for the total variation norm.
Let $y$ represent an element of $(\mathcal{X} \times \mathcal{A})^{N+1}$ such that $y_i \in \mathcal{X} \times \mathcal{A}$ for all $1 \leq i \leq N$. Observe that
\begin{eqnarray*}
\Vert \mu_{1:N}-\mu'_{1:N}\Vert_1
&=&\sum_{y\in (\mathcal{X} \times \mathcal{A})^{N+1}}  \vert \mu_{1:N}(y)-\mu'_{1:N}(y)\vert \\
&\geq&
\sum_{y_n\in \mathcal{X} \times \mathcal{A}} \bigg\vert \sum_{y_s \in \mathcal{X} \times \mathcal{A}\,,\,s\neq n}
\big( \mu_{1:N}(y)-\mu'_{1:N}(y)\big) \bigg\vert
\\
&=&\sum_{y_n\in \mathcal{X} \times \mathcal{A}} \vert \mu_n(y_n)-\mu'_n(y_n)\vert \quad \textrm{for all} \ n\in \{1,\cdots, N\}.
\end{eqnarray*}
In particular,
\begin{eqnarray*}
        \Vert \mu_{1:N}-\mu'_{1:N}\Vert_1 &\geq& \sup_{1\leq n\leq N}\Vert \mu_n-\mu'_n\Vert_1\ .
\end{eqnarray*}
This implies that
\begin{eqnarray*}
\Gamma(\mu,\mu') &\geq &\frac{1}{2} \sup_{1\leq n\leq N} \Vert \mu_{n}-\mu'_{n}\Vert_1^2\ ,
\end{eqnarray*}
proving that $\Gamma$ is $1$-strongly convex with respect to the $\sup_{1 \leq n \leq N}\|\cdot\|_1$ norm.

\end{proof}

\subsection{Complements of the proof of Theorem~\ref{thm:convergence_rate}}\label{proof_conv_MD_MFC} 
\begin{proof}
Here we  prove that if $(f_n)_{1 \leq n \leq N}$ are convex and Lipschitz with respect to the L$1$-norm, then so is $F$.  
\textbf{Convexity:} $F$ is convex as the sum of convex functions.

\textbf{Lipschitz:} Let $\mu, \mu' \in (\mathcal{X} \times \mathcal{A})^N$. As $f_n$ is Lipschitz with respect to $\|\cdot\|_1$ with constant $l_n$, then $|f_n(\mu_n) - f_n(\mu'_n)| \leq l_n \|\mu_n - \mu_n\|_1$ for all $1 \leq n \leq N$. 
Therefore,
\begin{equation*}
    \begin{split}
        |F(\mu) - F(\mu')| &= \bigg| \sum_{n=1}^N f_n(\mu_n) - f_n(\mu'_n) \bigg| \\
        &\leq \sum_{n=1}^N |f_n(\mu_n) - f_n(\mu'_n)| \\
        &\leq \sum_{n=1}^N l_n \|\mu_n - \mu_n'\|_1 \\
        &\leq \bigg(\sum_{n=1}^N l_n^2\bigg)^{1/2} \bigg(\sum_{n=1}^N \|\mu_n - \mu'_n \|_1^2 \bigg)^{1/2} \\
        &\leq L \|\mu - \mu'\|_1,
    \end{split}
\end{equation*}
where we use Cauchy-Schwarz in the second to last inequality.
Therefore, $F$ is Lipschitz with respect to the L$1$-norm with constant ${L := \big(\sum_{n=1}^N l_n^2)^{1/2}.}$

\end{proof}

\section{Algorithms}\label{algo_appen}

\begin{algorithm}[H]
\caption{Fictitious play for MFG (FP)}\label{alg:FP}
\begin{algorithmic}
    \STATE {\bfseries Input:} number of iterations K, initial policy $\pi^0$.
    \STATE {\bfseries Initialization:} $\bar{\mu}^0 = \mu^{\pi^0}$ as in Definition~\ref{mu_induced_by_pi}.
    \FOR{$k= 0,...,K$}
    \STATE $\pi^{k+1} \in \argmax_{\pi} J(\pi, \bar{\mu}^k)$, best response against $\bar{\mu}^k$.
    \STATE $\bar{\mu}^{k+1} = \frac{1}{k+1} \mu^{\pi^{k+1}} + \frac{k}{k+1} \bar{\mu}^k$.
    \ENDFOR
\STATE {\bfseries Return:} $\bar{\mu}^K$ and $\bar{\pi}^K$ s.t. $\bar{\pi}^K_n(a|x) := \sum_{k=0}^K \frac{\rho^{\pi^k}_n(x) \pi_n^k(a|x)}{\sum_{k=0}^K \rho^{\pi^k}_n(x)}$, \big($\rho_n^{\pi^k}(x) := \sum_{a\in \mathcal{A}} \mu_n^{\pi^k}(x,a)$ for all $k \leq K$\big).
\end{algorithmic}
\end{algorithm}

\begin{algorithm}[H]
\caption{Frank Wolfe}\label{alg:FW}
\begin{algorithmic}
 \STATE {\bfseries Input:} number of iterations K, initial distribution $\mu^0$, sequence $(\eta_k)_k$.
    \FOR{$k= 0,...,K$}
    \STATE $\mu^k \in \argmin_{\mu \in \mathcal{M}} \left\langle \mu, \nabla F(\bar{\mu}^k) \right\rangle_{|\mathcal{X} \times \mathcal{A}|}$.
    \STATE $\bar{\mu}^{k+1} = (1 - \eta_{k+1}) \bar{\mu}^k + \eta_{k+1} \mu^k$.
    \ENDFOR
\STATE {\bfseries Return:} $\bar{\mu}^K$
\end{algorithmic}
\end{algorithm}

The Online Mirror Descent for MFG algorithm uses the regular state-value function (or $Q$-function) at each iteration. It's definition is given by
\begin{equation}\label{Q_func_mfg}
    \begin{split}
        Q_n^{\pi,\mu}(x,a) &:= \mathbb{E}_\pi\left[\sum_{i=n}^N r_i(x_i,a_i,\mu_i) \bigg| x_n = x, a_n = a\right].
    \end{split}
\end{equation}
Note that, considering an initial state-action distribution $\mu_0$, $J_{\mu_0}(\pi, \mu) = \mathbb{E}_{(x,a) \sim \mu_0}[Q^{\pi, \mu}_0(x,a)].$ Furthermore, $Q^{\pi,\mu}$ is the solution of the backward equation, for all $n < N$, $(x,a) \in\mathcal{X} \times\mathcal{A}$:
\begin{equation}\label{Q_func_bellman_mfg}
\begin{cases}
    Q^{\pi,\mu}_N(x,a) = r_N(x,a,\mu_N) \\
    \!\begin{aligned}
    Q_n^{\pi,\mu}&(x,a) = r_n(x,a,\mu_n) + \sum_{x'} p(x'|x,a) \sum_{a'} \pi_{n+1}(a'|x') Q_{n+1}^{\pi,\mu}(x',a').
    \end{aligned}
\end{cases}
\end{equation}

\begin{algorithm}[H]
\caption{OMD for MFG}\label{alg:OMD_MFG}
\begin{algorithmic}
\STATE {\bfseries Input:} number of iterations K, $\pi^0 \in (\Delta_\mathcal{A})^{\mathcal{X} \times N}$.
    \FOR{$k= 0,...,K$}
    \STATE $\mu^k := \mu^{\pi^k}$, as in Definition~\ref{mu_induced_by_pi}.
    \STATE $Q^k := Q^{\pi^k, \mu^k}$ as in Equation~\eqref{Q_func_bellman_mfg}.
    \STATE $\pi_n^{k+1}(\cdot | x) := \argmax_{\pi(\cdot|x) \in \Delta_\mathcal{A}} \langle Q_n^k(x, \cdot), \pi(\cdot|x) \rangle + \tau D \left(\pi(\cdot|x), \pi^k_n(\cdot|x)  \right)$, $\forall x \in \mathcal{X}, \forall n \leq N$.
    \ENDFOR
    \STATE  {\bfseries Return:} $\mu^K, \pi^K$
\end{algorithmic}
\end{algorithm}

\section{Water heater application}\label{water heater-behavior}
\subsection{Standard cycling behavior of one water heater}\label{nominal_behavior}
Let us consider a time window $[t_0, t_0+T]$, and consider a discretisation of the time such that $t_n = t_0 + n \delta_t$ for $n = 0,...,N$, and $\delta_t = T/N$ the time frequency. At each time step $t_n$ (that for short we call $n$), the state of a water heater is described by a variable $X_n = (m_n, \theta_n) \in \{0,1\} \times \mathbb{R}^+$, where $m_n$ indicates the operating state of the heater (ON if $1$, OFF if $0$), and $\theta_n$ represents the average temperature of the water in the tank. 

The evolution of the temperature in the next time step $t_{n+1}$ is given by $\theta_{n+1} = \bar{T}_{t_{n+1}}^{t_{n}, m_n, \theta_n}$, where $t \mapsto \bar{T}_{t}^{t_{n}, m_n, \theta_n}$ is the solution of the ordinary differential equation (ODE) in Equation~\eqref{temperature_ode} on the interval $[t_{n}, t_{n+1}]$. This ODE models the impact of the heat loss to the environment temperature ($T_\text{amb}$), the Joule effect (heating) and water drains (hot water being withdrawn from the tanks for showers, taps, etc),
\begin{equation}\label{temperature_ode}
    \begin{cases}
    \frac{dT(t)}{dt} = -\underbrace{\rho (T(t)-T_{\text{amb}})}_{\text{heat loss}} + \underbrace{\sigma m_n p_{\text{max}}}_{\text{Joule effect}} - \underbrace{\tau(T(t) - T_{\text{in}})f(t)}_{\text{water drain}} \\
    T(t_n) = \theta_n.
    \end{cases}
\end{equation}

The parameters $\rho, \sigma, \tau$ are technical parameters of the water heater, $p_{\text{max}}$ is the maximum power, $T_\text{in}$ denotes the temperature of the cold water entering the tank, and $f(t)$ denotes the drain function. 

The dynamics follow a cyclic ON/OFF decision rule with a deadband to ensure that the temperature is between a lower limit $T_\text{min}$ and an upper limit $T_\text{max}$. Thus, if the water heater is turned on, it heats water with the maximum capacity until its temperature exceeds $T_\text{max}$. Then, the heater turns off. The water temperature then decreases until it reaches $T_\text{min}$, then the heater turns on again and a new cycle begins. Therefore, the nominal dynamics at a discretized time is given by Equation~\eqref{nominal_dynamics} and is illustrated at Figure~\ref{fig:nominal_temp_example}.
\begin{equation}\label{nominal_dynamics}
    \begin{cases}
    \theta_{n+1} = \bar{T}_{t_{n+1}}^{t_{n}, m_n, \theta_n} \\
    m_{n+1} = \begin{cases}
    m_n, \quad \text{if} \; \theta_{n+1} \in [T_\text{min},T_\text{max}] \\
    0, \quad \text{if} \; \theta_{n+1} \geq T_\text{max} \\
    1, \quad \text{if} \; \theta_{n+1} \leq T_\text{min}.
    \end{cases}
    \end{cases}
\end{equation}
Note that assuming the temperature set is finite prevents us from using the ODE on Equation~\eqref{temperature_ode} to compute the evolution of the mean temperature. In addition, we also have trouble computing the drain function $f(t)$, which in practice is not deterministic. Instead, we adapt this ODE to simplify our system. We start by making an Euler discretization of the ODE. We define a sequence $(d_n)_n$ denoting the amount of draining in liters at each time step. To decide whether hot water is drawn at each time step, we also consider a sequence $(\epsilon_n)_n$ of independent random variables following Bernoulli's laws of parameters $(q_n)_n$ respectively. The interest of having different parameters for each time step is to take into account the moments of the day when people are more inclined to use hot water (for taking a shower, doing the dishes, etc.). Assuming the existence of an independent water discharge at each time step is justified by assuming that the time frequency $\delta_t$ is large enough to contain all the time when hot water will be drawn from the water heater tank for a single use. In the interest of more realistic dynamics, we intend to weaken this assumption in future work. Therefore, we define
 \begin{equation}\label{appr_temperature_dynamics}
 \begin{split}
     \theta'_{n+1} = \theta_n + \delta_t\big( -\rho\left(\theta_n - T_\text{amb}\right) + \sigma m_n p_\text{max} - \epsilon_n \tau \left(\theta_n - T_\text{in}\right) d_n\big).
     \end{split}
 \end{equation}
 
 To tackle the finite-temperature state space problem, we assume that the space of possible temperatures $\Theta$ contains only integers from $T_\text{amb}$ (the room temperature) up to $T_\text{max}$, assuming that $T_\text{amb} < T_\text{min}$ (it is reasonable to assume that the ambient temperature is below the minimum temperature accepted for the heater). Given the dynamics of the operating state, $\theta_{n+1}$ never exceeds $T_\text{max}$ (the heater turns off when it reaches $T_\text{max}$ and when it is turned off, its temperature only decreases). On the other hand, drain may allow a temperature to be lower than $T_\text{min}$, but we assume that $T_\text{amb}$ is small enough that the mean temperature is never lower than it. Therefore, we can take $\theta_{n+1} = \text{Round}(\theta_{n+1}')$, where
 \begin{equation*}
     \text{Round}(\theta) = \begin{cases}
     \lfloor \theta \rfloor, \quad \text{if } B(\theta) = 0 \\
     \lceil \theta \rceil, \quad \text{if } B(\theta) = 1,
     \end{cases}
 \end{equation*}
 and $B(\theta)$ is a random variable following a Bernoulli of parameter $\theta - \lfloor \theta \rfloor$. Thus, the closer $\theta$ is to its smallest nearest integer, the greater the probability that we approximate $\theta$ by it, and vice-versa. We perform stochastic rounding instead of deterministic to have an unbiased temperature estimator, i.e. $\mathbb{E}[\theta_{n+1}] = \theta_{n+1}'$.

 \subsection{Complements of the proof of Theorem~\ref{thm:convergence_rate} for the DSM problem}\label{proof_conv_heater}
 
We show that the cost function considered in Problem~\eqref{main_optimisation_problem} concerning the water heater optimisation problem is convex and Lipschitz with respect to the $L_1$ norm $\|\cdot\|_1$.

    \paragraph{Convexity} for all $n\leq N$, each $f_n$ is given by
    \begin{equation*}
        f_n(\mu_n) = \left( \sum_{x,a} \mu_n(x,a) \varphi(x) - \gamma_n \right)^2.
    \end{equation*}
    Let $g$ be a real function such that $g(x) = (x - \gamma_n)^2$. The function $g$ is convex and non-decreasing on $\mathbb{R}_+$. 
    
    Let $h: \mathbb{R}^{|\mathcal{X} \times \mathcal{A}|} \rightarrow \mathbb{R}$, such that ${h(\mu_n) = \sum_{x,a} \mu_n(x,a) \varphi(x)}$. Note that ${\frac{\partial h}{\partial \mu_n(x,a)}(\mu_n) = \varphi(x)}$. Thus, for any $\mu_n, \mu_n'\in \Delta_{\mathcal{X} \times \mathcal{A}}$,
    \begin{equation*}
        h(\mu_n) - h(\mu_n') = \left( \sum_{x,a} \big(\mu_n(x,a) - \mu_n'(x,a) \big) \varphi(x) \right) = \langle \nabla h(\mu_n'), \mu_n - \mu_n' \rangle,
    \end{equation*}
    therefore, the function $h$ is also convex.
    As $f_n(\mu_n) = g(h(\mu_n))$, then $f_n$ is convex as $g$ and $h$ are convex, and $g$ is non decreasing in a univariate domain \citep{boyd}.

    \paragraph{Lipschitz} As $f_n$ is convex for all $1 \leq n \leq N$, to show that it is Lipschitz with respect to the $\|\cdot\|_1$ norm, it suffices to show that the sup-norm $\|\cdot\|_\infty$ of $\nabla f_n$ is bounded (the sup-norm is the dual norm of the $L_1$ norm). This result can be found in Lemma $2.6$ of \citet{shavel_opt}. 

    For any $\mu_n \in \Delta_{\mathcal{X} \times \mathcal{A}}$, 
    \begin{equation*}
        \begin{split}
            \|\nabla f_n(\mu)\| &= \sup_{ (x,a) \in \mathcal{X} \times \mathcal{A}} |\nabla f_n(\mu_n)(x,a)| \\
            &= 2 \sup_{ (x,a) \in \mathcal{X} \times \mathcal{A}}|\mu_n(\varphi) - \gamma_n| |\varphi(x)| \\
            &= 2 \sup_{ (x,a) \in \mathcal{X} \times \mathcal{A}} \bigg| \sum_{x',a'} \mu_n(x',a') \varphi(x') - \gamma_n \bigg| |\varphi(x)| \\
             &= 2 \sup_{ x \in \mathcal{X} } | \langle \rho_n, \varphi \rangle | |\varphi(x)| \\
            & \leq 2 \| \varphi \|_\infty^2.
        \end{split}
    \end{equation*}
    Thus, $f_n$ is Lipschitz with respect to the $L_1$ norm with Lipschitz constant $l_n = 2 \| \varphi \|_\infty^2$. In our particular case $\varphi$ is bounded by $1$ (see its definition in Equation~\eqref{varphi_func}), hence $l_n = 2$ for all $1 \leq n \leq N$. 

 \subsection{Simulation of the nominal behavior of a water heater}\label{water_heater_sim}
 
 Here we explain in details how the nominal dynamics are simulated in order to obtain the results in Section~\ref{results}.
 
 To simulate the nominal dynamics we use the nominal model presented in Equation~\eqref{nominal_dynamics} with the average temperature evolution introduced in Equation~\eqref{appr_temperature_dynamics}. To compute the sequences $(d_n)_n$ and $(q_n)_n$ regarding the amount of draining in liters and the probability of having a water withdrawal for each time step, respectively, we use data from the SMACH (\textit{Simulation Multi-Agents des Comportements Humains}) platform \citep{smach}, which simulates power consumption of people in their homes separated by appliance. The data we use simulates the consumption of $5132$ water heaters at a time step of one minute over a week in the summer of $2018$. 
 
  Since we want a time step large enough to contain all the time that hot water will be drawn from the water heater tank for a single use, we take $\delta_t = 10$ minutes instead of one minute (as initially provided by the data). Therefore we transform the data to contain for each water heater the average discharge over each $10$ minute interval. To compute $d_n$, we take the average discharge in liters over all water heaters with a water withdrawal during this time step. To calculate $(q_n)_n$, we calculate the percentage of water heaters with a water withdrawal over the entire population for each time step. The values of the parameters $\rho, \sigma, \tau$ and $p_\text{max}$ are computed in Equation~\eqref{parameters} using the variables introduced in Tables~\ref{table:water_heater_param} and~\ref{table:other_param}. We take $T_\text{min} = 50^\circ C$, $T_\text{max}=65^\circ C$, $T_\text{amb} = 25^\circ C$ and $T_\text{in} = 18^\circ C$.

\begin{table}[h!]
\caption{Water heater intrinsic parameters.}
\label{table:water_heater_param}
\vskip 0.15in
\begin{center}
\begin{small}
\begin{tabular}{cc}
\toprule
 Volume & $0.2$m$^3$ \\ 
 Height & $1.37$m\\
 EI (thickness of isolation) & $\frac{0.035}{4}$m \\
 $p_\text{max}$ & $3600 * 2200$W (in one hour)  \\
\bottomrule
\end{tabular}
\end{small}
\end{center}
\vskip -0.1in
\end{table}

\begin{table}[h!]
\caption{Other parameters specifications to compute Equation~\ref{parameters}.}
\label{table:other_param}
\vskip 0.15in
\begin{center}
\begin{small}
\begin{tabular}{cc}
\toprule
 denWater (water density) & $1000$ kg m$^{-3}$\\ 
 capWater (water capacity) & $4185$ J kg$^{-1}$ K$^{-1}$ \\
 CI (heat conductivity) & $0.033$ W/(m K) \\
 coefLoss (loss coeff.) & $\frac{\text{CI}}{\text{EI}} * 2 * 3.14 \sqrt{\frac{\text{vol} * 3.14}{\text{height}}}$  \\
\bottomrule
\end{tabular}
\end{small}
\end{center}
\vskip -0.1in
\end{table}

\begin{equation}\label{parameters}
\begin{split}
    \rho &= \frac{\text{coefLoss} * 3600}{\text{capWater} * \text{denWater} * \text{vol/height}} \quad \text{(fraction of heat loss by hour)} \\
    \sigma &= (\text{vol} * \text{denWater} * \text{capWater})^{-1} \\
    \tau &= (\text{vol} * \text{denWater})^{-1}.
    \end{split}
\end{equation}


\section{Potential games discussion}\label{potential_games_discussion}

In Subsection~\ref{potential_games} we mention an equivalence between the control Problem~\eqref{main_optimisation_problem} and a game problem in order to be able to compare Algorithm~\ref{alg:MD} with learning algorithms for MFG in the literature. For this, we use results similar to those of \citet{concave_utility} and we refer to it for further definitions on a MFG structure and the notion of Nash equilibrium (NE). 

In a mean field game problem, the goal of a representative player is to find a sequence of policies $\pi$ that maximises the expected sum of rewards when the population distributions sequence is given by $\mu := (\mu_n)_{1 \leq n \leq N}$ and the initial state-action pair is sampled from a fixed distribution $\mu_0$,
\begin{align}\label{J_func}
    \begin{split}
        J_{\mu_0}(\pi, \mu) &:= \mathbb{E}_\pi \left[ \sum_{n=1}^N r_n(x_n, a_n, \mu_n) \right].
    \end{split}
\end{align}
Let us define a game with the same transition probability $p$, and with reward defined as 
\begin{equation}\label{reward_def}
    r_n(x_n, a_n, \mu_n) := - \nabla f_n(\mu_n)(x_n, a_n)
\end{equation}
for all $(x_n,a_n, \mu_n) \in \mathcal{X} \times \mathcal{A} \times \Delta_{\mathcal{X} \times \mathcal{A}}$.

\begin{proposition}\label{opt_NE_mono}
The strategy $\pi^*$ is a minimizer of Problem~\eqref{main_optimisation_problem} if and only if,  $(\mu^{\pi^*}, \pi^*)$ is a NE of the MFG defined with reward as in Equation~\eqref{reward_def}. Furthermore, this game is monotone (and strictly monotone if $f_n$ is strictly convex for all $1 \leq n \leq N$. See Definition~\ref{mono_state-action}).
\end{proposition}

This Proposition connects the optimality conditions of Problem~\eqref{main_optimisation_problem} and a NE, and shows that convexity and monotonicity are equivalent. If the optimization problem is (strictly) convex, the (unique) existence of an optimizer implies the (unique) existence of a NE. Thus, the notion of monotonicity when the reward depends on the state-action distribution provides the (unique) existence of a NE in the case of a potential game. 

\begin{proof}
The convexity of each $f_n$ for $1 \leq n \leq N$, and the convexity of the set $\mathcal{M}_{\mu_0}$ ensure the existence of a minimizer of Problem~\eqref{main_optimisation_problem} satisfying the optimality conditions. Also, Proposition~\ref{opt_mu_equal_pi} shows, for a fixed initial state-action distribution $\mu_0$, a surjection between the sets $(\Delta_\mathcal{A})^{\mathcal{X} \times N}$ and $\mathcal{M}_{\mu_0}$.

Let $(\mu^*, \pi^*)$, where $\mu^* = \mu^{\pi^*}$, be a Nash equilibrium. 

By definition, a Nash equilibrium  $(\mu^*, \pi^*)$ satisfies $\pi^* = \argmax_\pi J(\pi, \mu^*)$. In other words, 
\begin{equation}\label{J_inequality}
    J(\pi^*, \mu^{\pi^*}) \geq J( \pi, \mu^{\pi^*}) \quad \forall \pi \in (\Delta_\mathcal{A})^{\mathcal{X} \times N}.
\end{equation}
Expanding the terms of the sum of expected rewards and using the definition of reward in a potential game, we obtain that
\begin{align*}\label{J_as_sum_fn}
     J(\pi, \mu^{\pi^*}) &=  \mathbb{E}_\pi \left[ \sum_{n=1}^N r_n(x_n, a_n, \mu_n^{\pi^*}) \right] \\
     &= \sum_{n=1}^N \sum_{x \in\mathcal{X}, a \in A}r_n(x,a,\mu_n^{\pi^*})\mu_n^\pi(x,a)  \\
     &= \sum_{n=1}^N -\left\langle \nabla f_n(\mu_n^{\pi^*}), \mu_n^\pi \right\rangle.
\end{align*}
Similarly,
\begin{equation*}
    J(\pi^*, \mu^{\pi^*}) = \sum_{n=1}^N -\left\langle \nabla f_n(\mu_n^{\pi^*}), \mu_n^{\pi^*} \right\rangle.
\end{equation*}
Thus, the Nash equilibrium condition in Inequality~\eqref{J_inequality} entails
\begin{equation}\label{convexity_min}
    \sum_{n=1}^N  \big\langle \nabla f_n(\mu^{\pi^*}_n), \mu_n^{\pi^*} - \mu_n^{\pi} \big\rangle \leq 0.
\end{equation}

As $f_n$ is convex for all $n \in \{1,...,N\}$, this yields
\begin{equation}\label{minimizer}
    \sum_{n=1}^N f_n(\mu^{\pi^*}_n) - f_n(\mu^{\pi}_n) \leq 0.
\end{equation}

Thus, $\pi^*$ satisfies the optimality conditions of Problem~\eqref{main_optimisation_problem}. We then proved that if $(\pi^*, \mu^*)$ is a NE with $\mu^* = \mu^{\pi^*}$, then $\pi^*$ is an optimum of Problem~\eqref{main_optimisation_problem}.

On the other way around, if $\pi^*$ is a minimizer of Problem~\eqref{main_optimisation_problem} then it satisfies Inequality~\eqref{minimizer} for all $\pi \in (\Delta_\mathcal{A})^{\mathcal{X} \times N}$. Again, by convexity of $(f_n)_{1 \leq n \leq N}$, $\pi^*$ also satisfies Inequality~\eqref{convexity_min}. Following the same calculations backwards, we obtain that $\pi^*$ then satisfies Inequality~\eqref{J_inequality}, and by definition is then a NE. This concludes the first part of the proof.

The second part concerns the monotonicity of the game, defined below for the mean field game framework.

\begin{definition}[Monotonicity]\label{mono_state-action}
 According to \citet{MFG_original}, a game where the reward depends on the population's state-action distribution (sometimes called ``extended MFG'' in the literature, see \citet{extended_MFG}) is (strictly) monotone if for any state-action distributions $\nu, \nu' \in \Delta_{\mathcal{X} \times\mathcal{A}}$ with $\nu \neq \nu'$, 
\begin{equation*}
    \int_{\mathcal{X},\mathcal{A}} [r(x, a, \nu) - r(x,a, \nu')] d(\nu - \nu')(x,a) \leq 0, \; (< 0). 
\end{equation*}
\end{definition}

Back to the proof, consider $\mu, \mu'$ two distributions over $\mathcal{X} \times\mathcal{A}$. As the result should be true to all $n$, we omit the time step index for the computations. Recall that the reward is of the form $r(x,a,\mu) = -\nabla f(\mu) (x,a)$ for all $(x,a) \in \mathcal{X} \times \mathcal{A}$, with $f$ a convex function. Then,  
\begin{equation*}
\begin{split}
    \int_{\mathcal{X} \times \mathcal{A}} [r(x,a, \mu) - r(x,a, \mu')] d(\mu - \mu')(x,a) 
    =&   \int_{\mathcal{X} \times \mathcal{A}} [\nabla f(\mu')(x,a)  - \nabla f(\mu)(x,a)] d(\mu - \mu')(x,a) \\
    = &  \left\langle \nabla f(\mu') - \nabla f(\mu), \mu - \mu' \right\rangle \leq 0
    \end{split}
\end{equation*}
where the last inequality comes from the convexity of $f$.
\end{proof}

\end{document}

%% file: graphs/temp_nominal.tex
\begin{tikzpicture}[scale=0.5]
\node at (-1,-0.5) (1) {water draining};
\draw [line width=0.25mm, black,->] (1) to [out=20,in=170] (3,0.5);
\node at (-2.2,5) (2) {heating};
\node at (-2.2,4.4) (2) {(ON)};
\draw [line width=0.25mm, black,->] (2) to [out=-5,in=-170] (2,4.4);
\node at (8.2,4) (3) {heat loss};
\node at (8.1,3.4) (2) {(OFF)};
\draw [line width=0.25mm, black,->] (3) to [out=-160,in=-30] (4.9,4.1); 

\begin{axis}[
legend cell align={left},
legend style={
  fill opacity=0.8,
  draw opacity=1,
  text opacity=1,
  at={(1.03,0.0)},
  anchor=south east,
  draw=white!80!black
},
tick align=outside,
tick pos=left,
x grid style={white!69.0196078431373!black},
xticklabels={,,},
xmin=-3.59166666666667, xmax=75.425,
xlabel={Time $t$},
y grid style={white!69.0196078431373!black},
ylabel={Temperature $(^\circ C)$},
ymin=45.557768437505, ymax=67.1877260530458,
ytick style={color=black}
]
\addplot [semithick, blue, forget plot]
table {%
0 60
0.166666666666667 59.8244625340118
0.333333333333333 59.649805450937
0.5 59.4760243353415
0.666666666666667 59.3031147939366
0.833333333333333 59.1310724554672
1 58.9598929706018
1.16666666666667 58.7895720118221
1.33333333333333 58.6201052733142
1.5 58.4514884708589
1.66666666666667 58.2837173417244
1.83333333333333 58.1167876445577
2 57.9506951592779
2.16666666666667 57.7854356869692
2.33333333333333 57.6210050497746
2.5 57.4573990907911
2.66666666666667 57.2946136739635
2.83333333333333 57.1326446839808
3 56.9714880261715
3.16666666666667 56.8111396264005
3.33333333333333 56.6515954309658
3.5 56.4928514064962
3.66666666666667 56.3349035398494
3.83333333333333 56.1777478380104
4 56.0213803279903
4.16666666666667 55.8657970567266
4.33333333333333 55.7109940909825
4.5 55.556967517248
4.66666666666667 55.4037134416405
4.83333333333333 55.251227989807
5 55.0995073068252
5.16666666666667 54.9485475571071
5.33333333333333 54.7983449243014
5.5 54.648895611197
5.66666666666667 54.5001958396274
5.83333333333333 54.3522418503748
6 54.2050299030751
6.16666666666667 54.0585562761239
6.33333333333333 53.9128172665814
6.5 53.7678091900799
6.66666666666667 53.6235283807297
6.83333333333333 53.479971191027
7 53.3371339917616
7.16666666666667 47.6937963088982
7.33333333333333 49.1570397694346
7.5 50.6129445428513
7.66666666666667 52.0615474352781
7.83333333333333 53.502885068249
8 54.9369938796277
8.16666666666667 56.3639101245293
8.33333333333333 57.7836698762363
8.5 59.1963090271112
8.66666666666667 60.6018632895032
8.83333333333333 62.0003681966516
9 59.720368554133
9.16666666666667 61.1232944713421
9.33333333333333 62.5191842153948
9.5 63.9080730752054
9.66666666666667 61.4453356827021
9.83333333333333 62.839610275378
10 64.2268920843714
10.1666666666667 65.607216180983
10.3333333333333 65.4035565287189
10.5 65.2009183021486
10.6666666666667 64.9992963784584
10.8333333333333 64.7986856605274
11 64.5990810767987
11.1666666666667 64.4004775811507
11.3333333333333 64.2028701527701
11.5 64.006253796025
11.6666666666667 63.8106235403379
11.8333333333333 63.6159744400611
12 63.4223015743504
12.1666666666667 63.229600047042
12.3333333333333 63.0378649865277
12.5 62.8470915456325
12.6666666666667 62.6572749014912
12.8333333333333 61.24270410489
13 61.0609340352005
13.1666666666667 60.8800756071518
13.3333333333333 60.7001242485363
13.5 60.5210754100777
13.6666666666667 60.342924565316
13.8333333333333 60.1656672104929
14 59.9892988644383
14.1666666666667 59.8138150684564
14.3333333333333 59.6392113862136
14.5 59.4654834036261
14.6666666666667 59.2926267287481
14.8333333333333 59.1206369916612
15 58.9495098443635
15.1666666666667 58.77924096066
15.3333333333333 58.6098260360533
15.5 58.4412607876341
15.6666666666667 58.2735409539737
15.8333333333333 58.1066622950162
16 57.9406205919707
16.1666666666667 57.7754116472052
16.3333333333333 57.6110312841404
16.5 57.4474753471438
16.6666666666667 57.2847397014252
16.8333333333333 57.1228202329316
17 56.9617128482436
17.1666666666667 56.8014134744715
17.3333333333333 56.6419180591529
17.5 56.4832225701498
17.6666666666667 56.3253229955469
17.8333333333333 53.6593539278871
18 53.5156170603054
18.1666666666667 53.3726010843095
18.3333333333333 53.2303023843714
18.5 53.0887173630962
18.6666666666667 52.9478424411312
18.8333333333333 52.8076740570754
19 52.6682086673896
19.1666666666667 52.5294427463066
19.3333333333333 52.3913727857422
19.5 52.2539952952066
19.6666666666667 52.1173068017159
19.8333333333333 51.9813038497046
20 51.8459830009379
20.1666666666667 51.7113408344252
20.3333333333333 51.5773739463332
20.5 51.4440789499001
20.6666666666667 51.3114524753499
20.8333333333333 49.0201314693941
21 50.4767228866983
21.1666666666667 51.9260089792489
21.3333333333333 53.3680263858517
21.5 54.8028115615554
21.6666666666667 56.230400778574
21.8333333333333 57.6508301272032
22 59.0641355167331
22.1666666666667 60.4703526763561
22.3333333333333 61.8695171560696
22.5 63.2616643275754
22.6666666666667 64.6468293851736
22.8333333333333 66.0250473466521
23 65.8192921222728
23.1666666666667 65.6145688336279
23.3333333333333 65.410872305192
23.5 65.208197387397
23.6666666666667 65.0065389565017
23.8333333333333 64.8058919144622
24 64.6062511888032
24.1666666666667 64.4076117324897
24.3333333333333 64.2099685237993
24.5 64.0133165661952
24.6666666666667 63.8176508882002
24.8333333333333 63.6229665432708
25 63.4292586096722
25.1666666666667 63.2365221903538
25.3333333333333 63.0447524128253
25.5 62.8539444290341
25.6666666666667 62.6640934152418
25.8333333333333 62.4751945719031
26 62.2872431235441
26.1666666666667 62.1002343186415
26.3333333333333 61.9141634295026
26.5 61.7290257521458
26.6666666666667 61.5448166061817
26.8333333333333 61.3615313346945
27 61.1791653041247
27.1666666666667 60.9977139041517
27.3333333333333 60.8171725475774
27.5 60.6375366702099
27.6666666666667 60.4588017307485
27.8333333333333 60.2809632106687
28 60.1040166141081
28.1666666666667 59.9279574677525
28.3333333333333 59.752781320723
28.5 59.5784837444636
28.6666666666667 59.4050603326289
28.8333333333333 59.2325067009729
29 59.0608184872382
29.1666666666667 58.8899913510459
29.3333333333333 58.7200209737852
29.5 58.550903058505
29.6666666666667 58.3826333298049
29.8333333333333 58.215207533727
30 58.0486214376486
30.1666666666667 57.8828708301753
30.3333333333333 52.442866651631
30.5 52.3052309010174
30.6666666666667 52.1682854427149
30.8333333333333 52.0320268146619
31 51.8964515721599
31.1666666666667 51.761556287787
31.3333333333333 51.6273375513111
31.5 46.7292273700611
31.6666666666667 48.1973084873773
31.8333333333333 49.6580266550021
32 51.1114188007506
32.1666666666667 52.557521667232
32.3333333333333 53.996371812778
32.5 55.4280056123671
32.6666666666667 56.8524592585449
32.8333333333333 58.2697687623378
33 59.6799699541645
33.1666666666667 61.0830984847413
33.3333333333333 62.4791898259834
33.5 63.8682792719019
33.6666666666667 65.2504019394955
33.8333333333333 65.0485318377394
34 64.8476741864439
34.1666666666667 64.6478239078093
34.3333333333333 64.448975949503
34.5 64.2511252845314
34.6666666666667 64.0542669111134
34.8333333333333 63.8583958525534
35 63.6635071571156
35.1666666666667 63.4695958978991
35.3333333333333 63.2766571727133
35.5 63.0846861039534
35.6666666666667 61.1932070393912
35.8333333333333 61.0116852151146
36 60.8310737874396
36.1666666666667 60.651368190403
36.3333333333333 60.4725638809414
36.5 60.2946563387767
36.6666666666667 58.8777599852205
36.8333333333333 58.7078509526158
37 58.5387940743265
37.1666666666667 58.3705850764949
37.3333333333333 58.2032197066986
37.5 58.0366937338425
37.6666666666667 57.8710029480521
37.8333333333333 57.7061431605666
38 55.8490150295044
38.1666666666667 55.6942962316041
38.3333333333333 55.5403534035817
38.5 55.3871826536718
38.6666666666667 55.2347801096276
38.8333333333333 55.0831419186233
39 54.9322642471559
39.1666666666667 54.7821432809492
39.3333333333333 54.6327752248569
39.5 54.4841563027663
39.6666666666667 54.3362827575039
39.8333333333333 54.1891508507392
40 54.0427568628911
40.1666666666667 53.8970970930333
40.3333333333333 53.7521678588011
40.5 53.6079654962982
40.6666666666667 53.464486360004
40.8333333333333 53.3217268226815
41 53.1796832752854
41.1666666666667 53.0383521268714
41.3333333333333 52.8977298045046
41.5 52.7578127531702
41.6666666666667 52.6185974356827
41.8333333333333 52.4800803325967
42 52.3422579421186
42.1666666666667 52.2051267800171
42.3333333333333 52.0686833795356
42.5 51.9329242913048
42.6666666666667 51.7978460832549
42.8333333333333 51.6634453405293
43 51.5297186653982
43.1666666666667 51.3966626771724
43.3333333333333 51.2642740121183
43.5 51.1325493233724
43.6666666666667 51.0014852808571
43.8333333333333 49.045362846264
44 50.5018277192265
44.1666666666667 51.9509879021005
44.3333333333333 51.6901817301783
44.5 53.1333818944299
44.6666666666667 54.5693438958317
44.8333333333333 55.9981040363483
45 54.7338710402997
45.1666666666667 56.161806018587
45.3333333333333 57.582579394369
45.5 56.4163296951948
45.6666666666667 57.8358265440846
45.8333333333333 59.2482041106936
46 60.6534981007579
46.1666666666667 62.0517440409369
46.3333333333333 63.4429772797112
46.5 64.8272329882759
46.6666666666667 66.2045461614304
46.8333333333333 65.9978906865626
47 65.7922716625061
47.1666666666667 65.5876838910908
47.3333333333333 65.3841222002173
47.5 65.1815814437264
47.6666666666667 64.9800565012684
47.8333333333333 64.7795422781742
48 64.5800337053262
48.1666666666667 64.3815257390303
48.3333333333333 64.1840133608881
48.5 63.9874915776704
48.6666666666667 63.791955421191
48.8333333333333 63.5973999481807
49 63.4038202401625
49.1666666666667 63.2112114033276
49.3333333333333 63.019568568411
49.5 62.8288868905693
49.6666666666667 62.6391615492572
49.8333333333333 62.4503877481064
50 62.2625607148042
50.1666666666667 62.0756757009726
50.3333333333333 61.8897279820484
50.5 61.7047128571638
50.6666666666667 61.5206256490275
50.8333333333333 61.3374617038066
51 61.1552163910087
51.1666666666667 60.9738851033649
51.3333333333333 60.7934632567137
51.5 60.6139462898844
51.6666666666667 60.4353296645826
51.8333333333333 60.2576088652747
52 60.0807793990743
52.1666666666667 59.9048367956285
52.3333333333333 59.7297766070047
52.5 59.5555944075782
52.6666666666667 59.3822857939204
52.8333333333333 59.2098463846877
53 59.0382718205102
53.1666666666667 58.867557763882
53.3333333333333 58.6976998990511
53.5 58.5286939319109
53.6666666666667 58.3605355898908
53.8333333333333 58.1932206218492
54 58.0267447979651
54.1666666666667 57.8611039096316
54.3333333333333 57.6962937693496
54.5 57.5323102106217
54.6666666666667 57.3691490878471
54.8333333333333 57.2068062762164
55 57.0452776716079
55.1666666666667 56.8845591904833
55.3333333333333 56.7246467697848
55.5 56.5655363668323
55.6666666666667 51.5890854701046
55.8333333333333 51.4557317362085
56 51.3230468187846
56.1666666666667 51.1910273634801
56.3333333333333 46.5409483291205
56.5 48.0099737328866
56.6666666666667 49.4716314510284
56.8333333333333 50.9259584351134
57 52.3729914513841
57.1666666666667 53.8127670816868
57.3333333333333 55.2453217243972
57.5 56.6706915953397
57.6666666666667 58.088912728704
57.8333333333333 59.500020977955
58 60.9040520167401
58.1666666666667 62.3010413397905
58.3333333333333 63.6910242638185
58.5 65.0740359284108
58.6666666666667 64.8730503650167
58.8333333333333 64.6730728158083
59 64.4740982252354
59.1666666666667 64.2761215631033
59.3333333333333 64.0791378244453
59.5 63.8831420293966
59.6666666666667 63.6881292230681
59.8333333333333 63.4940944754212
60 63.3010328811434
60.1666666666667 63.1089395595238
60.3333333333333 62.9178096543303
60.5 62.7276383336862
60.6666666666667 62.5384207899486
60.8333333333333 62.3501522395868
61 62.1628279230608
61.1666666666667 61.9764431047017
61.3333333333333 61.7909930725916
61.5 61.6064731384445
61.6666666666667 59.5846966886059
61.8333333333333 59.4112421166434
62 59.2386574811388
62.1666666666667 59.0669384190509
62.3333333333333 58.8960805892208
62.5 58.7260796722618
62.6666666666667 57.0694775049047
62.8333333333333 56.9086376529968
63 56.7486044702327
63.1666666666667 56.5893739108795
63.3333333333333 56.4309419494947
63.5 56.2733045808251
63.6666666666667 56.1164578197053
63.8333333333333 55.9603977009568
64 55.8051202792878
64.1666666666667 55.6506216291939
64.3333333333333 55.4968978448582
64.5 55.3439450400529
64.6666666666667 55.191759348041
64.8333333333333 55.0403369214788
65 54.889673932318
65.1666666666667 54.7397665717097
65.3333333333333 54.5906110499075
65.5 54.442203596172
65.6666666666667 54.2945404586754
65.8333333333333 54.1476179044068
66 54.0014322190773
66.1666666666667 53.855979707027
66.3333333333333 51.1570717494963
66.5 51.0258847182798
66.6666666666667 50.8953556367499
66.8333333333333 50.7654812050545
67 50.6362581398917
67.1666666666667 50.5076831744261
67.3333333333333 50.379753058207
67.5 50.2524645570858
67.6666666666667 50.1258144531341
67.8333333333333 49.9997995445629
68 51.4514775775407
68.1666666666667 52.8958749281373
68.3333333333333 54.3330281115682
68.5 55.7629734599117
68.6666666666667 57.1857471230281
68.8333333333333 56.8332276640109
69 58.2506336211001
69.1666666666667 59.6609307824755
69.3333333333333 61.0641548012795
69.5 62.4603411518414
69.6666666666667 63.849525130574
69.8333333333333 65.2317418568661
70 65.0299653420706
70.1666666666667 64.8292008083637
70.3333333333333 64.6294431802999
70.5 64.4306874078887
70.6666666666667 64.2329284664672
70.8333333333333 64.0361613565729
71 63.8403811038175
71.1666666666667 63.645582758761
71.3333333333333 63.4517613967865
71.5 63.2589121179758
71.6666666666667 63.0670300469857
71.8333333333333 62.8761103329244
};
\addplot [semithick, red]
table {%
-3.59166666666667 65
75.425 65
};
\addlegendentry{$T_\text{max}$}
\addplot [semithick, green!50!black]
table {%
-3.59166666666667 50
75.425 50
};
\addlegendentry{$T_\text{min}$}
\end{axis}

\end{tikzpicture}

%% file: graphs/drain_onehour.tex
\begin{tikzpicture}[scale=0.50]

\definecolor{darkgray176}{RGB}{176,176,176}
\definecolor{steelblue31119180}{RGB}{31,119,180}

\begin{axis}[
tick align=outside,
tick pos=left,
x grid style={darkgray176},
xmin=-1.19166666666667, xmax=24.5,
xlabel=Time (hours),
ylabel=Average drain ($J$),
xtick style={color=black},
y grid style={darkgray176},
ymin=-52721.7369893467, ymax=1107156.47677628,
ytick style={color=black}
]
\addplot [thick, steelblue31119180]
table {%
0 37215.70446825
0.166666666666667 25203.8334724138
0.333333333333333 22187.08998
0.5 10075.8323655
0.666666666666667 8796.1480875
0.833333333333333 7403.27701692857
1 11845.3851225
1.16666666666667 8747.9433324
1.33333333333333 11204.0121942
1.5 9522.5993874
1.66666666666667 6999.130675125
1.83333333333333 6605.680480875
2 13018.564917
2.16666666666667 11760.77907
2.33333333333333 856.327167
2.5 9230.16110528572
2.66666666666667 3124.054895625
2.83333333333333 11062.5035023125
3 11212.49112975
3.16666666666667 7371.98320114286
3.33333333333333 15767.9350600909
3.5 7282.590525
3.66666666666667 8977.302473625
3.83333333333333 28594.670327413
4 39006.9128007
4.16666666666667 29870.1486969545
4.33333333333333 51932.6603489118
4.5 32990.5566677647
4.66666666666667 38648.624265
4.83333333333333 128673.689131454
5 106377.410801106
5.16666666666667 94739.5444814685
5.33333333333333 161395.363136636
5.5 153175.39425
5.66666666666667 156795.634557341
5.83333333333333 332765.323655412
6 481849.711002988
6.16666666666667 434018.685308357
6.33333333333333 578042.188139064
6.5 753640.772606051
6.66666666666667 751263.439137814
6.83333333333333 914229.777648678
7 1037673.74831813
7.16666666666667 778115.596129253
7.33333333333333 691110.100749275
7.5 663377.027225308
7.66666666666667 550746.936361392
7.83333333333333 505095.127105537
8 499600.031511719
8.16666666666667 437300.686131364
8.33333333333333 288666.915448942
8.5 354728.075169865
8.66666666666667 329085.629899291
8.83333333333333 279357.151771295
9 350281.252102019
9.16666666666667 241434.06055425
9.33333333333333 174655.574886724
9.5 146320.134783231
9.66666666666667 155352.046490883
9.83333333333333 128158.187819206
10 124416.997431995
10.1666666666667 124191.178889063
10.3333333333333 109145.258328432
10.5 85500.7301171369
10.6666666666667 76646.7208669602
10.8333333333333 57042.3221585085
11 77947.0678773621
11.1666666666667 80179.2290940882
11.3333333333333 84412.1132438985
11.5 65471.5267064384
11.6666666666667 56760.9625265502
11.8333333333333 48956.1337559051
12 54312.2451862942
12.1666666666667 70494.4907111159
12.3333333333333 83751.3110183101
12.5 91013.3467686486
12.6666666666667 94174.5429610368
12.8333333333333 83321.4475610742
13 88889.112751383
13.1666666666667 98238.711699169
13.3333333333333 83527.2319956135
13.5 97802.2241744013
13.6666666666667 91500.6274020001
13.8333333333333 52433.4644476119
14 64585.9100367303
14.1666666666667 54937.7615987331
14.3333333333333 30664.4636682478
14.5 38996.68374
14.6666666666667 32525.38665
14.8333333333333 25385.1416752263
15 38700.9332031
15.1666666666667 31173.6264500454
15.3333333333333 35645.5330177215
15.5 29837.0668041325
15.6666666666667 36829.0110505109
15.8333333333333 30054.3611355
16 57851.5965897053
16.1666666666667 34294.1028409059
16.3333333333333 25376.6867910577
16.5 49224.826567101
16.6666666666667 53358.3895092245
16.8333333333333 35587.0021497187
17 47443.4728369286
17.1666666666667 78995.1127808733
17.3333333333333 62588.5408836736
17.5 63708.7292306832
17.6666666666667 89226.9605962374
17.8333333333333 86419.9385142545
18 113186.442548937
18.1666666666667 101683.0003125
18.3333333333333 95431.208920007
18.5 114016.638358033
18.6666666666667 130290.27892331
18.8333333333333 165339.053728837
19 175651.017208127
19.1666666666667 154718.320171408
19.3333333333333 167278.553226655
19.5 210141.355700364
19.6666666666667 220355.489706988
19.8333333333333 213649.363969251
20 226134.001510877
20.1666666666667 166601.026005282
20.3333333333333 161232.351357221
20.5 181515.26970254
20.6666666666667 166268.690543914
20.8333333333333 166804.384733292
21 106067.031947015
21.1666666666667 105992.477508119
21.3333333333333 108616.51802458
21.5 113058.570095595
21.6666666666667 86359.4778766154
21.8333333333333 108206.272443245
22 103765.147922332
22.1666666666667 120899.516889323
22.3333333333333 148533.642623547
22.5 77762.0153816321
22.6666666666667 79298.0312188235
22.8333333333333 117863.976611325
23 76531.4192104554
23.1666666666667 46069.019865
23.3333333333333 24481.7383667838
23.5 7531.2237186
23.6666666666667 26939.1608784574
23.8333333333333 32984.2612215
};
\end{axis}

\end{tikzpicture}

%% file: graphs/av_consumption_oneday.tex
\begin{tikzpicture}[scale=0.50]

\definecolor{darkgray176}{RGB}{176,176,176}

\begin{axis}[
tick align=outside,
tick pos=left,
x grid style={darkgray176},
xmin=-1.19166666666667, xmax=24.5,
ylabel=Average cons. ($\%$ of max. cons.),
xlabel=Time (hours),
xtick style={color=black},
ylabel= Average cons. ($\%$ of max. cons.),
y grid style={darkgray176},
ymin=-3.515, ymax=83.815,
ytick style={color=black}
]
\addplot [thick, orange]
table {%
0 15.4
0.166666666666667 13.8
0.333333333333333 12.1
0.5 11.1
0.666666666666667 10.6
0.833333333333333 9.6
1 8.1
1.16666666666667 7.5
1.33333333333333 6.9
1.5 6.6
1.66666666666667 5.6
1.83333333333333 4.9
2 4.9
2.16666666666667 4.7
2.33333333333333 4.5
2.5 4.1
2.66666666666667 4.2
2.83333333333333 4.2
3 3.9
3.16666666666667 3.6
3.33333333333333 3.7
3.5 3.9
3.66666666666667 3.8
3.83333333333333 4
4 3.9
4.16666666666667 3.9
4.33333333333333 3.9
4.5 4.1
4.66666666666667 4.2
4.83333333333333 4.3
5 5.7
5.16666666666667 6.4
5.33333333333333 6.7
5.5 8.6
5.66666666666667 9.3
5.83333333333333 10.3
6 12.2
6.16666666666667 15.1
6.33333333333333 19.3
6.5 23.6
6.66666666666667 29.5
6.83333333333333 35.8
7 44.1
7.16666666666667 51.2
7.33333333333333 57.7
7.5 61.9
7.66666666666667 65.9
7.83333333333333 69.3
8 71.3
8.16666666666667 74.1
8.33333333333333 75.1
8.5 74.7
8.66666666666667 74.5
8.83333333333333 74.2
9 74.9
9.16666666666667 72.4
9.33333333333333 71
9.5 67.5
9.66666666666667 64.3
9.83333333333333 58.4
10 53.3
10.1666666666667 48.2
10.3333333333333 41.5
10.5 36
10.6666666666667 31
10.8333333333333 26.2
11 21.5
11.1666666666667 17.9
11.3333333333333 13.9
11.5 11.3
11.6666666666667 10.3
11.8333333333333 8.9
12 7.5
12.1666666666667 6.1
12.3333333333333 5.5
12.5 5.3
12.6666666666667 5.2
12.8333333333333 5.1
13 5.4
13.1666666666667 5.2
13.3333333333333 5.3
13.5 4.6
13.6666666666667 4.2
13.8333333333333 3.9
14 3.9
14.1666666666667 4
14.3333333333333 3.3
14.5 3
14.6666666666667 2.8
14.8333333333333 2.5
15 2.9
15.1666666666667 2.5
15.3333333333333 2.6
15.5 2.7
15.6666666666667 2.6
15.8333333333333 2.6
16 2.5
16.1666666666667 2.7
16.3333333333333 3.2
16.5 3.1
16.6666666666667 3.7
16.8333333333333 3.8
17 4
17.1666666666667 4.1
17.3333333333333 4.8
17.5 5.5
17.6666666666667 5.8
17.8333333333333 6.5
18 7.9
18.1666666666667 8.8
18.3333333333333 10.1
18.5 11.1
18.6666666666667 11.1
18.8333333333333 12.3
19 15.8
19.1666666666667 17.5
19.3333333333333 19.7
19.5 21
19.6666666666667 24.9
19.8333333333333 28.1
20 30.7
20.1666666666667 34.8
20.3333333333333 37.3
20.5 39.9
20.6666666666667 42.4
20.8333333333333 44.7
21 46.1
21.1666666666667 46.4
21.3333333333333 46.5
21.5 45.9
21.6666666666667 45.8
21.8333333333333 42.2
22 40.2
22.1666666666667 37.6
22.3333333333333 34.7
22.5 31.4
22.6666666666667 29.6
22.8333333333333 26.5
23 24.7
23.1666666666667 22.4
23.3333333333333 20
23.5 17.8
23.6666666666667 17
23.8333333333333 15.8
};
\end{axis}

\end{tikzpicture}

%% file: graphs/dev_one.tex
\begin{tikzpicture}[scale=0.50]

\definecolor{darkgray176}{RGB}{176,176,176}
\definecolor{steelblue31119180}{RGB}{31,119,180}

\begin{axis}[
tick align=outside,
tick pos=left,
x grid style={darkgray176},
xmin=-1.19166666666667, xmax=24.5,
ylabel=Deviation ($\%$ of max. cons.),
xlabel=Time (hours),
xtick style={color=black},
y grid style={darkgray176},
ymin=-1.08333333333333, ymax=10.5277777777778,
ytick style={color=black}
]
\addplot [thick, black]
table {%
0 0
0.166666666666667 0
0.333333333333333 0
0.5 0
0.666666666666667 0
0.833333333333333 0
1 0
1.16666666666667 0
1.33333333333333 0
1.5 0
1.66666666666667 0
1.83333333333333 0
2 0
2.16666666666667 0
2.33333333333333 0
2.5 0
2.66666666666667 0
2.83333333333333 0
3 0
3.16666666666667 0
3.33333333333333 0
3.5 0
3.66666666666667 0
3.83333333333333 0
4 0
4.16666666666667 0
4.33333333333333 0
4.5 0
4.66666666666667 0
4.83333333333333 0
5 10
5.16666666666667 10
5.33333333333333 10
5.5 10
5.66666666666667 10
5.83333333333333 10
6 -0.555555555555556
6.16666666666667 -0.555555555555556
6.33333333333333 -0.555555555555556
6.5 -0.555555555555556
6.66666666666667 -0.555555555555556
6.83333333333333 -0.555555555555556
7 -0.555555555555556
7.16666666666667 -0.555555555555556
7.33333333333333 -0.555555555555556
7.5 -0.555555555555556
7.66666666666667 -0.555555555555556
7.83333333333333 -0.555555555555556
8 -0.555555555555556
8.16666666666667 -0.555555555555556
8.33333333333333 -0.555555555555556
8.5 -0.555555555555556
8.66666666666667 -0.555555555555556
8.83333333333333 -0.555555555555556
9 -0.555555555555556
9.16666666666667 -0.555555555555556
9.33333333333333 -0.555555555555556
9.5 -0.555555555555556
9.66666666666667 -0.555555555555556
9.83333333333333 -0.555555555555556
10 -0.555555555555556
10.1666666666667 -0.555555555555556
10.3333333333333 -0.555555555555556
10.5 -0.555555555555556
10.6666666666667 -0.555555555555556
10.8333333333333 -0.555555555555556
11 -0.555555555555556
11.1666666666667 -0.555555555555556
11.3333333333333 -0.555555555555556
11.5 -0.555555555555556
11.6666666666667 -0.555555555555556
11.8333333333333 -0.555555555555556
12 -0.555555555555556
12.1666666666667 -0.555555555555556
12.3333333333333 -0.555555555555556
12.5 -0.555555555555556
12.6666666666667 -0.555555555555556
12.8333333333333 -0.555555555555556
13 -0.555555555555556
13.1666666666667 -0.555555555555556
13.3333333333333 -0.555555555555556
13.5 -0.555555555555556
13.6666666666667 -0.555555555555556
13.8333333333333 -0.555555555555556
14 -0.555555555555556
14.1666666666667 -0.555555555555556
14.3333333333333 -0.555555555555556
14.5 -0.555555555555556
14.6666666666667 -0.555555555555556
14.8333333333333 -0.555555555555556
15 -0.555555555555556
15.1666666666667 -0.555555555555556
15.3333333333333 -0.555555555555556
15.5 -0.555555555555556
15.6666666666667 -0.555555555555556
15.8333333333333 -0.555555555555556
16 -0.555555555555556
16.1666666666667 -0.555555555555556
16.3333333333333 -0.555555555555556
16.5 -0.555555555555556
16.6666666666667 -0.555555555555556
16.8333333333333 -0.555555555555556
17 -0.555555555555556
17.1666666666667 -0.555555555555556
17.3333333333333 -0.555555555555556
17.5 -0.555555555555556
17.6666666666667 -0.555555555555556
17.8333333333333 -0.555555555555556
18 -0.555555555555556
18.1666666666667 -0.555555555555556
18.3333333333333 -0.555555555555556
18.5 -0.555555555555556
18.6666666666667 -0.555555555555556
18.8333333333333 -0.555555555555556
19 -0.555555555555556
19.1666666666667 -0.555555555555556
19.3333333333333 -0.555555555555556
19.5 -0.555555555555556
19.6666666666667 -0.555555555555556
19.8333333333333 -0.555555555555556
20 -0.555555555555556
20.1666666666667 -0.555555555555556
20.3333333333333 -0.555555555555556
20.5 -0.555555555555556
20.6666666666667 -0.555555555555556
20.8333333333333 -0.555555555555556
21 -0.555555555555556
21.1666666666667 -0.555555555555556
21.3333333333333 -0.555555555555556
21.5 -0.555555555555556
21.6666666666667 -0.555555555555556
21.8333333333333 -0.555555555555556
22 -0.555555555555556
22.1666666666667 -0.555555555555556
22.3333333333333 -0.555555555555556
22.5 -0.555555555555556
22.6666666666667 -0.555555555555556
22.8333333333333 -0.555555555555556
23 -0.555555555555556
23.1666666666667 -0.555555555555556
23.3333333333333 -0.555555555555556
23.5 -0.555555555555556
23.6666666666667 -0.555555555555556
23.8333333333333 -0.555555555555556
};
\end{axis}

\end{tikzpicture}

%% file: graphs/dev_eight.tex
\begin{tikzpicture}[scale=0.50]

\definecolor{darkgray176}{RGB}{176,176,176}
\definecolor{steelblue31119180}{RGB}{31,119,180}

\begin{axis}[
tick align=outside,
tick pos=left,
x grid style={darkgray176},
xmin=-1.19166666666667, xmax=24.5,
ylabel=Deviation ($\%$ of max. cons.),
xlabel=Time (hours),
xtick style={color=black},
y grid style={darkgray176},
ymin=-10.13636363636364, ymax=7.8636363636364,
ytick style={color=black}
]
\addplot [thick, black]
table {%
0 0
0.166666666666667 0
0.333333333333333 0
0.5 0
0.666666666666667 0
0.833333333333333 0
1 0
1.16666666666667 0
1.33333333333333 0
1.5 0
1.66666666666667 0
1.83333333333333 0
2 0
2.16666666666667 0
2.33333333333333 0
2.5 0
2.66666666666667 0
2.83333333333333 0
3 0
3.16666666666667 0
3.33333333333333 0
3.5 0
3.66666666666667 0
3.83333333333333 0
4 0
4.16666666666667 0
4.33333333333333 0
4.5 0
4.66666666666667 0
4.83333333333333 0
5 0
5.16666666666667 0
5.33333333333333 0
5.5 0
5.66666666666667 0
5.83333333333333 0
6 0
6.16666666666667 0
6.33333333333333 0
6.5 0
6.66666666666667 0
6.83333333333333 0
7 -8.88888888888889
7.16666666666667 -8.88888888888889
7.33333333333333 -8.88888888888889
7.5 -8.88888888888889
7.66666666666667 -8.88888888888889
7.83333333333333 -8.88888888888889
8 -8.88888888888889
8.16666666666667 -8.88888888888889
8.33333333333333 -8.88888888888889
8.5 -8.88888888888889
8.66666666666667 -8.88888888888889
8.83333333333333 -8.88888888888889
9 -8.88888888888889
9.16666666666667 -8.88888888888889
9.33333333333333 -8.88888888888889
9.5 -8.88888888888889
9.66666666666667 -8.88888888888889
9.83333333333333 -8.88888888888889
10 -8.88888888888889
10.1666666666667 -8.88888888888889
10.3333333333333 -8.88888888888889
10.5 -8.88888888888889
10.6666666666667 -8.88888888888889
10.8333333333333 -8.88888888888889
11 5
11.1666666666667 5
11.3333333333333 5
11.5 5
11.6666666666667 5
11.8333333333333 5
12 5
12.1666666666667 5
12.3333333333333 5
12.5 5
12.6666666666667 5
12.8333333333333 5
13 5
13.1666666666667 5
13.3333333333333 5
13.5 5
13.6666666666667 5
13.8333333333333 5
14 5
14.1666666666667 5
14.3333333333333 5
14.5 5
14.6666666666667 5
14.8333333333333 5
15 5
15.1666666666667 5
15.3333333333333 5
15.5 5
15.6666666666667 5
15.8333333333333 5
16 5
16.1666666666667 5
16.3333333333333 5
16.5 5
16.6666666666667 5
16.8333333333333 5
17 5
17.1666666666667 5
17.3333333333333 5
17.5 5
17.6666666666667 5
17.8333333333333 5
18 5
18.1666666666667 5
18.3333333333333 5
18.5 5
18.6666666666667 5
18.8333333333333 5
19 -8.88888888888889
19.1666666666667 -8.88888888888889
19.3333333333333 -8.88888888888889
19.5 -8.88888888888889
19.6666666666667 -8.88888888888889
19.8333333333333 -8.88888888888889
20 -8.88888888888889
20.1666666666667 -8.88888888888889
20.3333333333333 -8.88888888888889
20.5 -8.88888888888889
20.6666666666667 -8.88888888888889
20.8333333333333 -8.88888888888889
21 -8.88888888888889
21.1666666666667 -8.88888888888889
21.3333333333333 -8.88888888888889
21.5 -8.88888888888889
21.6666666666667 -8.88888888888889
21.8333333333333 -8.88888888888889
22 -8.88888888888889
22.1666666666667 -8.88888888888889
22.3333333333333 -8.88888888888889
22.5 -8.88888888888889
22.6666666666667 -8.88888888888889
22.8333333333333 -8.88888888888889
23 -8.88888888888889
23.1666666666667 -8.88888888888889
23.3333333333333 -8.88888888888889
23.5 -8.88888888888889
23.6666666666667 -8.88888888888889
23.8333333333333 -8.88888888888889
};
\end{axis}

\end{tikzpicture}

%% file: graphs/consumption_target_one_10_4.tex
\begin{tikzpicture}[scale=0.55]

\definecolor{color0}{rgb}{1,0.647058823529412,0}
\definecolor{color1}{rgb}{0.501960784313725,0,0.501960784313725}

\begin{axis}[
legend cell align={left},
legend style={fill opacity=0.8, draw opacity=1, text opacity=1, draw=white!80!black, at={(0.875,0.99)},
anchor=north east},
tick align=outside,
tick pos=left,
x grid style={white!69.0196078431373!black},
xticklabels={$0$,$4$,$8$,$12$,$16$,$20$,$24$},
xtick={0,24,48,72,96,120,141},
xlabel={Time (hours)},
xmin=-7, xmax=147,
xtick style={color=black},
yticklabels={0,$0$,$20$,$40$,$60$,$80$},
ylabel={Average cons. ($\%$ of max. cons.)},
y grid style={white!69.0196078431373!black},
ymin=-0.0826636363636364, ymax=0.899936363636364,
ytick style={color=black}
]
\addplot [very thick, red]
table {%
0 0.1619
1 0.1504
2 0.1389
3 0.1266
4 0.1176
5 0.1053
6 0.097
7 0.0896
8 0.0818
9 0.0755
10 0.0707
11 0.0679
12 0.0646
13 0.0633
14 0.0611
15 0.0602
16 0.0567
17 0.0544
18 0.053
19 0.0516
20 0.0494
21 0.0493
22 0.0472
23 0.0464
24 0.0453
25 0.0474
26 0.0481
27 0.0492
28 0.0497
29 0.0513
30 0.1572
31 0.1621
32 0.1669
33 0.1751
34 0.1859
35 0.1948
36 0.110544444444444
37 0.142444444444444
38 0.177244444444444
39 0.219644444444444
40 0.283644444444444
41 0.341544444444444
42 0.411744444444444
43 0.490344444444444
44 0.546744444444444
45 0.594744444444444
46 0.640544444444444
47 0.671144444444444
48 0.698444444444445
49 0.721544444444444
50 0.736844444444444
51 0.745644444444444
52 0.752444444444444
53 0.750944444444445
54 0.741144444444444
55 0.727344444444444
56 0.703944444444444
57 0.671044444444444
58 0.631444444444444
59 0.589744444444444
60 0.537844444444444
61 0.480544444444444
62 0.420344444444444
63 0.361044444444444
64 0.308744444444444
65 0.263944444444444
66 0.215144444444444
67 0.174044444444444
68 0.144944444444444
69 0.117144444444444
70 0.0981444444444444
71 0.0821444444444444
72 0.0680444444444444
73 0.0578444444444444
74 0.0499444444444444
75 0.0457444444444444
76 0.0412444444444444
77 0.0374444444444444
78 0.0363444444444445
79 0.0360444444444444
80 0.0365444444444444
81 0.0360444444444444
82 0.0343444444444444
83 0.0358444444444444
84 0.0337444444444444
85 0.0345444444444444
86 0.0349444444444444
87 0.0342444444444444
88 0.0327444444444444
89 0.0309444444444444
90 0.0294444444444444
91 0.0286444444444444
92 0.0266444444444444
93 0.0249444444444444
94 0.0238444444444444
95 0.0237444444444444
96 0.0237444444444444
97 0.0242444444444444
98 0.0245444444444444
99 0.0255444444444444
100 0.0279444444444444
101 0.0299444444444444
102 0.0322444444444444
103 0.0351444444444445
104 0.0403444444444444
105 0.0438444444444444
106 0.0479444444444444
107 0.0540444444444444
108 0.0619444444444444
109 0.0725444444444444
110 0.0834444444444444
111 0.0952444444444445
112 0.111044444444444
113 0.128044444444444
114 0.148544444444444
115 0.171744444444444
116 0.193144444444444
117 0.217244444444444
118 0.250644444444444
119 0.282744444444444
120 0.313444444444444
121 0.348244444444444
122 0.373444444444444
123 0.391244444444444
124 0.414344444444444
125 0.434044444444444
126 0.441944444444444
127 0.442144444444444
128 0.438244444444444
129 0.437244444444444
130 0.424844444444444
131 0.405944444444444
132 0.381144444444444
133 0.355244444444444
134 0.335744444444444
135 0.313644444444444
136 0.290344444444444
137 0.270144444444444
138 0.252644444444444
139 0.231644444444444
140 0.217244444444444
141 0.198744444444444
142 0.184244444444444
};
\addlegendentry{target}
\addplot [thick, color0]
table {%
0 0.1619
1 0.1563
2 0.1427
3 0.1282
4 0.1221
5 0.1142
6 0.1051
7 0.102
8 0.096
9 0.0813
10 0.0769
11 0.0761
12 0.0686
13 0.0783
14 0.0679
15 0.0697
16 0.0646
17 0.0566
18 0.0628
19 0.0598
20 0.0586
21 0.0574
22 0.0527
23 0.0417
24 0.0673
25 0.0457
26 0.0485
27 0.0674
28 0.0557
29 0.0665
30 0.163
31 0.1657
32 0.1711
33 0.1692
34 0.1866
35 0.1917
36 0.1153
37 0.1427
38 0.1724
39 0.1978
40 0.2819
41 0.3114
42 0.3873
43 0.4521
44 0.5179
45 0.5655
46 0.627
47 0.6768
48 0.6945
49 0.707
50 0.7436
51 0.7475
52 0.7456
53 0.751
54 0.7315
55 0.7284
56 0.6989
57 0.6775
58 0.626
59 0.6051
60 0.5499
61 0.4931
62 0.4176
63 0.3864
64 0.3167
65 0.2891
66 0.2304
67 0.1899
68 0.1556
69 0.1331
70 0.0983
71 0.0864
72 0.0822
73 0.0664
74 0.0562
75 0.0591
76 0.0516
77 0.0509
78 0.0395
79 0.0479
80 0.0467
81 0.045
82 0.0433
83 0.0427
84 0.0446
85 0.0414
86 0.0472
87 0.0483
88 0.0396
89 0.0295
90 0.037
91 0.0375
92 0.0393
93 0.0291
94 0.0329
95 0.0317
96 0.0299
97 0.0331
98 0.0289
99 0.0358
100 0.0548
101 0.0398
102 0.0429
103 0.0446
104 0.0422
105 0.057
106 0.0519
107 0.0733
108 0.0738
109 0.0826
110 0.1044
111 0.0869
112 0.1191
113 0.1367
114 0.1624
115 0.1846
116 0.1966
117 0.2269
118 0.2532
119 0.2762
120 0.3201
121 0.3409
122 0.3692
123 0.3922
124 0.4093
125 0.4355
126 0.4354
127 0.4338
128 0.4387
129 0.4408
130 0.4336
131 0.398
132 0.38
133 0.3629
134 0.3442
135 0.3338
136 0.3026
137 0.2824
138 0.257
139 0.2461
140 0.2262
141 0.2092
142 0.2012
};
\addlegendentry{FP-MFG}
\addplot [thick, color1]
table {%
0 0.1619
1 0.1504
2 0.1389
3 0.1266
4 0.1176
5 0.1053
6 0.097
7 0.0896
8 0.0818
9 0.0755
10 0.0707
11 0.0679
12 0.0646
13 0.0633
14 0.0611
15 0.0602
16 0.0567
17 0.0544
18 0.053
19 0.0516
20 0.0494
21 0.0493
22 0.0472
23 0.0464
24 0.0453
25 0.0474
26 0.0481
27 0.0492
28 0.0497
29 0.0513
30 0.0572
31 0.0621
32 0.0669
33 0.0751
34 0.0859
35 0.0948
36 0.1161
37 0.148
38 0.1828
39 0.2252
40 0.2892
41 0.3471
42 0.4173
43 0.4959
44 0.5523
45 0.6003
46 0.6461
47 0.6767
48 0.704
49 0.7271
50 0.7424
51 0.7512
52 0.758
53 0.7565
54 0.7467
55 0.7329
56 0.7095
57 0.6766
58 0.637
59 0.5953
60 0.5434
61 0.4861
62 0.4259
63 0.3666
64 0.3143
65 0.2695
66 0.2207
67 0.1796
68 0.1505
69 0.1227
70 0.1037
71 0.0877
72 0.0736
73 0.0634
74 0.0555
75 0.0513
76 0.0468
77 0.043
78 0.0419
79 0.0416
80 0.0421
81 0.0416
82 0.0399
83 0.0414
84 0.0393
85 0.0401
86 0.0405
87 0.0398
88 0.0383
89 0.0365
90 0.035
91 0.0342
92 0.0322
93 0.0305
94 0.0294
95 0.0293
96 0.0293
97 0.0298
98 0.0301
99 0.0311
100 0.0335
101 0.0355
102 0.0378
103 0.0407
104 0.0459
105 0.0494
106 0.0535
107 0.0596
108 0.0675
109 0.0781
110 0.089
111 0.1008
112 0.1166
113 0.1336
114 0.1541
115 0.1773
116 0.1987
117 0.2228
118 0.2562
119 0.2883
120 0.319
121 0.3538
122 0.379
123 0.3968
124 0.4199
125 0.4396
126 0.4475
127 0.4477
128 0.4438
129 0.4428
130 0.4304
131 0.4115
132 0.3867
133 0.3608
134 0.3413
135 0.3192
136 0.2959
137 0.2757
138 0.2582
139 0.2372
140 0.2228
141 0.2043
142 0.1898
};
\addlegendentry{nominal}
\addplot [thick, blue]
table {%
0 0.1619
1 0.1505
2 0.1375
3 0.1287
4 0.1223
5 0.1096
6 0.0985
7 0.0945
8 0.0827
9 0.0765
10 0.0698
11 0.0717
12 0.0632
13 0.0677
14 0.067
15 0.0633
16 0.0603
17 0.0597
18 0.0558
19 0.0609
20 0.0538
21 0.0561
22 0.0548
23 0.0504
24 0.0513
25 0.0489
26 0.053
27 0.0538
28 0.0568
29 0.0545
30 0.162
31 0.1623
32 0.1625
33 0.1715
34 0.1842
35 0.1882
36 0.0973
37 0.1254
38 0.1569
39 0.2035
40 0.2729
41 0.3184
42 0.4026
43 0.4766
44 0.5343
45 0.5826
46 0.637
47 0.6742
48 0.7024
49 0.7226
50 0.7459
51 0.7503
52 0.7593
53 0.7535
54 0.7472
55 0.7355
56 0.7185
57 0.6796
58 0.6394
59 0.5945
60 0.5363
61 0.4966
62 0.4247
63 0.3613
64 0.3233
65 0.2654
66 0.2197
67 0.1774
68 0.1529
69 0.1203
70 0.1031
71 0.086
72 0.0728
73 0.0621
74 0.0543
75 0.0515
76 0.0497
77 0.042
78 0.043
79 0.0421
80 0.0463
81 0.0451
82 0.0385
83 0.0449
84 0.0408
85 0.0415
86 0.0445
87 0.042
88 0.0424
89 0.0395
90 0.0424
91 0.0381
92 0.034
93 0.0352
94 0.0334
95 0.0291
96 0.0326
97 0.0313
98 0.0341
99 0.0324
100 0.0373
101 0.0429
102 0.0423
103 0.0431
104 0.0516
105 0.053
106 0.0564
107 0.0586
108 0.0695
109 0.0732
110 0.0925
111 0.1028
112 0.111
113 0.1256
114 0.1622
115 0.1678
116 0.1808
117 0.2072
118 0.2344
119 0.2798
120 0.3062
121 0.3504
122 0.368
123 0.3943
124 0.4205
125 0.4375
126 0.4392
127 0.4397
128 0.4375
129 0.447
130 0.4296
131 0.4149
132 0.388
133 0.3546
134 0.3376
135 0.3095
136 0.2879
137 0.2716
138 0.2446
139 0.2308
140 0.2219
141 0.2023
142 0.1872
};
\addlegendentry{MD-MFC}
\addplot [thick, green!50.1960784313725!black, dashed]
table {%
0 0.1619
1 0.1507
2 0.1474
3 0.1269
4 0.1208
5 0.1116
6 0.1007
7 0.0987
8 0.0811
9 0.0752
10 0.0744
11 0.0714
12 0.0675
13 0.061
14 0.0653
15 0.0614
16 0.0591
17 0.0563
18 0.0577
19 0.0512
20 0.0532
21 0.0512
22 0.0533
23 0.0512
24 0.0494
25 0.0493
26 0.0482
27 0.054
28 0.0501
29 0.0549
30 0.1593
31 0.1611
32 0.1689
33 0.1806
34 0.192
35 0.1909
36 0.1035
37 0.1192
38 0.1592
39 0.2009
40 0.2687
41 0.306
42 0.3864
43 0.4621
44 0.5243
45 0.5793
46 0.6285
47 0.6788
48 0.706
49 0.7226
50 0.7464
51 0.7498
52 0.7551
53 0.7459
54 0.7481
55 0.7366
56 0.7082
57 0.6754
58 0.6378
59 0.6064
60 0.5535
61 0.4949
62 0.4175
63 0.3772
64 0.3229
65 0.2672
66 0.2296
67 0.183
68 0.1582
69 0.1304
70 0.099
71 0.088
72 0.0751
73 0.0659
74 0.0535
75 0.0525
76 0.0503
77 0.0461
78 0.0423
79 0.0432
80 0.0439
81 0.0405
82 0.041
83 0.0443
84 0.0417
85 0.041
86 0.0399
87 0.0444
88 0.0402
89 0.0426
90 0.0368
91 0.0375
92 0.0366
93 0.0371
94 0.0358
95 0.0309
96 0.0314
97 0.0333
98 0.034
99 0.0356
100 0.038
101 0.0401
102 0.0428
103 0.0448
104 0.0445
105 0.0489
106 0.0506
107 0.0587
108 0.0646
109 0.0731
110 0.0852
111 0.101
112 0.121
113 0.1274
114 0.1529
115 0.1733
116 0.1818
117 0.233
118 0.2459
119 0.2753
120 0.3175
121 0.3561
122 0.3774
123 0.3914
124 0.4112
125 0.4383
126 0.4401
127 0.4479
128 0.4462
129 0.4378
130 0.4266
131 0.4057
132 0.3804
133 0.3566
134 0.3257
135 0.3111
136 0.2946
137 0.274
138 0.2539
139 0.2268
140 0.2243
141 0.1933
142 0.1855
};
\addlegendentry{OMD-MFG}
\end{axis}

\end{tikzpicture}

%% file: graphs/consumption_target_eight_10_4heaters.tex
\begin{tikzpicture}[scale=0.55]

\definecolor{color0}{rgb}{1,0.647058823529412,0}
\definecolor{color1}{rgb}{0.501960784313725,0,0.501960784313725}

\begin{axis}[
legend cell align={left},
legend style={fill opacity=0.8, draw opacity=1, text opacity=1, draw=white!80!black, at={(0.875,0.99)},
anchor=north east},
tick align=outside,
tick pos=left,
x grid style={white!69.0196078431373!black},
xticklabels={$0$,$4$,$8$,$12$,$16$,$20$,$24$},
xtick={0,24,48,72,96,120,141},
xlabel={Time (hours)},
xmin=-7, xmax=147,
xtick style={color=black},
yticklabels={0,$0$,$20$,$40$,$60$,$80$},
ylabel={Average cons. ($\%$ of max. cons.)},
y grid style={white!69.0196078431373!black},
ymin=-0.0826636363636364, ymax=0.899936363636364,
ytick style={color=black}
]
\addplot [very thick, red]
table {%
0 0.1614
1 0.1511
2 0.1374
3 0.1263
4 0.118
5 0.1084
6 0.0977
7 0.09
8 0.0821
9 0.0773
10 0.0733
11 0.0703
12 0.0649
13 0.0647
14 0.0616
15 0.0577
16 0.0555
17 0.0547
18 0.0547
19 0.0532
20 0.0514
21 0.05
22 0.0485
23 0.0484
24 0.0487
25 0.0509
26 0.0507
27 0.0517
28 0.0518
29 0.0521
30 0.0578
31 0.0623
32 0.0672
33 0.0767
34 0.0831
35 0.0954
36 0.116
37 0.1481
38 0.1829
39 0.2231
40 0.2861
41 0.3478
42 0.332011111111111
43 0.409311111111111
44 0.465111111111111
45 0.512111111111111
46 0.555411111111111
47 0.586611111111111
48 0.610611111111111
49 0.632211111111111
50 0.650511111111111
51 0.658411111111111
52 0.665611111111111
53 0.664411111111111
54 0.653811111111111
55 0.641411111111111
56 0.616811111111111
57 0.585011111111111
58 0.548411111111111
59 0.504311111111111
60 0.452811111111111
61 0.395411111111111
62 0.341511111111111
63 0.282011111111111
64 0.228311111111111
65 0.177211111111111
66 0.2749
67 0.2362
68 0.206
69 0.179
70 0.1562
71 0.1391
72 0.1266
73 0.116
74 0.1101
75 0.1051
76 0.1006
77 0.0967
78 0.095
79 0.0942
80 0.0932
81 0.0913
82 0.0904
83 0.09
84 0.0894
85 0.0892
86 0.0879
87 0.086
88 0.0867
89 0.0853
90 0.0849
91 0.0837
92 0.0829
93 0.0823
94 0.0813
95 0.0806
96 0.08
97 0.0821
98 0.0829
99 0.0841
100 0.0861
101 0.088
102 0.0893
103 0.0919
104 0.0962
105 0.1004
106 0.105
107 0.1113
108 0.1182
109 0.1282
110 0.1383
111 0.1484
112 0.159
113 0.1761
114 0.0589111111111111
115 0.0807111111111111
116 0.101411111111111
117 0.126111111111111
118 0.163011111111111
119 0.196211111111111
120 0.230411111111111
121 0.264411111111111
122 0.289311111111111
123 0.312011111111111
124 0.337211111111111
125 0.359611111111111
126 0.371411111111111
127 0.372311111111111
128 0.366911111111111
129 0.359111111111111
130 0.342511111111111
131 0.322911111111111
132 0.299811111111111
133 0.272511111111111
134 0.252711111111111
135 0.230511111111111
136 0.209111111111111
137 0.183611111111111
138 0.165311111111111
139 0.146011111111111
140 0.130011111111111
141 0.112511111111111
142 0.0991111111111111
};
\addlegendentry{target}
\addplot [thick, color0]
table {%
0 0.1614
1 0.1559
2 0.1468
3 0.1292
4 0.1255
5 0.1105
6 0.0994
7 0.0942
8 0.0832
9 0.0823
10 0.08
11 0.0815
12 0.069
13 0.0658
14 0.0684
15 0.1182
16 0.0635
17 0.0673
18 0.0647
19 0.0679
20 0.0648
21 0.0591
22 0.06
23 0.0611
24 0.0547
25 0.0527
26 0.0809
27 0.0588
28 0.0621
29 0.0647
30 0.0566
31 0.0767
32 0.0767
33 0.0834
34 0.0784
35 0.1118
36 0.1086
37 0.1401
38 0.1861
39 0.2087
40 0.2801
41 0.3183
42 0.3009
43 0.3661
44 0.4403
45 0.4917
46 0.5404
47 0.5942
48 0.6094
49 0.6334
50 0.6407
51 0.6721
52 0.6645
53 0.6757
54 0.651
55 0.6457
56 0.6267
57 0.5867
58 0.5609
59 0.5171
60 0.4663
61 0.4136
62 0.3609
63 0.2882
64 0.2471
65 0.2027
66 0.2781
67 0.2499
68 0.2186
69 0.1935
70 0.1687
71 0.1407
72 0.1362
73 0.1307
74 0.1237
75 0.1174
76 0.1175
77 0.1032
78 0.1133
79 0.1012
80 0.0933
81 0.0935
82 0.0978
83 0.0975
84 0.0996
85 0.0936
86 0.0983
87 0.096
88 0.0882
89 0.0906
90 0.0945
91 0.086
92 0.0829
93 0.0947
94 0.0919
95 0.0831
96 0.0882
97 0.094
98 0.0854
99 0.1008
100 0.0901
101 0.0906
102 0.0952
103 0.0982
104 0.1078
105 0.1137
106 0.1119
107 0.1097
108 0.1212
109 0.1412
110 0.1441
111 0.1641
112 0.1774
113 0.1887
114 0.07
115 0.0884
116 0.1004
117 0.1504
118 0.1827
119 0.2026
120 0.2393
121 0.2812
122 0.2942
123 0.3253
124 0.3365
125 0.3647
126 0.3697
127 0.3721
128 0.3731
129 0.3631
130 0.3465
131 0.3234
132 0.3099
133 0.2795
134 0.2693
135 0.2306
136 0.2195
137 0.1902
138 0.1701
139 0.1601
140 0.139
141 0.1154
142 0.1122
};
\addlegendentry{FP-MFG}
\addplot [thick, color1]
table {%
0 0.1614
1 0.1511
2 0.1374
3 0.1263
4 0.118
5 0.1084
6 0.0977
7 0.09
8 0.0821
9 0.0773
10 0.0733
11 0.0703
12 0.0649
13 0.0647
14 0.0616
15 0.0577
16 0.0555
17 0.0547
18 0.0547
19 0.0532
20 0.0514
21 0.05
22 0.0485
23 0.0484
24 0.0487
25 0.0509
26 0.0507
27 0.0517
28 0.0518
29 0.0521
30 0.0578
31 0.0623
32 0.0672
33 0.0767
34 0.0831
35 0.0954
36 0.116
37 0.1481
38 0.1829
39 0.2231
40 0.2861
41 0.3478
42 0.4209
43 0.4982
44 0.554
45 0.601
46 0.6443
47 0.6755
48 0.6995
49 0.7211
50 0.7394
51 0.7473
52 0.7545
53 0.7533
54 0.7427
55 0.7303
56 0.7057
57 0.6739
58 0.6373
59 0.5932
60 0.5417
61 0.4843
62 0.4304
63 0.3709
64 0.3172
65 0.2661
66 0.2249
67 0.1862
68 0.156
69 0.129
70 0.1062
71 0.0891
72 0.0766
73 0.066
74 0.0601
75 0.0551
76 0.0506
77 0.0467
78 0.045
79 0.0442
80 0.0432
81 0.0413
82 0.0404
83 0.04
84 0.0394
85 0.0392
86 0.0379
87 0.036
88 0.0367
89 0.0353
90 0.0349
91 0.0337
92 0.0329
93 0.0323
94 0.0313
95 0.0306
96 0.03
97 0.0321
98 0.0329
99 0.0341
100 0.0361
101 0.038
102 0.0393
103 0.0419
104 0.0462
105 0.0504
106 0.055
107 0.0613
108 0.0682
109 0.0782
110 0.0883
111 0.0984
112 0.109
113 0.1261
114 0.1478
115 0.1696
116 0.1903
117 0.215
118 0.2519
119 0.2851
120 0.3193
121 0.3533
122 0.3782
123 0.4009
124 0.4261
125 0.4485
126 0.4603
127 0.4612
128 0.4558
129 0.448
130 0.4314
131 0.4118
132 0.3887
133 0.3614
134 0.3416
135 0.3194
136 0.298
137 0.2725
138 0.2542
139 0.2349
140 0.2189
141 0.2014
142 0.188
};
\addlegendentry{nominal}
\addplot [thick, blue]
table {%
0 0.1614
1 0.1509
2 0.145
3 0.1284
4 0.1215
5 0.1089
6 0.0987
7 0.0956
8 0.0859
9 0.0828
10 0.0723
11 0.0743
12 0.0694
13 0.0727
14 0.064
15 0.0617
16 0.0597
17 0.062
18 0.0624
19 0.0609
20 0.0578
21 0.0543
22 0.0538
23 0.0549
24 0.0538
25 0.0567
26 0.0553
27 0.061
28 0.0599
29 0.0543
30 0.059
31 0.0625
32 0.0694
33 0.08
34 0.0894
35 0.1002
36 0.1018
37 0.1306
38 0.1642
39 0.2017
40 0.2672
41 0.3231
42 0.3003
43 0.3784
44 0.442
45 0.4993
46 0.5439
47 0.5983
48 0.6221
49 0.6317
50 0.6588
51 0.6692
52 0.6716
53 0.6686
54 0.6578
55 0.6426
56 0.6376
57 0.5921
58 0.5571
59 0.5189
60 0.4704
61 0.3978
62 0.3484
63 0.2893
64 0.2405
65 0.1806
66 0.2767
67 0.2454
68 0.2092
69 0.1801
70 0.1637
71 0.1412
72 0.1319
73 0.1229
74 0.1073
75 0.1054
76 0.1039
77 0.0962
78 0.1015
79 0.0951
80 0.089
81 0.0978
82 0.0911
83 0.0914
84 0.0935
85 0.0947
86 0.0945
87 0.0907
88 0.0911
89 0.0928
90 0.0905
91 0.0848
92 0.0942
93 0.0836
94 0.0878
95 0.0868
96 0.0771
97 0.0846
98 0.0852
99 0.0887
100 0.0812
101 0.0922
102 0.0908
103 0.0858
104 0.1058
105 0.092
106 0.1074
107 0.1038
108 0.1137
109 0.1294
110 0.1389
111 0.1486
112 0.1468
113 0.1791
114 0.0636
115 0.0734
116 0.0983
117 0.1285
118 0.1658
119 0.1938
120 0.2373
121 0.2705
122 0.3018
123 0.3214
124 0.3396
125 0.3602
126 0.3731
127 0.3661
128 0.367
129 0.3605
130 0.3423
131 0.3244
132 0.3065
133 0.2726
134 0.2614
135 0.2338
136 0.2135
137 0.1842
138 0.1676
139 0.1454
140 0.1286
141 0.1123
142 0.1056
};
\addlegendentry{MD-MFC}
\addplot [thick, green!50.1960784313725!black, dashed]
table {%
0 0.1614
1 0.154
2 0.1403
3 0.1266
4 0.1197
5 0.1173
6 0.1003
7 0.0983
8 0.0851
9 0.0816
10 0.0741
11 0.0736
12 0.066
13 0.0645
14 0.0683
15 0.0642
16 0.0591
17 0.058
18 0.058
19 0.0553
20 0.0553
21 0.0567
22 0.0542
23 0.0532
24 0.0522
25 0.0528
26 0.0552
27 0.0587
28 0.0548
29 0.0584
30 0.0587
31 0.0627
32 0.063
33 0.077
34 0.0808
35 0.0986
36 0.1076
37 0.1419
38 0.1811
39 0.2112
40 0.2759
41 0.3291
42 0.314
43 0.3914
44 0.4569
45 0.5109
46 0.5772
47 0.5979
48 0.6228
49 0.6445
50 0.6699
51 0.6708
52 0.6794
53 0.6791
54 0.6653
55 0.6539
56 0.6231
57 0.5949
58 0.5586
59 0.5183
60 0.4677
61 0.4143
62 0.3436
63 0.3029
64 0.2412
65 0.179
66 0.2822
67 0.244
68 0.2147
69 0.189
70 0.1658
71 0.1426
72 0.1316
73 0.1225
74 0.1153
75 0.114
76 0.1069
77 0.1019
78 0.1051
79 0.0957
80 0.0954
81 0.0938
82 0.0973
83 0.0944
84 0.0881
85 0.0923
86 0.089
87 0.0899
88 0.0893
89 0.0838
90 0.0855
91 0.0825
92 0.0828
93 0.0908
94 0.083
95 0.0839
96 0.0822
97 0.0894
98 0.0899
99 0.0861
100 0.0876
101 0.0868
102 0.0877
103 0.0899
104 0.0983
105 0.099
106 0.1059
107 0.1086
108 0.1202
109 0.1322
110 0.1422
111 0.1453
112 0.1684
113 0.1737
114 0.0641
115 0.0785
116 0.0994
117 0.1268
118 0.1515
119 0.1966
120 0.2345
121 0.2722
122 0.2867
123 0.3124
124 0.3271
125 0.3652
126 0.3555
127 0.3711
128 0.3755
129 0.35
130 0.3508
131 0.3261
132 0.2931
133 0.2734
134 0.2521
135 0.2289
136 0.2079
137 0.1891
138 0.1742
139 0.1443
140 0.1393
141 0.1115
142 0.1024
};
\addlegendentry{OMD-MFG}
\end{axis}

\end{tikzpicture}

%% file: graphs/logobjfunc_onehour_line.tex
\begin{tikzpicture}[scale=0.50]

\definecolor{color0}{rgb}{1,0.647058823529412,0}
\definecolor{color1}{rgb}{0.75,0,0.75}
\definecolor{color2}{rgb}{0,0.75,0.75}

\begin{axis}[
legend cell align={left},
legend style={fill opacity=0.8, draw opacity=1, text opacity=1, draw=white!80!black},
log basis x={10},
log basis y={10},
tick align=outside,
tick pos=left,
title={objective function per iteration},
x grid style={white!69.0196078431373!black},
xlabel={Iteration},
xmin=0.794727498311053, xmax=124.571001016567,
xmode=log,
xtick style={color=black},
y grid style={white!69.0196078431373!black},
ymin=2.25241850613944e-06, ymax=0.69286061031323,
ymode=log,
ytick style={color=black}
]
\addplot [semithick, blue]
table {%
0 0.0797075162490787
1 0.0610519046442798
2 0.0480220916710991
3 0.0376200333974469
4 0.027002337997377
5 0.0204493010821758
6 0.0143518575229638
7 0.0104915698996841
8 0.00794718775611963
9 0.00606514577372441
10 0.00492420571434318
11 0.00404111188316429
12 0.00327176460243835
13 0.00276094650552044
14 0.0022635455343334
15 0.00170789243633414
16 0.0016203047236197
17 0.00160339404840269
18 0.00136449299259707
19 0.000992469342582338
20 0.000924241435460338
21 0.000924241435460338
22 0.00072431785520802
23 0.00072431785520802
24 0.000672674025799592
25 0.000672674025799592
26 0.000672674025799592
27 0.000672674025799592
28 0.000429122467611607
29 0.000429122467611607
30 0.000429122467611607
31 0.000429122467611607
32 0.000429122467611607
33 0.000429122467611607
34 0.000429122467611607
35 0.000383043803059704
36 0.000350870420102422
37 0.000216830953095684
38 0.000216830953095684
39 0.000216830953095684
40 0.000216830953095684
41 0.000216830953095684
42 0.000216830953095684
43 0.000216830953095684
44 0.000216830953095684
45 0.000216830953095684
46 0.000216830953095684
47 0.000216830953095684
48 0.000216830953095684
49 0.00021389778815392
50 0.000196125801236818
51 0.000196125801236818
52 0.000196125801236818
53 0.000135789218384441
54 0.000135789218384441
55 0.000135789218384441
56 0.000135789218384441
57 0.000135789218384441
58 0.000135789218384441
59 0.000135789218384441
60 0.000135789218384441
61 0.000135789218384441
62 8.59310444212433e-05
63 8.59310444212433e-05
64 8.59310444212433e-05
65 8.59310444212433e-05
66 8.59310444212433e-05
67 8.59310444212433e-05
68 8.59310444212433e-05
69 8.59310444212433e-05
70 8.59310444212433e-05
71 8.59310444212433e-05
72 8.59310444212433e-05
73 8.59310444212433e-05
74 8.59310444212433e-05
75 8.59310444212433e-05
76 8.59310444212433e-05
77 8.59310444212433e-05
78 8.59310444212433e-05
79 8.59310444212433e-05
80 8.59310444212433e-05
81 8.59310444212433e-05
82 8.59310444212433e-05
83 8.59310444212433e-05
84 8.59310444212433e-05
85 8.59310444212433e-05
86 8.59310444212433e-05
87 8.59310444212433e-05
88 8.59310444212433e-05
89 8.59310444212433e-05
90 8.59310444212433e-05
91 8.59310444212433e-05
92 8.59310444212433e-05
93 8.31008857471599e-05
94 7.60801598622452e-05
95 7.60801598622452e-05
96 7.60801598622452e-05
97 6.55399770402604e-05
98 6.55399770402604e-05
};
\addlegendentry{MD-MFC}
\addplot [semithick, color0]
table {%
0 0.0767277356749265
1 0.0392878248648412
2 0.0559452000901789
3 0.032438558055752
4 0.0159096929180389
5 0.00961212697391143
6 0.00664674330035222
7 0.00476428961449954
8 0.00338794638513705
9 0.00255940242004639
10 0.00213292507300692
11 0.00180512249100398
12 0.00137134348552382
13 0.00115578384060043
14 0.000935064594277694
15 0.000815377768692356
16 0.000708568248958085
17 0.00059320235958596
18 0.000490618899730615
19 0.000425407717559192
20 0.000392016499445243
21 0.00034472119797947
22 0.000327522985167508
23 0.000292489925863277
24 0.000256963815040937
25 0.000228898007474534
26 0.000205644965268627
27 0.000197479056153187
28 0.000178172293788816
29 0.000154795977536544
30 0.000132136150170047
31 0.000132121064640085
32 0.000135565467411089
33 0.000129991006830733
34 0.000106660026514638
35 9.93829565101645e-05
36 8.65309586708397e-05
37 9.07347836889701e-05
38 8.88070069810256e-05
39 8.66033991602903e-05
40 7.83366524705769e-05
41 6.99990750219464e-05
42 7.522403973652e-05
43 6.74421657297009e-05
44 6.32074341660933e-05
45 6.83515977823317e-05
46 6.25508209045441e-05
47 5.51626686570554e-05
48 5.60559046578977e-05
49 5.58758414503959e-05
50 5.34437208187082e-05
51 4.74928221930276e-05
52 4.32972609429489e-05
53 4.58712113185518e-05
54 4.17679529473219e-05
55 4.06953674154945e-05
56 4.28723261611931e-05
57 4.20469114833805e-05
58 3.98452793132646e-05
59 3.51591562459115e-05
60 3.09737547312378e-05
61 3.23414707580148e-05
62 3.35263404042599e-05
63 2.77930709438439e-05
64 2.75791935456394e-05
65 2.82006909557737e-05
66 2.38357186982189e-05
67 2.22038468058595e-05
68 2.21881372470258e-05
69 2.52788079651496e-05
70 2.06275324713572e-05
71 2.14633068373651e-05
72 2.49175546821132e-05
73 2.36385794638245e-05
74 1.97887252723572e-05
75 2.03982752393011e-05
76 1.9643235193301e-05
77 1.57260525364035e-05
78 1.54338422697708e-05
79 1.80636145350569e-05
80 1.51425592405528e-05
81 1.47771496998583e-05
82 1.80208022666523e-05
83 1.6028996158011e-05
84 1.64494280304441e-05
85 1.66178053773187e-05
86 1.55814625055186e-05
87 1.53319908617068e-05
88 1.43085624583027e-05
89 1.34491822246585e-05
90 1.42070448526541e-05
91 1.34664026779983e-05
92 1.22942589583315e-05
93 1.43021047579057e-05
94 1.38476881439585e-05
95 1.27448630743537e-05
96 1.16536599122684e-05
97 1.2460839826462e-05
98 1.28077295843325e-05
};
\addlegendentry{FP-MFG}
\addplot [semithick, green!50.1960784313725!black]
table {%
0 0.0791938656562518
1 0.0625229716732428
2 0.0479650095183262
3 0.0379976034390582
4 0.0277590173382003
5 0.0204982293945598
6 0.0146105164768446
7 0.0109224585878291
8 0.0083024938508318
9 0.0061226688809725
10 0.00498845850161055
11 0.00377063663657695
12 0.00333423877601872
13 0.00268016197746575
14 0.00239567070786283
15 0.00183256657729481
16 0.00178476602706928
17 0.00144629974109666
18 0.00128605188056669
19 0.00122244794097261
20 0.00100917149986991
21 0.000780695785688435
22 0.00076475438944828
23 0.000688207900202563
24 0.000688207900202563
25 0.000648396444387917
26 0.000648396444387917
27 0.00048256666170028
28 0.000444205664262557
29 0.000444205664262557
30 0.00038601563156511
31 0.000371896935708369
32 0.00033815175730776
33 0.00033815175730776
34 0.00033815175730776
35 0.000310513027978236
36 0.000310513027978236
37 0.000263028121648836
38 0.000263028121648836
39 0.000263028121648836
40 0.000255455500791073
41 0.000255455500791073
42 0.000193391578447466
43 0.000193391578447466
44 0.000193391578447466
45 0.000193391578447466
46 0.000193391578447466
47 0.00018751005884975
48 0.000162060935854795
49 0.000162060935854795
50 0.000162060935854795
51 0.000162060935854795
52 0.000148869526631869
53 0.000148869526631869
54 0.000134847935847864
55 0.000134847935847864
56 0.000134847935847864
57 0.000134847935847864
58 0.000134847935847864
59 0.000134847935847864
60 0.000134847935847864
61 0.000134847935847864
62 0.000134847935847864
63 0.000134847935847864
64 0.000134847935847864
65 0.000108457854830458
66 0.000108457854830458
67 0.000106373559593056
68 0.000106373559593056
69 0.000106373559593056
70 0.000106373559593056
71 0.000106373559593056
72 0.000106373559593056
73 0.000106373559593056
74 0.000106373559593056
75 0.000106373559593056
76 0.000106373559593056
77 0.000106373559593056
78 0.000106373559593056
79 9.901420051395e-05
80 9.901420051395e-05
81 9.901420051395e-05
82 9.901420051395e-05
83 9.901420051395e-05
84 9.901420051395e-05
85 9.901420051395e-05
86 9.901420051395e-05
87 9.901420051395e-05
88 7.47676106307684e-05
89 7.47676106307684e-05
90 7.47676106307684e-05
91 7.47676106307684e-05
92 7.47676106307684e-05
93 7.47676106307684e-05
94 7.47676106307684e-05
95 7.47676106307684e-05
96 7.47676106307684e-05
97 7.47676106307684e-05
98 7.47676106307684e-05
};
\addlegendentry{OMD-MFG}
\addplot [semithick, color1]
table {%
1 0.3299211585026
2 0.08248028962565
3 0.0366579065002889
4 0.0206200724064125
5 0.013196846340104
6 0.00916447662507222
7 0.0067330848674
8 0.00515501810160313
9 0.00407310072225432
10 0.003299211585026
11 0.00272662114464959
12 0.00229111915626806
13 0.0019521962041574
14 0.00168327121685
15 0.00146631626001156
16 0.00128875452540078
17 0.00114159570416125
18 0.00101827518056358
19 0.000913909026322992
20 0.0008248028962565
21 0.000748120540822222
22 0.000681655286162396
23 0.000623669486772401
24 0.000572779789067014
25 0.00052787385360416
26 0.000488049051039349
27 0.000452566746917147
28 0.0004208178042125
29 0.000392296264569084
30 0.000366579065002889
31 0.000343310258587513
32 0.000322188631350195
33 0.000302957904961065
34 0.000285398926040311
35 0.000269323394696
36 0.000254568795140895
37 0.00024099427209832
38 0.000228477256580748
39 0.000216910689350822
40 0.000206200724064125
41 0.000196264817669601
42 0.000187030135205556
43 0.000178432211196647
44 0.000170413821540599
45 0.000162924028890173
46 0.0001559173716931
47 0.000149353172703757
48 0.000143194947266753
49 0.000137409895253061
50 0.00013196846340104
51 0.000126843967129027
52 0.000122012262759837
53 0.000117451462621075
54 0.000113141686729287
55 0.000109064845785983
56 0.000105204451053125
57 0.000101545447369221
58 9.80740661422711e-05
59 9.47776956341855e-05
60 9.16447662507222e-05
61 8.86646488854071e-05
62 8.58275646468782e-05
63 8.31245045358024e-05
64 8.05471578375488e-05
65 7.80878481662959e-05
66 7.57394762402663e-05
67 7.34954685904656e-05
68 7.13497315100778e-05
69 6.92966096413778e-05
70 6.73308486739999e-05
71 6.5447561694624e-05
72 6.36421987852237e-05
73 6.19105195163445e-05
74 6.02485680245799e-05
75 5.86526504004622e-05
76 5.71193141451869e-05
77 5.56453294826446e-05
78 5.42276723377054e-05
79 5.28635088131068e-05
80 5.15501810160313e-05
81 5.02851941019051e-05
82 4.90662044174003e-05
83 4.78910086373348e-05
84 4.67575338013889e-05
85 4.56638281664498e-05
86 4.46080527991617e-05
87 4.35884738410094e-05
88 4.26034553851497e-05
89 4.16514529103143e-05
90 4.07310072225432e-05
91 3.9840738860355e-05
92 3.8979342923275e-05
93 3.81455842875014e-05
94 3.73382931759393e-05
95 3.65563610529196e-05
96 3.57987368166884e-05
97 3.50644232652354e-05
98 3.43524738132653e-05
99 3.36619894401183e-05
};
\addlegendentry{$K^{-2}$}
\addplot [semithick, color2]
table {%
1 0.39011394307076
2 0.0689630536451893
3 0.0250258211162811
4 0.0121910607209613
5 0.00697857036541362
6 0.00442398195402096
7 0.00300916756951616
8 0.00215509542641217
9 0.0016054071731307
10 0.00123364860709286
11 0.000972097351353702
12 0.000782056909883784
13 0.00064022568275606
14 0.000531950698532878
15 0.000447675497670033
16 0.000380970647530039
17 0.00032739283397251
18 0.000283798574671561
19 0.000247917663553233
20 0.000218080323919176
21 0.000193038189588538
22 0.000171844157278921
23 0.000153770094999801
24 0.000138249436063155
25 0.000124836461782643
26 0.000113176980441649
27 0.000102986918173996
28 9.40364865473798e-05
29 8.6138336012019e-05
30 7.91385950433857e-05
31 7.29100234860354e-05
32 6.73467320753802e-05
33 6.23600741647314e-05
34 5.78754232534607e-05
35 5.38296259290076e-05
36 5.01689741603344e-05
37 4.68475801128801e-05
38 4.3826065268604e-05
39 4.10704966979248e-05
40 3.8551518971652e-05
41 3.62436419122892e-05
42 3.41246532215074e-05
43 3.21751317388351e-05
44 3.03780422298032e-05
45 2.87183965654882e-05
46 2.71829692295146e-05
47 2.57600574872008e-05
48 2.44392784338682e-05
49 2.32113966246659e-05
50 2.20681771664606e-05
51 2.10022600916425e-05
52 2.00070525861269e-05
53 1.90766362493579e-05
54 1.82056870535842e-05
55 1.7389406066995e-05
56 1.66234593291524e-05
57 1.59039255321469e-05
58 1.5227250378556e-05
59 1.459020666671e-05
60 1.39898593021885e-05
61 1.34235345576725e-05
62 1.28887930058615e-05
63 1.23834056358689e-05
64 1.19053327353137e-05
65 1.14527051806701e-05
66 1.10238078292944e-05
67 1.06170647495215e-05
68 1.02310260616409e-05
69 9.86435619349435e-06
70 9.51582338078411e-06
71 9.18429026467089e-06
72 8.86870545848628e-06
73 8.56809597189111e-06
74 8.28156039499938e-06
75 8.00826275720995e-06
76 7.74742698603853e-06
77 7.49833190036873e-06
78 7.2603066804506e-06
79 7.03272676384323e-06
80 6.81501012247425e-06
81 6.60661388119629e-06
82 6.40703124276916e-06
83 6.21578868817277e-06
84 6.03244342464182e-06
85 5.85658105687067e-06
86 5.6878134595252e-06
87 5.52577683156507e-06
88 5.37012991496628e-06
89 5.22055236227657e-06
90 5.07674323906528e-06
91 4.9384196487728e-06
92 4.80531546874377e-06
93 4.67718018736493e-06
94 4.55377783323875e-06
95 4.4348859882236e-06
96 4.3202948769736e-06
97 4.2098065263264e-06
98 4.10323398852795e-06
99 4.00040062285474e-06
};
\addlegendentry{$K^{-5/2}$}
\end{axis}

\end{tikzpicture}

%% file: graphs/logobjfunc_eighthours_line.tex
\begin{tikzpicture}[scale=0.50]

\definecolor{color0}{rgb}{1,0.647058823529412,0}
\definecolor{color1}{rgb}{0.75,0,0.75}
\definecolor{color2}{rgb}{0,0.75,0.75}

\begin{axis}[
legend cell align={left},
legend style={fill opacity=0.8, draw opacity=1, text opacity=1, draw=white!80!black},
log basis x={10},
log basis y={10},
tick align=outside,
tick pos=left,
title={objective function per iteration},
x grid style={white!69.0196078431373!black},
xlabel={Iteration},
xmin=0.794727498311053, xmax=124.571001016567,
xmode=log,
xtick style={color=black},
y grid style={white!69.0196078431373!black},
ymin=1.79523020152852e-06, ymax=0.691199213919795,
ymode=log,
ytick style={color=black}
]
\addplot [semithick, blue]
table {%
0 0.0714384571525096
1 0.0555612107550575
2 0.0459446569100494
3 0.0341885072220812
4 0.0257939989781748
5 0.0194495286151402
6 0.014625185047048
7 0.0106812857007943
8 0.00813898866558732
9 0.00602948285116236
10 0.00456233707715415
11 0.00386048041730888
12 0.00316639776726702
13 0.00283964116709898
14 0.00229681399873428
15 0.00193109116560567
16 0.0014694349804711
17 0.00134624952275776
18 0.00134624952275776
19 0.000931559963810777
20 0.000834337752377123
21 0.000834337752377123
22 0.000754204411447178
23 0.000754204411447178
24 0.000482843612681259
25 0.000482843612681259
26 0.000427605116774803
27 0.000427605116774803
28 0.000427605116774803
29 0.000406352101196229
30 0.000369447601098381
31 0.000367303102740718
32 0.000367303102740718
33 0.000312315385864676
34 0.000312315385864676
35 0.000312315385864676
36 0.000200802307501128
37 0.000200802307501128
38 0.000200802307501128
39 0.000200802307501128
40 0.000200802307501128
41 0.000200802307501128
42 0.000200802307501128
43 0.000195253677168873
44 0.000195253677168873
45 0.000185708986150387
46 0.000185708986150387
47 0.000185708986150387
48 0.000117016894740429
49 0.000117016894740429
50 0.000117016894740429
51 0.000117016894740429
52 0.000117016894740429
53 0.000117016894740429
54 0.000105963438551656
55 0.000105963438551656
56 0.000103225260219154
57 7.38401786047135e-05
58 7.38401786047135e-05
59 7.38401786047135e-05
60 7.38401786047135e-05
61 7.38401786047135e-05
62 7.38401786047135e-05
63 7.38401786047135e-05
64 7.38401786047135e-05
65 6.08119744311909e-05
66 6.08119744311909e-05
67 6.08119744311909e-05
68 6.08119744311909e-05
69 6.08119744311909e-05
70 6.08119744311909e-05
71 6.08119744311909e-05
72 6.08119744311909e-05
73 6.08119744311909e-05
74 6.08119744311909e-05
75 6.08119744311909e-05
76 6.08119744311909e-05
77 6.08119744311909e-05
78 6.08119744311909e-05
79 6.08119744311909e-05
80 6.08119744311909e-05
81 6.08119744311909e-05
82 6.08119744311909e-05
83 6.08119744311909e-05
84 6.08119744311909e-05
85 6.08119744311909e-05
86 6.08119744311909e-05
87 6.08119744311909e-05
88 6.08119744311909e-05
89 6.08119744311909e-05
90 6.08119744311909e-05
91 6.08119744311909e-05
92 6.08119744311909e-05
93 6.08119744311909e-05
94 6.08119744311909e-05
95 6.08119744311909e-05
96 6.08119744311909e-05
97 6.08119744311909e-05
98 6.08119744311909e-05
};
\addlegendentry{MD-MFC}
\addplot [semithick, color0]
table {%
0 0.0719928539314193
1 0.0291021875280729
2 0.0608227472727651
3 0.0340925154123662
4 0.0182706377954331
5 0.0108208184173151
6 0.00654362819734942
7 0.004706868697325
8 0.00338699258949203
9 0.00273017760589023
10 0.00217587096057032
11 0.00170519514142266
12 0.00143192895964922
13 0.0012445059710758
14 0.00106358593114713
15 0.000972314951479059
16 0.000795689517732337
17 0.000680545593708717
18 0.000566336911603868
19 0.000470405551284437
20 0.000408516269442475
21 0.000369720508135427
22 0.000309621041940937
23 0.000305188539460707
24 0.000299555425886248
25 0.000268795927415664
26 0.000247252765503861
27 0.00022519669565524
28 0.000210771736468782
29 0.000196810869080758
30 0.00016963680431487
31 0.000163539516584208
32 0.000155551692122065
33 0.000131056095696749
34 0.000135134823971138
35 0.000130912980939175
36 0.000114439026499608
37 9.94247945717777e-05
38 9.39052051257611e-05
39 0.000101853482366459
40 9.99589689761971e-05
41 9.5394032878753e-05
42 8.86235254938381e-05
43 7.93575143457156e-05
44 7.766645102354e-05
45 7.07636823070897e-05
46 6.68028849263833e-05
47 6.97015099274576e-05
48 6.93765126236353e-05
49 5.90907197903239e-05
50 5.25797427796152e-05
51 5.14456310664474e-05
52 4.82476413067908e-05
53 4.84628765398381e-05
54 4.27640924591798e-05
55 5.01161250903286e-05
56 5.33348131670022e-05
57 5.69175694948525e-05
58 5.37258258845264e-05
59 5.44081869759637e-05
60 5.0289633074272e-05
61 4.49904907346844e-05
62 4.17299692410136e-05
63 4.23634150255495e-05
64 3.60849757181007e-05
65 3.68381247510472e-05
66 3.45563741342153e-05
67 3.15241358133079e-05
68 3.19466883722981e-05
69 2.91054094126966e-05
70 2.75750687719663e-05
71 2.96169804246956e-05
72 2.77365842283729e-05
73 3.02569803677529e-05
74 2.68209878798644e-05
75 3.09648032841798e-05
76 3.1628205129001e-05
77 2.93917029009586e-05
78 3.16082174237719e-05
79 2.74128014934944e-05
80 2.58350387781058e-05
81 2.61034621996723e-05
82 2.48982029637185e-05
83 2.35655606967113e-05
84 2.38416715124636e-05
85 2.13463745998362e-05
86 1.97797617709545e-05
87 1.93034700559086e-05
88 1.82989964290879e-05
89 1.79140374496341e-05
90 1.77201213066945e-05
91 1.61366652302219e-05
92 1.68318340057318e-05
93 1.78380589395838e-05
94 1.77794772744607e-05
95 1.72549224095733e-05
96 1.85443538088699e-05
97 1.74620864474644e-05
98 1.74113752115703e-05
};
\addlegendentry{FP-MFG}
\addplot [semithick, green!50.1960784313725!black]
table {%
0 0.0680579884734846
1 0.0579669348111695
2 0.0453450306593374
3 0.0360746416219856
4 0.0271177118992263
5 0.0199266337995056
6 0.0150357500950594
7 0.0111563262059957
8 0.00837051300910535
9 0.00648535575624235
10 0.0048889284263034
11 0.0037260076696217
12 0.00319251682628848
13 0.00261663094880691
14 0.00198188803907058
15 0.00195756882560105
16 0.0016432049196507
17 0.00135517716823126
18 0.00121089173170311
19 0.00111733435570493
20 0.00101195731683097
21 0.000757216888132596
22 0.000666104347627961
23 0.000568946885481458
24 0.000568946885481458
25 0.000456370615818753
26 0.000456370615818753
27 0.000456370615818753
28 0.000456370615818753
29 0.000456370615818753
30 0.000344207265464615
31 0.000344207265464615
32 0.000339640747947404
33 0.000339640747947404
34 0.000339640747947404
35 0.000339640747947404
36 0.000263948814317541
37 0.000210786107115781
38 0.000210786107115781
39 0.000210786107115781
40 0.000210786107115781
41 0.000210786107115781
42 0.000210786107115781
43 0.000210786107115781
44 0.000210786107115781
45 0.000138581423782154
46 0.000132947063218038
47 0.000132947063218038
48 0.000104678036886661
49 0.000104678036886661
50 0.000104678036886661
51 0.000104678036886661
52 0.000104678036886661
53 0.000104678036886661
54 9.22794914517563e-05
55 9.22794914517563e-05
56 8.38328974158129e-05
57 8.38328974158129e-05
58 8.38328974158129e-05
59 8.38328974158129e-05
60 8.38328974158129e-05
61 8.38328974158129e-05
62 8.38328974158129e-05
63 8.38328974158129e-05
64 8.38328974158129e-05
65 8.38328974158129e-05
66 8.38328974158129e-05
67 6.74098241919881e-05
68 6.74098241919881e-05
69 6.74098241919881e-05
70 6.74098241919881e-05
71 6.74098241919881e-05
72 6.74098241919881e-05
73 6.74098241919881e-05
74 6.74098241919881e-05
75 6.74098241919881e-05
76 6.74098241919881e-05
77 6.74098241919881e-05
78 6.74098241919881e-05
79 6.74098241919881e-05
80 6.74098241919881e-05
81 6.74098241919881e-05
82 6.74098241919881e-05
83 6.74098241919881e-05
84 6.74098241919881e-05
85 6.74098241919881e-05
86 6.74098241919881e-05
87 6.74098241919881e-05
88 6.74098241919881e-05
89 6.74098241919881e-05
90 6.74098241919881e-05
91 6.74098241919881e-05
92 6.74098241919881e-05
93 6.74098241919881e-05
94 6.74098241919881e-05
95 6.74098241919881e-05
96 6.74098241919881e-05
97 6.74098241919881e-05
98 6.74098241919881e-05
};
\addlegendentry{OMD-MFG}
\addplot [semithick, color1]
table {%
1 0.385227817819922
2 0.0963069544549805
3 0.0428030908688802
4 0.0240767386137451
5 0.0154091127127969
6 0.0107007727172201
7 0.00786179220040658
8 0.00601918465343628
9 0.00475589898543113
10 0.00385227817819922
11 0.00318370097371836
12 0.00267519317930501
13 0.00227945454331315
14 0.00196544805010164
15 0.00171212363475521
16 0.00150479616335907
17 0.00133296822775059
18 0.00118897474635778
19 0.00106711306875325
20 0.000963069544549805
21 0.000873532466711841
22 0.000795925243429591
23 0.000728218937277735
24 0.000668798294826253
25 0.000616364508511876
26 0.000569863635828287
27 0.00052843322060346
28 0.000491362012525411
29 0.000458059236408944
30 0.000428030908688802
31 0.000400861412923956
32 0.000376199040839768
33 0.000353744552635374
34 0.000333242056937649
35 0.000314471688016263
36 0.000297243686589446
37 0.000281393584967072
38 0.000266778267188312
39 0.000253272727034795
40 0.000240767386137451
41 0.000229165864259323
42 0.00021838311667796
43 0.000208343871184382
44 0.000198981310857398
45 0.000190235959417246
46 0.000182054734319434
47 0.000174390139348086
48 0.000167199573706563
49 0.000160444738783808
50 0.000154091127127969
51 0.000148107580861177
52 0.000142465908957072
53 0.000137140554581674
54 0.000132108305150865
55 0.000127348038948735
56 0.000122840503131353
57 0.000118568118750361
58 0.000114514809102236
59 0.000110665848267717
60 0.000107007727172201
61 0.00010352803488845
62 0.000100215353230989
63 9.70591629679825e-05
64 9.40497602099421e-05
65 9.11781817325261e-05
66 8.84361381588436e-05
67 8.58159540699315e-05
68 8.33105142344122e-05
69 8.09132152530817e-05
70 7.86179220040657e-05
71 7.64189283515022e-05
72 7.43109216473616e-05
73 7.2288950613609e-05
74 7.03483962417681e-05
75 6.84849453902085e-05
76 6.6694566797078e-05
77 6.49734892595585e-05
78 6.33181817586986e-05
79 6.17253353340686e-05
80 6.01918465343629e-05
81 5.87148022892733e-05
82 5.72914660648308e-05
83 5.591926517926e-05
84 5.45957791694902e-05
85 5.33187291100238e-05
86 5.20859677960955e-05
87 5.0895470712105e-05
88 4.97453277143494e-05
89 4.86337353642119e-05
90 4.75589898543114e-05
91 4.65194804757786e-05
92 4.55136835798585e-05
93 4.45401569915507e-05
94 4.35975348370216e-05
95 4.26845227501299e-05
96 4.17998934266409e-05
97 4.09424824976004e-05
98 4.01111846959519e-05
99 3.93049502928194e-05
};
\addlegendentry{$K^{-2}$}
\addplot [semithick, color2]
table {%
1 0.314118653647896
2 0.055528857522904
3 0.0201507210267884
4 0.00981620792649676
5 0.00561912530045927
6 0.00356217787096011
7 0.00242297329364072
8 0.00173527679759075
9 0.0012926693565757
10 0.000993330401072911
11 0.000782730062961438
12 0.000629710032087138
13 0.000515507925492582
14 0.000428325211641814
15 0.000360467056092266
16 0.000306756497703024
17 0.000263615792380844
18 0.000228513816966683
19 0.000199622607892147
20 0.000175597665639352
21 0.000155433809258082
22 0.000138368433839652
23 0.000123815249545914
24 0.000111318058467503
25 0.000100517969167327
26 9.11297874678037e-05
27 8.29247778880182e-05
28 7.57179154262724e-05
29 6.93583467501384e-05
30 6.37221749392982e-05
31 5.87069465771846e-05
32 5.42273999247109e-05
33 5.02121569503999e-05
34 4.66011286050899e-05
35 4.33434639379776e-05
36 4.03959173929908e-05
37 3.77215402143437e-05
38 3.52886249296701e-05
39 3.30698488391698e-05
40 3.10415750335285e-05
41 2.91832788932636e-05
42 2.74770751380116e-05
43 2.59073259037895e-05
44 2.4460314467545e-05
45 2.31239724298736e-05
46 2.1887650642055e-05
47 2.0741926094921e-05
48 1.96784385027231e-05
49 1.86897515111499e-05
50 1.77692344073293e-05
51 1.69109609659686e-05
52 1.61096226716432e-05
53 1.53604540448124e-05
54 1.46591681932515e-05
55 1.40019010305157e-05
56 1.33851628638067e-05
57 1.28057962669851e-05
58 1.22609393297277e-05
59 1.17479935182942e-05
60 1.12645955028833e-05
61 1.08085924057508e-05
62 1.03780200068709e-05
63 9.97108351292476e-06
64 9.5861405532195e-06
65 9.2216861147305e-06
66 8.87633916940776e-06
67 8.54883078149125e-06
68 8.23799351190138e-06
69 7.94275196167932e-06
70 7.66211431766464e-06
71 7.39516478170743e-06
72 7.14105678020884e-06
73 6.89900686407465e-06
74 6.66828922059086e-06
75 6.4482307285723e-06
76 6.23820649662959e-06
77 6.03763583174896e-06
78 5.84597859174776e-06
79 5.66273188069893e-06
80 5.48742705122976e-06
81 5.31962698179302e-06
82 5.15892360067123e-06
83 5.00493563167647e-06
84 4.85730653931508e-06
85 4.71570265364831e-06
86 4.57981145724487e-06
87 4.44934001852665e-06
88 4.32401355748911e-06
89 4.20357413126148e-06
90 4.08777942828359e-06
91 3.97640166103845e-06
92 3.8692265483098e-06
93 3.76605237884872e-06
94 3.66668914914721e-06
95 3.57095776874101e-06
96 3.47868932710948e-06
97 3.38972441681717e-06
98 3.30391250805641e-06
99 3.22111137021168e-06
};
\addlegendentry{$K^{-5/2}$}
\end{axis}

\end{tikzpicture}